\documentclass[12pt]{amsart}
\usepackage[top=1in,bottom=1in, left=1 in, right= 1 in, includehead,includefoot]{geometry}
\usepackage[T1]{fontenc}

\usepackage{color,hyperref}

\usepackage{amsthm}
\usepackage{amsmath,amssymb}

\usepackage{mathtools} 
\usepackage[numbers,sort&compress]{natbib}
\usepackage{enumerate}

\newtheorem{theorem}{Theorem}[section]

\newtheorem{lemma}[theorem]{Lemma}
\newtheorem{proposition}[theorem]{Proposition}

\theoremstyle{definition}

\theoremstyle{remark}
\newtheorem{remark}[theorem]{Remark}
\theoremstyle{remark}

\numberwithin{equation}{section}

\newcommand{\ind}{{\bf 1}}

\def\inddd#1{{\ind}_{\left\{#1\right\}}} 
\newcommand{\proba}{\mathbb P}
\newcommand{\esp}{{\mathbb E}}

\newcommand{\inv}{^{-1}}
\newcommand{\cov}{{\rm{Cov}}}
\newcommand{\var}{{\rm{Var}}}

\newcommand{\eqnh}{\begin{eqnarray*}}
	\newcommand{\eqne}{\end{eqnarray*}}
\newcommand{\eqnhn}{\begin{eqnarray}}
	\newcommand{\eqnen}{\end{eqnarray}}
\newcommand{\equh}{\begin{equation}}
	\newcommand{\eque}{\end{equation}}

\def\summ#1#2#3{\sum_{#1 = #2}^{#3}}
\def\prodd#1#2#3{\prod_{#1 = #2}^{#3}}
\def\sif#1#2{\sum_{#1=#2}^\infty}

\newcommand{\eqd}{\stackrel{d}{=}}

\def\topp#1{^{(#1)}}

\def\nn#1{{\left\|#1\right\|}}

\def\abs#1{\left|#1\right|}
\def\sabs#1{|#1|}

\def\ccbb#1{\left\{#1\right\}} 

\def\pp#1{\left(#1\right)}
\def\spp#1{(#1)}
\def\bb#1{\left[#1\right]}

\def\mmid{\;\middle\vert\;}

\def\floor#1{\left\lfloor #1 \right\rfloor}
\def\sfloor#1{\lfloor #1 \rfloor}
\def\ceil#1{\left\lceil #1 \right\rceil}

\def\aa#1{\left\langle #1\right\rangle}

\def\qmand{\quad\mbox{ and }\quad}

\def\qmwith{\quad\mbox{ with }\quad}
\def\mfa{\mbox{ for all }}

\def\mmas{\mbox{ as }}

\def\wt#1{\widetilde{#1}}
\def\wb#1{\overline{#1}}
\def\what#1{\widehat{#1}}

\def\limn{\lim_{n\to\infty}}

\def\limsupn{\limsup_{n\to\infty}}
\def\liminfn{\liminf_{n\to\infty}}
\def\weakto{\Rightarrow}

\def\R{{\mathbb R}}

\def\N{{\mathbb N}} 
\def\B{{\mathbb B}}

\newcommand{\calC}{{\mathcal C}}
\newcommand{\calF}{{\mathcal F}}
\newcommand{\calG}{{\mathcal G}}

\newcommand{\calP}{{\mathcal P}}

\newcommand{\calN}{\mathcal{N}}
\newcommand{\calM}{\mathcal{M}}
\newcommand{\calZ}{\mathcal{Z}}

\renewcommand{\d}{{\rm d}}
\newcommand{\aswto}{\stackrel{a.s.w.}\to}

\newcommand{\ddelta}[1]{\delta_{\pp{#1}}}

\date{}
\def\Var{\mathop{\rm Var}}
\def\Cov{\mathop{\rm Cov}}

\def\R{\mathbb{R}}

\def\N{\mathbb{N}}
\def\P{\mathbb{P}}
\def\E{\mathbb{E}}
\def\1{\ind}

\newcommand{\PPP}{{\rm PPP}}
\title[A phase transition in random partitions]{Second-order fluctuations for a phase transition in random partitions}

\author{Jaime Garza}
\address{Department of Mathematics and Statistics\\University of Ottawa\\50 Louis Pasteur Private, Ottawa\\Ontario, K1N 6N5, Canada.}\email{jgarza@uottawa.ca}

\author{Yizao Wang}
\address{Department of Mathematical Sciences\\University of Cincinnati\\2815 Commons Way\\Cincinnati, OH, 45221-0025, USA.}\email{yizao.wang@uc.edu}

\begin{document}\sloppy
\begin{abstract}

In a recent paper, \citet{banderier24phase} investigated the limiting behavior of component counts of random partitions induced by the Chinese restaurant process with parameters $\alpha\in(0,1)$ and $\theta>-\alpha$. Let $C_j(n)$ denote the number of components of size $j$ of a partition of $\{1,\ldots,n\}$ and consider $j=j_n\to\infty$ as $n\to\infty$.
They identified a phase transition in the first-order limit behavior of $C_{j_n}(n)$, where the critical regime corresponds to $j_n\sim rn^{\alpha/(1+\alpha)}$ for some $r>0$. A natural next question is to understand the corresponding second-order fluctuations.

We establish second-order limit theorems in the critical regime and, under an additional rate condition in the subcritical regime ($j_n\ll n^{\alpha/(1+\alpha)}/(\log\log n)^{1/(1+\alpha)}$), for the counting process $(C_{j_n}(n(1+t/j_n)_+))_{t\in\mathbb R}$. In the subcritical regime, after appropriate normalization, the limit is a stationary Ornstein--Uhlenbeck Gaussian process, whereas in the critical regime the limit is a stationary $M/M/\infty$ queue. We also establish a more refined point-process convergence in the critical regime. We first establish these results for the more general Karlin infinite urn model and then adapt the analysis to the Chinese restaurant process. For the latter model, most of our limit theorems are established in the quenched sense.
\end{abstract}
\maketitle
%\tableofcontents

\section{Introduction and main results}

Exchangeable random partitions induced by the Chinese restaurant process play a fundamental role in combinatorial stochastic processes \citep{arratia03logarithmic,pitman06combinatorial}. This is a sequence of random partitions $(\Pi_n)_{n\in\N}$, each of $[n] = \{1,\dots,n\}$,  indexed by $(\alpha,\theta)$ with $\alpha\in[0,1)$ and $\theta>-\alpha$ (see Section \ref{sec:CRP} for details). We shall refer to each $\Pi_n$ as an $(\alpha,\theta)$-partition  in the sequel.  In the case $\alpha=0$ and $\theta>0$, $\Pi_n$ is known to be distributed according to the Ewens sampling formula \citep{ewens72sampling,crane16ubiquitous}, and in the case $(\alpha,\theta) = (0,1)$ it is the partition induced by a uniform random permutation. Exchangeable random partitions are intrinsically related to exchangeable random variables in probability theory. They have found applications in population genetics \citep{ewens72sampling,kingman82coalescent} and nonparametric Bayesian inference \citep{ferguson73bayesian,feng10poisson,ghosal17fundamentals}. More recently, the induced random permutations have also attracted attention from the random matrix community \citep{benarous15fluctuations,wieand00eigenvalue,bahier22smooth,garza24limit}, and these permutations have also served as the building blocks of regular random graphs \citep{johnson14cycles,dumitriu13functional,ganguly20random}.

A central object of study is the number of components of a given size $j$ of the partition  (sometimes referred to as~the $j$-cycle counts for the corresponding random {\em permutations}). Throughout, given a random partition of $[n]$, we let $C_j(n)$ denote the number of all components of size $j$. The law of large numbers and the central limit theorem have both been established for $C_j(n)$ for $(\alpha,\theta)$-partitions. For $(0,\theta)$-partitions and $(\alpha,\theta)$-partitions with $\alpha\in(0,1)$, the asymptotic behaviors are drastically different, as are the proof methods. Most early developments concerned the case $\alpha=0$, where the so-called Feller coupling has proven to be very powerful \citep{arratia03logarithmic}. 

In this paper we focus on the case 
\[
\alpha\in(0,1).
\]
Let 
\[%\equh\label{eq:p_j}
p\topp\alpha_j:=\frac \alpha{\Gamma(1-\alpha)}\frac{\Gamma(j-\alpha)}{\Gamma(j+1)}, \quad j\in\N,
\]%\eque
denote the probability mass function of the $\alpha$-Sibuya distribution. It is well known that
\equh\label{eq:P_j1}
\limn \frac{C_j(n)}{n^\alpha} = p_j\topp\alpha S_\alpha, \quad \mfa j\in\N,
\eque
almost surely, where $S_\alpha$ is the so-called $\alpha$-diversity of the Chinese restaurant process. 

Recently,
 \citet{banderier24phase} initiated the study of the same statistic but in the regime $j_n\to\infty$ (i.e., $C_{j_n}(n)$ as $n\to\infty$), and revealed the following phase transition at the level of the first-order law of large numbers. We write $a_n\ll b_n$ if $a_n = o(b_n)$ (as $n\to\infty$) and $a_n\sim b_n$ if $\limn a_n/b_n = 1$. We let $\weakto$ denote convergence in distribution.
\begin{enumerate}[(i)]
\item (subcritical regime) If $j_n\to\infty, j_n\ll n^{\alpha/(1+\alpha)}$, 
then
\equh\label{eq:BKW1}
 \frac{C_{j_n}(n)}{n^\alpha p_{j_n}\topp\alpha} \weakto S_\alpha,
\eque
as $n\to\infty$. The convergence of all moments also holds.
\item (Critical regime) If $j_n\sim r n^{\alpha/(1+\alpha)}$ for some $r\in(0,\infty)$, then
\equh\label{eq:BKW2}
C_{j_n}(n)\weakto{\rm Poi}(c_\alpha(r)S_\alpha) \qmwith c_\alpha(r) := \frac\alpha{\Gamma(1-\alpha)}r^{-\alpha-1},
\eque
as $n\to\infty$, where the limit is understood as a mixed Poisson random variable with random parameter $c_\alpha(r)S_\alpha$. The convergence of all moments also holds.
\item (Supercritical regime) If $j_n\gg n^{\alpha/(1+\alpha)}$, then $C_{j_n}(n)\to 0$ in probability as $n\to\infty$. 
\end{enumerate}

The above first-order limit theorems extend the earlier result \eqref{eq:P_j1}, where
$j$ is fixed. Once this first-order phase transition has been identified, the
next natural question is to understand the corresponding second-order fluctuations.
For fixed $j$, related second-order limit theorems are available. For comparison, Theorem \ref{thm:j fixed} gives a functional
decomposition of the fluctuations into two non-negligible components. The remaining case is therefore
the regime $j=j_n\to\infty$.

The main contribution of this paper is to establish second-order limit theorems
for the subcritical and critical regimes, where the limiting processes turn out to be of different types.  In the
subcritical regime, the limit is a stationary
Ornstein--Uhlenbeck Gaussian process under the time scale $n(1+t/j_n)$. In the
critical regime, the limit is instead a stationary $M/M/\infty$ queue, and a
more detailed point-process convergence characterizing the formation and disappearance of components of size $j_n$ is also obtained.

Before stating our results we first introduce some notation.
Throughout, we consider the sequence of random partitions $(\Pi_n)_{n\in\N}$ induced by a Chinese restaurant process, and hence $(C_j(n))_{n,j}$ are defined on a common probability space. Moreover, letting 
\[
K_n:=\summ j1n C_j(n)= |\Pi_n|
\]
denote the total number of components of $\Pi_n$, it is well known that
\[
\limn \frac{K_n}{n^\alpha} = S_\alpha,
\]
almost surely. Let $\Pi_{n,j}$ denote the $j$-th component of $\Pi_n$ in the order of appearance (the number of customers at the $j$-th table after $n$ customers have entered the restaurant), and $|\Pi_{n,j}|$ its size. It is also well known that
\[
\limn \frac{|\Pi_{n,j}|}n = P_j, \mbox{ almost surely,}
\]
where $(P_j)_{j\in\N}$ is referred to as the asymptotic frequencies of the components. This family of random variables follows the Griffiths--Engen--McCloskey (GEM) 
distribution with parameters $(\alpha,\theta)$, and the random sequence ordered in decreasing values $(P_j^\downarrow)_{j\in\N}$ follows the Poisson--Dirichlet distribution. We set $\calP :=\sigma((P_j)_{j\in\N})$. It is well known that $S_\alpha$ is $\calP$-measurable.

We first state the two counting-process limit theorems for the Chinese restaurant process, which refine \eqref{eq:BKW1} and \eqref{eq:BKW2}, respectively. In particular, our results are {\em quenched} limit theorems in the sense of almost sure weak convergence \citep{grubel16functional}. We say a sequence of random elements $(X_n)_{n\in\N}$ converges almost surely weakly to $X$ with respect to a $\sigma$-algebra $\calF$, denoted by `$X_n\aswto X$ with respect to $\calF$' as $n\to\infty$, if for all continuous and bounded functions $f$ we have $\limn \esp(f(X_n)\mid\calF) = \esp (f(X)\mid\calF)$ almost surely. Implicitly, $(X_n)_{n\in\N}, X, \calF$ are on the same probability space. Write
\[
s_n(t):=n\pp{1+\frac t{j_n}}_+.
\]
This is the correct time scaling under which a non-trivial dependence structure emerges in the limit of $C_{j_n}(\floor{s_n(t)})$ (in contrast to the linear scaling $nt$ for limit theorems for $C_j(\floor{nt})$).

\begin{theorem}[subcritical regime]\label{thm:CRP}
 If $j_n/(\log n)^2\to\infty$ and 
 \equh\label{eq:j_n upper}
 j_n = o\pp{\frac{n^{\alpha/(1+\alpha)}}{(\log\log n)^{1/(1+\alpha)}}},
 \eque
  then
\equh\label{eq:CRP limit}
\frac{j_n^{(\alpha+1)/2}}{n^{\alpha/2}}\pp{C_{j_n}(\floor{s_n(t)})-s_n(t)^\alpha p_{j_n}\topp\alpha S_\alpha}_{t\in\R}\aswto \pp{\frac{\alpha}{\Gamma(1-\alpha)}S_\alpha}^{1/2}\pp{\zeta_t}_{t\in\R}
\eque
in $D(\R)$ equipped with the local $J_1$ topology with respect to $\calP$, where $(\zeta_t)_{t\in\R}$ is a stationary Ornstein--Uhlenbeck Gaussian process, independent of $\calP$. Namely, $(\zeta_t)_{t\in\R}$  is a centered Gaussian process with covariance function
\[
\cov(\zeta(s),\zeta(t)) = e^{-|s-t|}, s,t\in\R.
\]
If $j_n\to\infty$ and \eqref{eq:j_n upper} hold but $j_n/(\log n)^2\not\to\infty$, then the convergence in \eqref{eq:CRP limit} still holds in the sense of weak convergence. 
\end{theorem}
\begin{theorem}[Critical regime]\label{thm:crit} If $j_n = \floor{rn^{\alpha/(1+\alpha)}}$ for $r>0$, then
\[
\pp{C_{j_n}(\floor{s_n(t)})}_{t\in\R}\aswto (\calC_{\alpha,r}(t))_{t\in\R}
\]
in $D(\R)$ equipped with the local $J_1$ topology, with respect to $\calP$, where, conditionally on $\calP$, $(\calC_{\alpha,r}(t))_{t\in\R}$ is a stationary immigration--death process with immigration rate $c_\alpha(r)S_\alpha$ and unit per-particle death rate.
\end{theorem}
The limiting process $(\calC_{\alpha,r}(t))_{t\in\R}$ is also known as an $M/M/\infty$ queue. Its off-diagonal transition rates are
\[
q_{i,i+1}=c_\alpha(r)S_\alpha, i\geq0,\qmand
q_{i,i-1}=i, i\geq1,
\]
and all other off-diagonal transition rates are zero.
Its stationary distribution (i.e., the marginal distribution of $\calC_{\alpha,r}(t)$) is a mixed Poisson distribution with random parameter $c_\alpha(r)S_\alpha$. 
In fact, we shall establish a more detailed point-process convergence in Theorem \ref{thm:PP conv}, describing the asymptotic behavior of times of formation and disappearance of each component of size $j_n$; the above convergence is then established following essentially a continuous-mapping argument (see Theorem \ref{thm:discrete-cmt}).

Moreover, we shall establish the aforementioned limit theorems for a large family of random partitions. 
We exploit the well-known connection between Chinese restaurant processes and infinite urn schemes through Kingman’s representation theorem. As a first step, we establish limit theorems in both regimes for the urn counts of the infinite urn model, or equivalently, for the component counts of the induced paintbox partitions.
Here, the corresponding urn model is with sampling frequencies decaying at a polynomial rate 
\equh\label{eq:P_j down}
P_j^\downarrow\sim D^{1/\alpha}j^{-1/\alpha} \mbox{ as $j\to\infty$ almost surely}
\eque
 with $D = S_\alpha/\Gamma(1-\alpha)$. This infinite urn scheme is often referred to as the Karlin model, following the seminal work of \citet{karlin67central} on urn models with polynomially decaying sampling frequencies. Since then, refined analyses of $C_j(n)$ have been obtained for fixed $j$ \citep{chebunin16functional,barbour09small,gnedin07notes,garza25functional,durieu16infinite}.

There is a delicate issue in applying the methodology developed for $C_j(n)$ to $C_{j_n}(n)$. For the Karlin model, assumption \eqref{eq:P_j down} is enough to establish most limit theorems concerning $C_j(n)$ with $j$ fixed. Now that $j = j_n\to\infty$ as $n\to\infty$, some additional care is needed, as the fluctuations of $P_j^\downarrow$ around $D^{1/\alpha}j^{-1/\alpha}$ may have a non-negligible impact on the limit.  More specifically, for sampling frequencies $(p_j)_{j\in\N}$ of the Karlin model if $p_j$ {\em is exactly} $D^{1/\alpha}j^{-1/\alpha}$, or sufficiently close to it in an appropriate sense, one should expect the same limit behavior. In this case we say the fluctuations of $p_j$ do not affect the limit. 

It turns out that the fluctuations of $P_j^\downarrow$ (as asymptotic frequencies of $(\alpha,\theta)$-partitions) are not negligible. In the subcritical regime, this essentially leads to the constraint $j_n\ll n^{\alpha/(1+\alpha)}/(\log\log n)^{1/(1+\alpha)}$ (compared to $j_n\ll n^{\alpha/(1+\alpha)}$ in \eqref{eq:BKW1}), and it is not clear to us whether this additional constraint can be removed. In this regime, Theorem \ref{thm:CRP} follows as a corollary of Theorem \ref{thm:1} for the Karlin model. In the critical regime, we shall show that the fluctuations of $P_j^\downarrow$ around $D^{1/\alpha}j^{-1/\alpha}$ are strong enough that at the level of point-process convergence, a different centering is necessary. That is, in the critical regime, Theorem \ref{thm:crit} is not a simple corollary of a general result on the Karlin model, but has to be dealt with as a special case taking into account the fluctuation of $P_j^\downarrow$; compare Theorems \ref{thm:PP conv} and \ref{thm:PP conv general}.

Another delicate point in Theorem \ref{thm:CRP} is that when $j_n\to\infty$ but $j_n/(\log n)^2\not\to\infty$, we only establish that \eqref{eq:CRP limit} holds in the annealed sense. This is better explained in the next result. 
\begin{proposition}\label{prop:0 process}
Suppose that
\[
j_n\to\infty
\qquad\mbox{and}\qquad
j_n=o\pp{
\frac{n^{\alpha/(1+\alpha)}}
{(\log\log n)^{1/(1+\alpha)}}
}.
\]
Then, for every $T>0$,
\begin{equation}\label{eq:centering probability}
\limn\frac{j_n^{(\alpha+1)/2}}{n^{\alpha/2}}\sup_{t\in[-T,T]}
\abs{
{
\esp\pp{C_{j_n}(\floor{s_n(t)})\mid\calP}
-s_n(t)^\alpha p_{j_n}\topp\alpha S_\alpha
}}
=0
\qquad\mbox{in probability}.
\end{equation}
If, in addition,
\begin{equation}\label{eq:lower growth combined}
\frac{j_n}{(\log n)^2}\to\infty,
\end{equation}
then the convergence in \eqref{eq:centering probability} holds almost
surely.
\end{proposition}

In fact, applying Theorem \ref{thm:1} to the Karlin model, we do not obtain directly Theorem \ref{thm:CRP}, but only
\[
\pp{\frac{j_n^{(\alpha+1)/2}}{n^{\alpha/2}}\pp{C_{j_n}(\floor{s_n(t)})-\esp\pp{C_{j_n}(\floor{s_n(t)})\mmid\calP}}}_{t\in\R}
 \aswto \pp{\frac\alpha{\Gamma(1-\alpha)}S_\alpha}^{1/2}\pp{\zeta(t)}_{t\in\R}
\]
with respect to $\calP$ as $n\to\infty$ (the difference is in the centering term). So, together with Proposition \ref{prop:0 process}, the above yields Theorem \ref{thm:CRP}.

There is also a delicate difference between fixed $j$ and increasing $j_n$ in the subcritical regime. 
With $j_n = j$ fixed, one has  the following functional central limit theorem concerning $C_j(n)$.
\begin{theorem}\label{thm:j fixed}
Under the notations above,
\begin{multline}\label{eq:j decomp}
\pp{\frac{C_{j}(\floor{nt})-\esp\pp{C_j(\floor{nt})\mmid\calP}}{n^{\alpha/2}}, \frac{\esp\pp{C_j(\floor{nt})\mmid\calP} - \floor{nt}^\alpha p_j\topp\alpha S_\alpha}{n^{\alpha/2}}}_{t\in[0,1],j\in\N}\\
\weakto\pp{\pp{\frac {S_\alpha}{\Gamma(1-\alpha)}}^{1/2}\zeta_{\alpha,j}\topp1(t),
\pp{\frac {S_\alpha}{\Gamma(1-\alpha)}}^{1/2}\zeta_{\alpha,j}\topp2(t)}_{t\in[0,1],j\in\N},
\end{multline}
in $(D[0,1]\times D[0,1])^\N$ with
\begin{align*}
\zeta\topp1_{\alpha,j}(t)&:=
\int_{\R_+\times\Omega'}\pp{\inddd{N'(tr) = j} - \proba(N'(tr) = j)}\d M_\alpha(\d r\d\omega'),\\
\zeta\topp2_{\alpha,j}(t)&:=\int_{\R_+\times\Omega'}\proba(N'(tr) = j)\d M_\alpha(\d r\d\omega'),
\end{align*}
where $M_\alpha$ is a Gaussian random measure on $(0,\infty)\times\Omega'$ with control measure $\alpha r^{-\alpha-1}\d r\d\proba'$ and $N'$ on $(\Omega',\proba')$ is a standard Poisson process, and $M_\alpha$ is independent of $S_\alpha$.
\end{theorem}
Note that the second component process here is not negligible when $j_n=j$ is fixed. In contrast, when $j_n\to\infty$, the second component process is negligible compared to the first one. 
The proof of Theorem \ref{thm:j fixed} can be found in Appendix \ref{sec:decomp} by again exploiting the connection to the Karlin model, following a recent development in \citet{wang26central}. The above result is an improvement of a recent development by \citet{bercu24martingale}, who proved the annealed convergence of $(C_j(n)-n^\alpha p_j\topp\alpha S_\alpha)/n^{\alpha/2}$ (which is the sum of the two components at $t=1$ above) by a martingale approach.

We conclude the introduction by emphasizing that our method is completely different from the one by \citet{banderier24phase}. Their approach is analytic and relies on certain related generating functions, and exploits the very nice structure underlying the Chinese restaurant process. It is a very powerful general method as it can be applied to establish limit theorems for several combinatorial structures besides the Chinese restaurant process. However, when applied to the Chinese restaurant process, the method seems limited; it is not clear how their method can be modified to obtain the more refined quenched limit theorems as established here. It would also be interesting to see whether second-order limit theorems can be established for phase transitions in other examples revealed in their paper.

\subsection*{The paper is organized as follows}
Section \ref{sec:prelim} recalls the preliminaries, while Sections \ref{sec:sub} and \ref{sec:crit} establish the limit theorems in the subcritical and critical regimes, respectively.

\subsection*{Acknowledgements}
 Y.W.~was partially supported by the Simons Foundation (MP-TSM-00002359).

%\newpage
\section{Preliminary results}\label{sec:prelim}

  \subsection{Chinese restaurant process}\label{sec:CRP}

Consider the Chinese restaurant process with $(\alpha,\theta)$-seating with $\alpha\in(0,1)$ and $\theta>-\alpha$. 
The process consists of a family of exchangeable random partitions $(\Pi_n)_{n\in\N}$, each of $[n]$, constructed consecutively. 
 The procedure goes as follows. Set $\Pi_1 = \{\{1\}\}$. 
 Suppose a sequence of partitions $\Pi_1,\dots,\Pi_n$ has been sampled, and moreover $\Pi_n = \{\Pi_{n,1},\dots,\Pi_{n,k}\}$ ($(\Pi_{n,j})_{j=1,\dots,k}$ are disjoint non-empty subsets of $[n]$ and $\bigcup_{j=1}^k \Pi_{n,j} = [n]$; in this case $\Pi_n$ is said to have $k$ components). Then, the partition $\Pi_{n+1}$ is obtained by 
\begin{enumerate}[(i)]
\item adding element $n+1$ to an existing block $j$ (i.e., setting $\Pi_{n+1,j} := \Pi_{n,j}\cup \{n+1\}$) with probability $(|\Pi_{n,j}|-\alpha)/(n+\theta)$; 
\item creating a new block with a single element $n+1$ (i.e., setting $\Pi_{n+1,k+1} := \{n+1\})$ with probability $(k\alpha+\theta)/(n+\theta)$;
\end{enumerate}
and all other existing blocks remain unchanged (i.e., setting $\Pi_{n+1,j} = \Pi_{n,j}$, for all $j=1,\dots,k$ that have not been involved). 

Let $(P_j)_{j\in\N}\equiv (P_j^{\alpha,\theta})_{j\in\N}$ denote the asymptotic frequencies of blocks (i.e., blocks of $(\alpha,\theta)$-partitions), and set $\calP = \sigma((P_j)_{j\in\N})$. The law of the decreasingly ordered sequence $(P_j^\downarrow)_{j\in\N}\equiv (P_j^{\alpha,\theta,\downarrow})_{j\in\N}$ is known as the Poisson--Dirichlet distribution with parameter $(\alpha,\theta)$. These are random elements from $\Delta_\infty = \{(p_j)_{j\in\N}: p_1\ge p_2\ge\cdots\ge 0, \sif j1 p_j = 1\}$.   
It is well known that
\[%\equh\label{eq:diversity}
\limn\frac{|\Pi_n|}{n^\alpha} = S_\alpha, \mbox{ almost surely,}
\]%\eque 
and also 
\equh\label{eq:S_alpha 2}
S_\alpha = \lim_{j\to\infty} j (P_j^{\downarrow})^\alpha\Gamma(1-\alpha), \mbox{ almost surely.}
\eque

 The following can be read from \citep{pitman06combinatorial} and some more details can be found in \citet{wang26central}.
\begin{lemma}\label{lem:P_j}
Let $(P_j^\downarrow)_{j\in\N}$ follow the Poisson--Dirichlet distribution with parameter $(\alpha,\theta)$, and $S_\alpha$ be as in \eqref{eq:S_alpha 2}. Set
\equh\label{eq:Gamma_j}
\Gamma_j:=\frac{S_{\alpha}}{\Gamma(1-\alpha)}\pp{P_j^\downarrow}^{-\alpha}, j\in\N.
\eque
\begin{enumerate}[(i)]
\item When $\theta=0$, the sequence $(\Gamma_j)_{j\in\N}$ has the law of consecutive arrival times of a standard Poisson process.
\item More generally for all $\theta>-\alpha$, 
the law of $(P_j^{\alpha,\theta,\downarrow})_{j\in\N}$ is absolutely continuous with respect to the law of $(P_j^{\alpha,0,\downarrow})_{j\in\N}$. 
More precisely, for every continuous and bounded function $f:\Delta_\infty\to\R$, we have 
\[%\equh\label{eq:RN}
\esp f\pp{(P_j^{\alpha,\theta,\downarrow})_{j\in\N}} = \esp \pp{\frac{\Gamma(\theta+1)}{\Gamma(\theta/\alpha+1)}S_\alpha^{\theta/\alpha}f\pp{(P_j^{\alpha,0,\downarrow})_{j\in\N}}},
\]%\eque
where $S_\alpha$ on the right-hand side is defined via \eqref{eq:S_alpha 2} using $P_j^\downarrow\equiv P_j^{\alpha,0,\downarrow}$. 
%:
\end{enumerate}
In particular,  by \eqref{eq:S_alpha 2} and \eqref{eq:Gamma_j}, for all $\alpha\in(0,1)$ and $\theta>-\alpha$, as $j\to\infty$
\[%\equh\label{eq:P_j}
P_j^\downarrow =  \pp{\frac{S_{\alpha}}{\Gamma(1-\alpha)}}^{1/\alpha} \Gamma_j^{-1/\alpha}    \sim \pp{\frac{S_{\alpha}}{\Gamma(1-\alpha)}}^{1/\alpha} j^{-1/\alpha}, \mbox{ almost surely.}
\]%\eque
\end{lemma}
\subsection{Karlin model}%\label{sec:Karlin}
We recall the notion of an infinite urn model. Let $(Y_i)_{i\in\N}$ be i.i.d.~random variables  taking values in $\N$ with $\proba(Y_1 = \ell) = p_\ell, \ell\in\N$. Without loss of generality we assume that $p_\ell$ is non-increasing in $\ell$. The event $Y_n = \ell$ is interpreted as  throwing a ball into the $\ell$-th urn in the $n$-th round. The numbers $(p_\ell)_{\ell\in\N}$ are referred to as the sampling frequencies of the model (with $\sif\ell1p_\ell=1$ and $p_\ell\ge 0,\ell\in\N$).
Karlin investigated thoroughly the case when $p_j$ decays at a polynomial rate. For analytical convenience, an equivalent assumption is on the function
\[%\equh\label{eq:Karlin}
\nu(x):=\max\ccbb{j\in\N:p_j\ge \frac 1x}, x>0,
\]%\eque
where $\nu(x)$ is regularly varying at infinity with index $\alpha\in(0,1)$, denoted by $\nu\in RV_\alpha$. That is, $\nu(x) = x^\alpha L(x)$ where $L$ is a slowly varying function at infinity. This is equivalent to $p_j$ decaying as $j^{-1/\alpha}$ as $j\to\infty$ up to a multiplicative slowly varying function.

The urn model induces a random partition of each $n\in\N$: $i$ and $j$ are in the same component of the partition, if and only if $Y_i = Y_j$. This is also known as the paintbox partition. The statistic of interest for us is the number of urns with exactly $j$ balls after $n$ rounds. That is,
\[
C_j(n) := \sif \ell1 \inddd{K_{n,\ell} = j} \qmwith K_{n,\ell}:=\summ i1n \inddd{Y_i = \ell}. 
\]
We use the same notation $C_j(n)$ as for the component counts of $(\alpha,\theta)$-partitions. The reason is  that the component count of size $j$ of an $(\alpha,\theta)$-partition has the same law as the urn count of size $j$ from a Karlin model with {\em random} sampling frequencies $(P_j^\downarrow)_{j\in\N}$. This is because $(\alpha,\theta)$-partitions are exchangeable and the claim then follows from de Finetti's theorem. This result  is essentially Kingman's representation theorem. A standard reference of ours is \citet{pitman06combinatorial}.

When working with the Karlin model corresponding to $(\alpha,\theta)$-partitions, the following lemma will be used several times.
\begin{lemma}\label{lem:N(D)} Suppose $\alpha\in(0,1),\theta = 0$. Then, 
\[
P_j^\downarrow = \pp{\frac{S_\alpha}{\Gamma(1-\alpha)}}^{1/\alpha}\Gamma_j^{-1/\alpha} = \frac{\Gamma_j^{-1/\alpha}}{\sif \ell1 \Gamma_\ell^{-1/\alpha}},
\]
where $(\Gamma_j)_{j\in\N}$ is a sequence of arrival times of a standard Poisson process $(N(t))_{t\ge 0}$ (i.e., $N(t):=\max\{j\in\N:\Gamma_j\le t\}$). Then, 
\[
\nu(t) = N(Dt^\alpha) \qmwith D := \frac{S_\alpha}{\Gamma(1-\alpha)} = \pp{\sif\ell1\Gamma_\ell^{-1/\alpha}}^{-\alpha}.
\]
We have
\equh\label{eq:Vishakha}
\pp{\frac{ N(Dn t) - Dnt}{n^{1/2}}}_{t\in[0,\infty)}\weakto \pp{\B_{Dt}}_{t\in[0,\infty)},
\eque
in $D[0,\infty)$ as $n\to\infty$, where in the limit the Brownian motion $\B$ is independent of $D$. 
\end{lemma}
\begin{proof}
It suffices to notice that by definition,
\[
\nu(t) = \max\ccbb{j\in\N:P_j^\downarrow \ge \frac 1t} =\max\ccbb{j\in\N:\Gamma_j^{1/\alpha} \le D^{1/\alpha}t}.
\]
The convergence \eqref{eq:Vishakha} was established in \citet[Section 3.1]{wang26central}. Note that this convergence is not immediate as on the left-hand side $D$ depends on $N$.
\end{proof}
As a standard approach, we shall also work with the Poissonized model. Given the sampling frequencies of the Karlin model $(p_\ell)_{\ell\in\N}$ (which might be random), let $(\calN_\ell(t))_{t\ge 0}, \ell\in\N$ be a family of conditionally independent Poisson processes with respective rates $p_\ell$. Note that $\calN(t):=\sif \ell1 \calN_\ell(t)$ as a process indexed by $t$ is a standard Poisson process. This process is not to be confused with the standard Poisson process $(N(t))_{t\ge0}$  associated to $(P_j^\downarrow)_{j\in\N}$ as explained in Lemma \ref{lem:N(D)}, and both shall be involved in the proofs.   The processes $(\calN_\ell)_{\ell\in\N}$ also induce a sequence of random partitions in a way similar to the paintbox construction: $i\sim j$ if the $i$-th and $j$-th arrival times of $\calN$ are both from $\calN_\ell$ for some $\ell\in\N$.  Again, the statistic of interest is the number of components of size $j$ for the random partition at time $t$ (of $\{1,\dots,\calN(t)\}$), denoted by $\wt C_j(t)$. That is,
\equh\label{eq:wt C}
\wt C_j(t):=\sif \ell1 \inddd{\calN_\ell(t) = j}, j\in\N.
\eque
%\newpage
\section{The subcritical regime }\label{sec:sub}
In this section, the main result, Theorem \ref{thm:1},  is a functional central limit theorem for  $C_{j_n}(n)$ of the Karlin model in the subcritical regime. Theorem \ref{thm:CRP} is proved in Section \ref{sec:0 process} as a corollary. 

	In the subcritical regime, we consider $j_n\to\infty$ and 
\[%	\begin{equation}\label{eq: j_n subcrit}
		j_n\ll n^{\alpha/(\alpha+1)}.
\]%	\end{equation}
Let $(p_j)_{j\in\N}$ denote the sampling frequencies of the Karlin model and recall that $\nu(x) = \max\{j\in\N:p_j\ge 1/x\}$.  Write
	\begin{equation*}
		\sigma_{n,j_n}^2:=\frac\alpha{j_n}\nu\pp{\frac n{j_n}}.%\sim \frac{\alpha n^\alpha \Gamma(j_n-\alpha)L(n/j_n)}{\Gamma(j_n+1)}.
	\end{equation*}
Recall that Karlin's condition imposes that $\nu\in RV_\alpha$; that is
\[%\equh\label{eq:RV}
\limn \frac{\nu(nx)}{\nu(n)} = x^{\alpha}, \mfa x>0,
\]%\eque
and moreover the convergence is locally uniform in $x$ \citep{resnick87extreme}. 
In addition to the standard regular variation assumption, we need to impose a further assumption on the rate of convergence above. 
To state our condition we introduce
\[
\rho_{[a,b]}(y)\equiv\rho_{[a,b]}\topp\alpha(y) :=\sup_{x\in[a,b]}\abs{\frac{\nu(yx)}{\nu(y)}-x^\alpha}, \mbox{ for } 0\le a<b<\infty, y>0.
\]	
Set \[
s_n(t):=n\pp{1+\frac t{j_n}}_+, t\in\R.
\] 
	\begin{theorem}\label{thm:1}
Assume that $\alpha\in(0,1), \nu\in RV_\alpha$, $j_n\to\infty, j_n = o(n^{\alpha/(1+\alpha)})$,  $\sigma_{n,j_n}\to\infty$, and for some $\epsilon\in(0,1)$,
\equh\label{eq:j_n rate}
\limn j_n^{1/2}\rho_{[1-\epsilon,1+\epsilon]}(n/j_n) = 0. 
\eque
Then 
\[%\begin{equation}\label{eq: SC Main result}
\left(\frac{1}{\sigma_{n,j_n} }\left(C_{j_n}\pp{\floor{s_n(t)}}-\E C_{j_n}\pp{\floor{s_n(t)}}\right)\right)_{t\in \R} \Rightarrow (\zeta(t))_{t\in \R},
\]%\end{equation}
in $D(\R)$ as $n\to\infty$, where $(\zeta(t))_{t\in \R}$ is an Ornstein--Uhlenbeck process as in Theorem \ref{thm:CRP}.
\end{theorem}

\begin{remark}
If the rate at which  $\rho_{[1-\epsilon,1+\epsilon]}(y)\to 0$ is assumed to be faster than a polynomial rate, it then follows that so is $\rho_{[a,b]}(y)$ for all $0<a<b$.  Indeed, assume say 
\[
\rho_{[1-\epsilon,1+\epsilon]}(y)\le \wb \rho_{[1-\epsilon,1+\epsilon]}(y), \mfa y>0,
\]
with $\wb\rho_{[1-\epsilon,1+\epsilon]}(y) = y^{-\gamma}L(y)$ for some $\gamma\ge 0$ and a slowly varying function $L$. It then follows that for all $0<a<b<\infty$, one can find a constant $C_{\epsilon,a,b}$ such that 
\[
\rho_{[a,b]}(y)\le C_{\epsilon, a,b}\wb \rho_{[1-\epsilon,1+\epsilon]}(y), \mfa y>0. 
\]
To see the above, we explain only how to bound $|\nu(yx)/\nu(y)-x^\alpha|$ for $x>1$ using $\rho_{[1-\epsilon,1+\epsilon]}(y)$. Let $m\in\N$ be such that one can set $x_0 := x, x_1 := x/(1+\epsilon),\dots,x_m := x/(1+\epsilon)^m, x_{m+1} := 1$ such that $1<x_m/x_{m+1}\le1+\epsilon$. Then, one can write
\begin{align*}
\abs{\frac{\nu(yx)}{\nu(y)}-x^\alpha} & = \abs{\prodd i0m \frac{\nu(yx_i)}{\nu(yx_{i+1})} - \prodd i0m \frac{x_i^\alpha}{x_{i+1}^\alpha}}\\
%& \le \abs{\frac{\nu(yx_0)}{\nu(yx_{1})} - \pp{\frac{x_0}{x_1}}^\alpha} \prodd i0m \frac{\nu(yx_i)}{\nu(yx_{i+1})}\\
& \le \summ j0m \pp{\frac{x_0}{x_j}}^\alpha \abs{\frac{\nu(yx_j)}{\nu(yx_{j+1})} - \pp{\frac{x_j}{x_{j+1}}}^\alpha}\prodd i{j+1}m\frac{\nu(yx_i)}{\nu(yx_{i+1})},
\end{align*}
with $\prodd i{m+1}m(\cdots)\equiv 1$. It is now clear that the right-hand side above is bounded by $C\summ j0{m}\rho_{[1-\epsilon,1+\epsilon]}(yx_{j+1}) \le C\wb\rho_{[1-\epsilon,1+\epsilon]}(y)$ where the constant does not depend on $y$. 
\end{remark}
\begin{remark}
Note that under the general assumption \eqref{eq:j_n rate}, $j_n\to\infty$ is allowed but the rate cannot be too fast unless $\rho_{[1-\epsilon,1+\epsilon]}(y)\to 0$ fast enough. 
For example, if $\nu(x) = x^\alpha \log x$ for $x$ large enough, then $\rho_{[1-\epsilon,1+\epsilon]}(y) \le C (\log y)\inv$ for $y$ large enough. Then, \eqref{eq:j_n rate} is equivalent to $j_n = o(\log ^2n)$. 
\end{remark}
\begin{remark}The assumption $\sigma_{n,j_n}^2 = (\alpha/j_n)\nu(n/j_n)\to\infty$ is needed if $j_n = n^{\alpha/(1+\alpha)}L(n)$ for some slowly varying function $L$. This assumption is equivalent to $j_n = o(\nu(n/j_n))$, which is always satisfied for $(\alpha,\theta)$-partitions under \eqref{eq:j_n upper} (and hence there is no need to mention this assumption in Theorem \ref{thm:CRP}).
\end{remark}

%\newpage

\subsection{Convergence of the Poissonized model}
Let $\wt C_{j_n}(t) = \sif\ell1 \inddd{\calN_\ell(t) = j_n}$ denote the urn count of the Poissonized Karlin model as in \eqref{eq:wt C}.
In this section we shall prove the following.
\begin{proposition}\label{prop:1}
Under the assumption of Theorem \ref{thm:1}, we have
\[%\begin{equation}\label{eq: SC Main result}
\left(\frac{1}{\sigma_{n,j_n} }\left(\wt C_{j_n}\pp{s_n(t)}-\E \wt C_{j_n}\pp{s_n(t)}\right)\right)_{t\in \R} \Rightarrow (\zeta(t))_{t\in \R},
\]%\end{equation}
in $D(\R)$ as $n\to\infty$, where $(\zeta(t))_{t\in \R}$ is an Ornstein--Uhlenbeck process as in Theorem \ref{thm:CRP}.
\end{proposition}
We first compute the asymptotic behavior of the first moment. Introduce
\[
I_{n,j}(t):=\int_0^\infty \pp{j-z}e^{-z}z^{j-1}\nu\pp{\frac {nt}z}\d z, \quad t>0.
\]
\begin{lemma}%\label{lem:I_n}
Assume $\alpha$, $\nu$, and $(j_n)_{n\in\N}$ satisfy the assumptions in Theorem \ref{thm:1}. Then for every $t\in(1-\epsilon,1+\epsilon)$, 
 we have
\equh\label{eq:I_n asymp}
I_{n,j_n}(t)\sim \alpha \Gamma(j_n)\nu\pp{\frac n{j_n}}t^\alpha.
\eque
In particular, 
\equh\label{eq:E wt C}
\esp \wt C_{j_n}(nt) =\frac{I_{n,j_n}(t)}{\Gamma(j_n+1)} \sim \alpha t^\alpha \frac1{j_n}\nu\pp{\frac n{j_n}} = t^\alpha\sigma_{n,j_n}^2.
\eque
Moreover, we have local uniform convergence in $t$ in the following sense:
\begin{align}%\label{eq:I_n uniform}
\limn\sup_{t\in[1-\epsilon',1+\epsilon']}\abs{\frac{I_{n,j_n}(t)}{\alpha\Gamma(j_n)\nu(n/j_n) t^\alpha} - 1} = 0, \nonumber\\
\limn\sup_{t\in[1-\epsilon',1+\epsilon']}\abs{\frac{\esp\wt C_{j_n}(nt)}{\sigma_{n,j_n}^2 t^\alpha} - 1} = 0, \label{eq:EC uniform}
\end{align}
for all $\epsilon'\in(0,\epsilon)$.
\end{lemma}
\begin{remark}
In the sequel, we shall only need $t$ in a shrinking neighborhood of $1$. If the assumption \eqref{eq:j_n rate} is relaxed to $\limn j_n^{1/2}\rho_{[a,b]}(n/j_n) = 0$ for $0<a<b<\infty$, then the corresponding local uniform convergence holds over $(a,b)$. 
\end{remark}
\begin{proof}
Note that 
\begin{align*}
\esp\wt C_{j_n}(nt)& = \sif\ell1\frac{(ntp_\ell)^{j_n}}{j_n!}e^{-ntp_\ell} = \frac1{\Gamma(j_n+1)}\int_0^\infty (nt/x)^{j_n}e^{-nt/x}\d \nu(x)\\
& = \frac1{\Gamma(j_n+1)}\int_0^\infty e^{-nt/x}\pp{\frac{j_n}x-\frac{nt}{x^2}}\pp{\frac{nt}x}^{j_n}\nu(x)\d x = \frac{I_{n,j_n}(t)}{\Gamma(j_n+1)}. 
\end{align*}
It suffices to prove the statements concerning $I_{n,j_n}(t)$. 

Let $G_x$ be a Gamma random variable with parameter $x>0$ (i.e., with density function proportional to $z^{x-1}e^{-z}, z>0$). Write $W_n :=(j_n-G_{j_n})/\sqrt{j_n}$. So $\esp W_n = 0$ and $W_n\weakto\calN(0,1)$ as $n\to\infty$. Write
\begin{align*}
I_{n,j_n}(t)&= \int_0^\infty e^{-z}(j_n-z)z^{j_n-1}\nu\pp{\frac {nt}z}\d z = \Gamma(j_n)\esp\pp{(j_n-G_{j_n})\nu\pp{\frac {nt}{G_{j_n}}}} \nonumber\\
&= \nu(n/j_n)\Gamma(j_n)j_n^{1/2}\esp \pp{W_n\frac{\nu\spp{(n/j_n)(1-W_n/j_n^{1/2})^{-1}t}}{\nu(n/j_n)}}.
\end{align*}
Thus,
\equh
\frac{I_{n,j_n}(t)}{\Gamma(j_n)\nu\pp{n/{j_n}}} = j_n^{1/2} (J_{n,1}(t)+J_{n,2}(t)),\label{eq:decomp I_n}
\eque
with 
\begin{align*}
J_{n,1}(t)&:= \esp\pp{ W_n\pp{1-\frac{W_n}{j_n^{1/2}}}^{-\alpha}}t^\alpha,\\
J_{n,2}(t)&:=\esp \pp{W_n\pp{\frac{\nu\spp{(n/j_n)(1-W_n/j_n^{1/2})^{-1}t}}{\nu(n/j_n)}- \pp{1-\frac{W_n}{j_n^{1/2}}}^{-\alpha}t^\alpha}}.
\end{align*}
We have
\begin{align*}
 \esp\pp{W_n\pp{1-\frac{W_n}{j_n^{1/2}}}^{-\alpha}} \sim \frac\alpha{\sqrt{j_n}}. 
\end{align*}
Indeed,
set
$\Omega_{n,\delta}:=\{|W_n/\sqrt{j_n}|\le \delta\}$
	for some $\delta\in (0,1)$ fixed. We have, by Taylor's expansion,
	\begin{align*}
		\E \left(W_n\left(1-\frac{W_n}{\sqrt{j_n}}\right)^{-\alpha}\ind_{\Omega_{n,\delta}}\right)
		&=\E \pp{W_n\ind_{\Omega_{n,\delta}}}+\frac{\alpha \E (W_n^2\ind_{\Omega_{n,\delta}})}{\sqrt{j_n}}+ \E \left(O\left(\frac{W_n^3}{j_n}\right)\ind_{\Omega_{n,\delta}}\right)\\
		& = -\esp \pp{W_n\ind_{\Omega_{n,\delta}^c}} + \frac \alpha{\sqrt{j_n}} - \frac\alpha{\sqrt{j_n}}\esp(W_n^2\ind_{\Omega_{n,\delta}^c})+ \E \left(O\left(\frac{W_n^3}{j_n}\right)\ind_{\Omega_{n,\delta}}\right)\\
		&=\frac\alpha{\sqrt{j_n}}(1+o(1)).	\end{align*}
Restricted to the event $\Omega_{n,\delta}^c$, we consider 2 subcases. First,
\begin{align*}
		\E\left|W_n\left(1-\frac{W_n}{\sqrt{j_n}}\right)^{-\alpha}\inddd{G_{j_n}<(1-\delta)j_n}\right|&\leq j_n^{\alpha+1/2}\esp\pp{G_{j_n}^{-\alpha}\inddd{G_{j_n}<(1-\delta)j_n}}
		\\
		& = \frac{j_n^{\alpha+1/2}\Gamma(j_n-\alpha)}{\Gamma(j_n)}\proba\pp{G_{j_n-\alpha}<(1-\delta)j_n},
\end{align*}
Using Chernoff's bound we see that the above expression decays exponentially in $j_n$. Similarly,
	\begin{align*}
		\E\left|W_n\left(1-\frac{W_n}{\sqrt{j_n}}\right)^{-\alpha}\inddd{G_{j_n}>(1+\delta)j_n}\right|&\leq \left(1+\delta\right)^{-\alpha}\E\pp{\left|W_n\right|\1_{G_{j_n}>(1+\delta)j_n}},
	\end{align*}
	and by the Cauchy--Schwarz inequality first and then Chernoff's bound again we see that the above decays exponentially in $j_n$.

Recall the decomposition of $I_{n,j_n}(t)$ in \eqref{eq:decomp I_n}. Therefore,
\equh\label{eq:J_n,1}
\limn j_n^{1/2} J_{n,1}(t)= \alpha t^\alpha,
\eque
and moreover the above asymptotic equivalence is uniform in $t\ge 0$. 
We next examine $J_{n,2}(t)$. Introduce 
\[
Y_n:=\pp{1-\frac{W_n}{j_n^{1/2}}}\inv = \frac{j_n}{G_{j_n}},
\]
and
\[
\Omega_{n,\delta}:=\ccbb{\abs{\frac{W_n}{\sqrt{j_n}}}\le \delta} = \ccbb{\frac1{1+\delta}\le Y_n = \frac{j_n}{G_{j_n}}\le \frac1{1-\delta}}, \delta>0.
\]
So
\equh
|J_{n,2}(t)| 
\le \esp\abs{W_n\pp{\frac{\nu((n/j_n)Y_nt)}{\nu(n/j_n)}-(Y_nt)^\alpha}\ind_{\Omega_{n,\delta}}}
+ \esp\abs{W_n\pp{\frac{\nu((n/j_n)Y_nt)}{\nu(n/j_n)}-(Y_nt)^\alpha}\ind_{\Omega_{n,\delta}^c}}.\label{eq:two cases}
\eque
We take $\delta>0$ small enough so that, on $\Omega_{n,\delta}$, $Y_n t\in[1-\epsilon,1+\epsilon]$ for all 
$t\in[1-\epsilon',1+\epsilon']$. Then, applying \eqref{eq:j_n rate} to the first term on the right-hand side of \eqref{eq:two cases} (this is the only place we need \eqref{eq:j_n rate}) we have
\[
\esp\abs{W_n\pp{\frac{\nu((n/j_n)Y_nt)}{\nu(n/j_n)}-(Y_nt)^\alpha}\ind_{\Omega_{n,\delta}}}\le C\esp|W_n|\rho_{[1-\epsilon,1+\epsilon]}(n/j_n)\le C\rho_{[1-\epsilon,1+\epsilon]}(n/j_n).
\]
For the second term on the right-hand side of \eqref{eq:two cases}, for every $\delta>0$ there exists a constant $C$ possibly depending on $\delta$ such that  (thanks to Potter's bound for $n$ large enough) 
\begin{align*}
\esp\abs{W_n\pp{\frac{\nu((n/j_n)Y_nt)}{\nu(n/j_n)}-(Y_nt)^\alpha}\ind_{\Omega_{n,\delta}^c}} & \le 
C\pp{\esp \pp{W_n^2\pp{|Y_nt|^{2(\alpha+\delta)}+|Y_nt|^{2(\alpha-\delta)}}}}^{1/2}\proba\pp{\Omega_{n,\delta}^c}^{1/2} 
\\
&\le
C\pp{t^{\alpha+\delta}\vee t^{\alpha-\delta}}\proba\pp{\Omega_{n,\delta}^c}^{1/2}.
\end{align*}
(In the inequality  we used the fact that $W_n\weakto\calN(0,1)$,  $Y_n\to 1$ almost surely, and also the corresponding convergence of moments.)
Moreover, $\proba(\Omega_{n,\delta}^c)$ decays to zero exponentially (in $j_n$). In particular, the above is of smaller order than $j_n^{-1/2}$, which is the order of $J_{n,1}(t)$.
That is, 
\[
\sup_{t\in[c,d]}|J_{n,2}(t)|\le C\rho_{[1-\epsilon,1+\epsilon]}(n/j_n) + o(j_n^{-1/2}).
\]
In order to have $J_{n,2}(t) = o(J_{n,1}(t))$, it suffices to impose 
$j_n^{1/2}\rho_{[1-\epsilon,1+\epsilon]} (n/j_n)\to 0$, which also implies immediately that the relation holds locally uniformly in $t$. This completes the proof of \eqref{eq:I_n asymp}. We have also seen that the estimates for $J_{n,1}$ in \eqref{eq:J_n,1} and for $J_{n,2}$ above hold locally uniformly. This completes the proof.
\end{proof}
We next compute the limiting covariance function. This calculation already suggests the limit to be an Ornstein--Uhlenbeck process.
	\begin{proposition}%\label{Prop: Covariance SC}
		Under \eqref{eq:j_n rate} we have
\[%\begin{equation}\label{eq: Cov asym}
\limn\frac1{\sigma_{n,j_n}^2}
\Cov \left(\wt C_{j_n}(s_n(t)), \wt C_{j_n}(s_n(t'))\right) =   e^{-|t-t'|}, \mfa t,t'\in\R.
\]%		\end{equation}
	\end{proposition}
	
	\begin{proof}
Assume $t<t'$. 	Using independence and conditioning we have 
\[
\Cov \left(\wt C_{j_n}(s_n(t)), \wt C_{j_n}(s_n(t'))\right)
 = \sif\ell1\cov\pp{\inddd{\calN_\ell(s_n(t)) = j_n}, \inddd{\calN_\ell(s_n(t')) = j_n}}
%&= \sum_{\ell=1}^\infty \pp{\frac{(s_n(t)p_\ell)^{j_n}e^{-s_n(t)p_\ell}}{j_n!} e^{-(s_n(t')-s_n(t))p_\ell}- \frac{ (s_n(t)p_\ell)^{j_n} e^{-s_n(t)p_\ell}}{j_n!}\frac{(s_n(t')p_\ell)^{j_n}e^{-s_n(t')p_\ell}}{j_n!}}\\
= \Psi_{n,1} - \Psi_{n,2},
\]
with
\begin{align*}
\Psi_{n,1}&:= \sum_{\ell=1}^\infty \frac{(s_n(t)p_\ell)^{j_n}e^{-s_n(t)p_\ell}}{j_n!} e^{-(s_n(t')-s_n(t))p_\ell} = \frac{1}{j_n!}\int_0^\infty e^{-s_n(t')/x} \left(\frac{s_n(t)}{x}\right)^{j_n} \d \nu(x),\\
\Psi_{n,2}&:= \sum_{\ell=1}^\infty \frac{ (s_n(t)p_\ell)^{j_n} e^{-s_n(t)p_\ell}}{j_n!}\frac{(s_n(t')p_\ell)^{j_n}e^{-s_n(t')p_\ell}}{j_n!} \\
&= \frac{1}{(j_n!)^2}\int_0^\infty e^{-(s_n(t)+s_n(t'))/x} \left(\frac{s_n(t)s_n(t')}{x^2}\right)^{j_n}\d \nu(x).
\end{align*}
First,
\begin{align*}
\Psi_{n,1}&:=\frac{s_n(t)^{j_n}}{j_n!}\int_0^\infty \pp{j_n-s_n(t')z}e^{-s_n(t')z}z^{j_n-1}\nu\pp{\frac 1z}\d z \\
& = \frac{(s_n(t)/s_n(t'))^{j_n}}{j_n!}\int_0^\infty \pp{j_n-z}e^{-z}z^{j_n-1}\nu \pp{\frac {s_n(t')}z}\d z = \frac{(s_n(t)/s_n(t'))^{j_n}}{j_n!}I_{n,j_n}(1+t'/j_n)\\
& \sim 
 e^{-(t'-t)}\frac{I_{n,j_n}(1+t'/j_n)}{\Gamma(j_n+1)} \sim \alpha e^{-(t'-t)}\frac1{j_n}\nu\pp{\frac n{j_n}} = e^{-(t'-t)}\sigma_{n,j_n}^2.
\end{align*}
We next show that $\Psi_{n,2} = o(\Psi_{n,1})$. 
		For $\Psi_{n,2}$, we integrate by parts again and use the substitution $z=(s_n(t)+s_n(t'))/x$:
\[
\Psi_{n,2}
= \left(\frac{s_n(t)s_n(t')}{(s_n(t)+s_n(t'))^2}\right)^{j_n}\frac1{\Gamma(j_n+1)^2}\int_0^\infty e^{-z} (2j_n-z)z^{2j_n-1}\nu\pp{\frac{s_n(t)+s_n(t')}z}\d z.
\]
We can write the integral on the right-hand side above as
\[
\Gamma(2j_n)\esp \pp{(2j_n-G_{2j_n})\nu\pp{\frac n{j_n}\frac{2j_n}{G_{2j_n}}\pp{1+\frac{t+t'}{2j_n}}}}.
\]	
Again, $G_{2j_n}$ is a Gamma random variable with parameter $2j_n$ and $G_{2j_n}/(2j_n)$ is concentrated around $1$. We can apply the same analysis as before to $I_{n,j_n}$ and conclude with 
\[
\Psi_{n,2} \sim 
 \left(\frac{s_n(t)s_n(t')}{(s_n(t)+s_n(t'))^2}\right)^{j_n}\frac1{\Gamma(j_n+1)^2} \alpha \Gamma(2j_n)\nu\pp{\frac n{j_n}}\sim\frac\alpha{2\sqrt\pi}\frac1{j_n^{3/2}}\nu\pp{\frac n{j_n}} = o(\Psi_{n,1}).
\]
The details are omitted.
This completes the proof.
	\end{proof}

Now, we are ready to prove the convergence of the Poissonized model stated in Proposition \ref{prop:1}. We write
\[%	\begin{equation}\label{eq: SC Poissonized process}
		Z_n(t):=\wt C_{j_n}\pp{s_n(t)} - \esp \wt C_{j_n}\pp{s_n(t)} \qmwith s_n(t) = n\pp{1+\frac t{j_n}}_+.
\]%	\end{equation}
Recall that  $(\zeta_t)_{t\in\R}$ denotes the Ornstein--Uhlenbeck process with covariance function $e^{-|s-t|}$. We proceed by proving the convergence of finite-dimensional distributions and the tightness, respectively.
\begin{proof}[Proof of convergence of finite-dimensional distributions] 
		Fix $t_1,\dots,t_d$ and $a_1,...,a_d\in \R$. By the Cram\'er--Wold device, it suffices to show the following:
\begin{equation}\label{eq:fdd CLT}
\frac1{\sigma_{n,j_n}}\sum_{i=1}^d a_i Z_n(t_i) \Rightarrow \summ i1d a_i \zeta_{t_i}.
\end{equation}
Note that 
\[
\summ i1d a_iZ_n(t_i) = \sif\ell1\pp{\summ i1d a_i\inddd{\calN_\ell(s_n(t_i)) = j_n} - \esp \pp{\summ i1d a_i\inddd{\calN_\ell(s_n(t_i)) = j_n} }},
\]
which is a summation of independent bounded random variables. Moreover, 
\begin{align*}
\var\left(\sum_{i=1}^d a_i Z_n(t_i)\right) &= \summ i1d\summ j1d a_ia_j\cov\pp{\wt C_{j_n}\pp{s_n(t_i)},\wt C_{j_n}\pp{s_n(t_j)}}\\
& \sim \sigma_{n,j_n}^2\summ i1d\summ j1d a_ia_je^{-|t_i-t_j|},
		\end{align*}
		as $n\to\infty$. By the Lindeberg--Feller central limit theorem, we have thus proved \eqref{eq:fdd CLT}.	\end{proof}

\begin{proof}[Proof of tightness]
We first introduce a decomposition of $\wt C_j(t)-\esp \wt C_j(t)$. 
 Set
\[
H_{\ell,j}(s):=\inddd{\calN_\ell(s-)=j-1} - \inddd{\calN_\ell(s-) = j}. 
\]
Recall that $(\calN_\ell(t))_{t\ge 0}$ is a Poisson process with parameter $p_\ell$, and we let $\what \calN_\ell(t):=\calN_\ell(t)-p_\ell t$ denote the compensated process. 
Then,
	\begin{equation*}
		\wt C_{j}(t)=\sum_{\ell=1}^\infty \int_0^t  H_{\ell,j}(s)\d \calN_\ell(s)=\sum_{\ell=1}^\infty \int_0^t H_{\ell,j}(s)p_\ell \d s+ \sum_{\ell=1}^\infty \int_0^t  H_{\ell,j}(s)\d \what \calN_\ell(s).
	\end{equation*}
Note that the second series is already centered. Thus, with
\[
		B_{j}(s):= \sum_{\ell=1}^\infty H_{\ell,j}(s)p_\ell,
\qmand 		M_{j}(t):= \sum_{\ell=1}^\infty \int_{0}^{t}  H_{\ell,j}(s) \d \what \calN_\ell (s),
\]
we have
\[
\wt C_j(t)-\esp \wt C_j(t) = \int_0^t\pp{B_j(s)-\esp B_j(s)}\d s + M_j(t).
\]
	Using the decomposition above for $\wt C_j(t)$, we have
	\begin{equation*}
		Z_n(t)=\int_{0}^{n(1+t/j_n)}(B_{j_n}(s)-\E B_{j_n}(s))\d s+ M_{j_n}(n(1+t/j_n)).
	\end{equation*}

We shall prove the tightness of $
\sigma_{n,j_n}\inv(Z_n(t))_{t\in[a,b]}$ for all $-\infty<a<b<\infty$. For this purpose we decompose further (assuming $n$ large enough so that $1+a/j_n\ge 0$)
 into a martingale part and a compensator:
\begin{align*}
Z_n(t)& = Z_n(a)+\int_{n(1+a/j_n)}^{n(1+t/j_n)}(B_{j_n}(s)-\esp B_{j_n}(s))\d s+M_{j_n}(n(1+t/j_n))-M_{j_n}(n(1+a/j_n))\\
& =:Z_n(a)+A_n\topp a(t)+M_n\topp a(t), \quad t\in[a,b],
\end{align*}
and we shall prove the tightness of $\sigma_{n,j_n}\inv Z_n(a), \sigma_{n,j_n}\inv(A_n\topp a(t))_{t\in[a,b]}, \sigma_{n,j_n}\inv(M\topp a_n(t))_{t\in[a,b]}$ respectively. 

For notational simplicity, we give the proof for $a=0, b=T$.
(The proof is essentially the same for all $a<b$.) 
 From now on, write
	\begin{equation*}%\label{eq: Prelim OU SDE}
		Z_n(t)=Z_n(0)+A_n(t)+ M'_n(t),
	\end{equation*}
	where 
	\begin{align*}
		M'_n(t)\equiv M\topp0_n(t) &:=M_{j_n}\pp{n(1+ t/{j_n})}-M_{j_n}(n),\\
		A_n(t)\equiv A\topp0_n(t) &:= \int_{n}^{n(1+t/j_n)}(B_{j_n}(s)-\E B_{j_n}(s))\d s.
	\end{align*}
Note that $Z_n(0) = \wt C_{j_n}(n)-\esp \wt C_{j_n}(n)$ and we have shown that $\limn\var(Z_n(0)/\sigma_{n,j_n}) = 1$. So $\sigma_{n,j_n}\inv(Z_n(0))_{n\in\N}$ is a tight sequence.	The proof is completed by establishing the tightness of $\sigma_{n,j_n}\inv M_n'$ and $\sigma_{n,j_n}\inv A_n$ in Lemmas \ref{lem:M_n'} and \ref{lem:A_n} below, respectively.
	\end{proof}
\begin{lemma}\label{lem:M_n'} 
With the notation above,
\[
\frac1{\sigma_{n,j_n}}\pp{M'_n(t)}_{t\ge 0} \weakto (\B_{2t})_{t\ge 0}
\]
in $D[0,\infty)$, where $(\B_t)_{t\ge 0}$ is a standard Brownian motion. 
\end{lemma}
\begin{proof}
We apply  the functional central limit theorem on martingales  from  \citet[Theorem 2.1]{whitt07proofs}. For all $T>0$ we define
	\begin{equation*}
		J(x,T):=\sup\{|x(t)-x(t-)|: 0<t \leq T\}
	\end{equation*}
Then, notice that for each $n\geq 1$, $M'_n$ is a martingale in $D([0,\infty))$ with respect to the filtration 
		\[
		\calF_t\topp n:=\sigma\ccbb{\calN_\ell(s): 0\le s\leq n\pp{1+\frac t{j_n}}, \ell\in \N},
		\] and $M'_n(0)=0$. 
Then, to prove the desired convergence it suffices to show that for all $T>0$, 
		\begin{equation}\label{eq: jumps of pred var}
			\lim_{n\rightarrow\infty} \frac1{\sigma_{n,j_n}}\E J(\langle M'_n\rangle,T)=0, 
		\end{equation}
		\begin{equation}\label{eq: square jumps}
			\lim_{n\rightarrow\infty}\frac1{\sigma_{n,j_n}^2} \E J(M'_n,T)^2=0,
		\end{equation}
		and 
		\begin{equation}\label{eq: quad var convergence}
\limn \frac1{\sigma_{n,j_n}^2}\langle M'_n \rangle(t) =  2 t \mbox{ in probability} \mfa t\ge 0.
		\end{equation}

The bracket process has no jumps (the definition will be recalled below), and hence
 \eqref{eq: jumps of pred var} follows immediately.
The condition \eqref{eq: square jumps} is also easy to check. It follows from the assumption that $\sigma_{n,j_n}\to\infty$ while $J(M_n',T)\le 1$ almost surely. The latter is because  for all $t\ge 0$, 
$\left|M'_n(t)-M'_n(t-)\right|\le 1$,		almost surely, since $|H_{\ell,j}(s)|\le 1$ for all $s>0$ and at each time $t$ there exists at most one jump from all the Poisson processes involved. 
		
It remains to prove \eqref{eq: quad var convergence}. We compute the bracket process for $M'_n$. We have
		\begin{align*}
			\langle M'_n \rangle (t)&= \sum_{\ell=1}^\infty \int_{n}^{n(1+t/j_n)} p_\ell H_{\ell,j_n}^2(s)\d s\\
			&=\sum_{\ell=1}^\infty \int_{n}^{n(1+t/j_n)} p_\ell \left(\inddd{\calN_\ell(s)=j_n-1}+\inddd{\calN_\ell(s)=j_n}\right)\d s=\int_n^{n(1+t/j_n)}D_n(s) \d s,
		\end{align*}
%\newpage
with
\[
D_n(s) := \sif\ell1p_\ell\pp{\inddd{\calN_\ell(s) = j_n-1}+\inddd{\calN_\ell(s) = j_n}}.
\]
First we shall show that
\equh\label{eq:<M_n> mean}
\limn\frac1{\sigma_{n,j_n}^2}\E \langle M'_n \rangle (t)= 2t, \mfa t\ge 0.
\eque
Indeed,
		\begin{align*}
\int_n^{n(1+t/j_n)} \E D_n(s) \d s&=  \int_{n}^{n(1+t/j_n)}\left(\sum_{\ell=1}^\infty \frac{p_\ell (p_\ell s)^{j_n-1}e^{-p_\ell s}}{(j_n-1)!} + \frac{p_\ell (p_\ell s)^{j_n} e^{-p_\ell s}}{j_n!}\right) \d s\\
&= \int_n^{n(1+t/j_n)}\left( \frac{j_n}{s} \E \wt C_{j_n}(s) + \frac{j_n+1}{s} \E \wt C_{j_n+1}(s)\right) \d s\\
&= \int_{0}^{t}\frac{\E \wt C_{j_n}(n(1+s/j_n))}{1+s/j_n}\d s + \frac{j_n+1}{j_n}\int_{0}^{t}\frac{\E \wt C_{j_n+1}(n(1+s/j_n))}{1+s/j_n}\d s.
		\end{align*}
By \eqref{eq:EC uniform}, 
\[
\limn \frac{\esp \wt C_{j_n}(n(1+s/j_n))}{\sigma_{n,j_n}^2(1+s/j_n)^\alpha} = 1,
\]
uniformly for $s\in[0,t]$, and a similar limit theorem holds for $\esp\wt C_{j_n+1}(n(1+s/j_n))$ (note also that $\sigma_{n,j_n}^2\sim \sigma_{n,j_n+1}^2$).  Therefore 
\begin{equation*}
\limn	\frac1{\sigma_{n,j_n}^2}\int_n^{n(1+t/j_n)} \E D_n(s) \d s= 2t.
\end{equation*}
We have thus proved \eqref{eq:<M_n> mean}.		
		To finish the proof of \eqref{eq: quad var convergence}, it suffices to show that 
\equh\label{eq:var <M_n>}
\limn\frac1{\sigma_{n,j_n}^4}\Var(\langle M'_n\rangle(t))=0, \mfa t>0.
\eque
We start by expanding:
\begin{equation}
\Var(\langle M'_n\rangle(t))=2\int_n^{n(1+t/j_n)}\int_r^{n(1+t/j_n)} \Cov (D_n(r), D_n(s))\d s\d r.\label{eq:var <M_n> 1}
\end{equation}
We shall show that the above is of order at most $\sigma_{n,j_n}^2$.
Let us define 
\[
X_{\ell,j_n}(r)=\inddd{\calN_\ell(r)=j_n-1}+\inddd{\calN_\ell(r)=j_n}.
\] Assume $r\le s$. We compute a bound for $\Cov (D_n(r), D_n(s))$. We start with
\[
\Cov (D_n(r), D_n(s))
=\sum_{\ell=1}^\infty p_\ell^2 \Cov(X_{\ell,j_n}(r),X_{\ell,j_n}(s))\leq \sum_{\ell=1}^\infty p_\ell^2 \E \left(X_{\ell,j_n}(r) X_{\ell,j_n}(s)\right).
\]
Write
\begin{align*}
\esp (X_{\ell,j_n}(r)X_{\ell,j_n}(s))& = \proba\pp{\calN_\ell(r) = j_n-1, \calN_\ell(s)\in\{j_n-1,j_n\}} +\proba\pp{\calN_\ell(r)=\calN_\ell(s) = j_n}\\
 & = \proba(\calN_\ell(r) = j_n-1)e^{-p_\ell(s-r)}\pp{1+p_\ell(s-r)} + \proba(\calN_\ell(r) = j_n)e^{-p_\ell(s-r)}.
\end{align*}
Notice also
\[
\int_r^{n(1+t/j_n)}  p_\ell^2e^{-p_\ell(s-r)}\d s \le \int_0^\infty p_\ell^2e^{-p_\ell s}\d s= p_\ell,
\]
and similarly
\[
\int_r^{n(1+t/j_n)}  p_\ell^2e^{-p_\ell(s-r)}(1+p_\ell(s-r))\d s \le \int_0^\infty p_\ell^2e^{-p_\ell s}(1+p_\ell s)\d s= 2p_\ell.
\]
Then, for the inner integral of \eqref{eq:var <M_n> 1},
\begin{align*}
\int_r^{n(1+t/j_n)} \Cov (D_n(r), D_n(s))\d s &\le 2\sif\ell1 p_\ell\pp{\proba\pp{\calN_\ell(r)=j_n-1}+\proba\pp{\calN_\ell(r) = j_n}} \\
& = 2\pp{\frac{j_n}r \esp\wt C_{j_n}(r) + \frac{j_n+1}r\esp\wt C_{j_n+1}(r)}.
\end{align*}
Thus, \eqref{eq:var <M_n> 1} becomes
\begin{align*}
\Var(\langle M'_n\rangle(t))&=2\int_n^{n(1+t/j_n)}\int_r^{n(1+t/j_n)} \Cov (D_n(r), D_n(s))\d s\d r\\
& \le 4\int_n^{n(1+t/j_n)}\pp{\frac{j_n}r \esp\wt C_{j_n}(r) + \frac{j_n+1}r\esp\wt C_{j_n+1}(r)}\d r \sim 8t \sigma_{n,j_n}^2,
\end{align*}
where the last step follows again from the uniform estimate in \eqref{eq:EC uniform}.
We have proved that $\var(\aa{M'_n}(t))\le C \sigma_{n,j_n}^2$ and hence \eqref{eq:var <M_n>} holds. 
This completes the proof.
	\end{proof}
	\begin{lemma}\label{lem:A_n}For all $T>0$, the process $\sigma_{n,j_n}\inv(A_n(t))_{t\in[0,T]}$ is tight in $C[0,T]$.
	\end{lemma}
	\begin{proof}
We shall apply  \citet[Theorem 7.3]{billingsley99convergence}. Notice that  $A_n(0) = 0$. 	Therefore, it suffices to show
\equh\label{eq:A tightness}
\lim_{\delta\downarrow 0}\limsup_{n\to\infty}\P\left(\frac1{\sigma_{n,j_n}} \sup_{u,v\in[0,T], |u-v|<\delta}\left| A_n(u)-A_n(v)\right|\geq \epsilon\right) = 0, \mfa \epsilon>0.	
\eque
We start by using the bound given in \citet[Theorem 7.4]{billingsley99convergence}. Suppose that $\{t_i\}_{i=1,...,m}$ is such that $t_0=0, t_m=T$ and $t_i-t_{i-1}=\delta$ with $m=\lceil T/\delta\rceil$. For all $\epsilon>0$, we have
\begin{multline*}%\label{eq: A_n tightness}
\P\left(\frac1{\sigma_{n,j_n}} \sup_{u,v\in[0,T],|u-v|<\delta}\left| A_n(u)-A_n(v)\right|\geq 3\epsilon\right)\\
\leq \sum_{i=1}^m \P\left(\frac1{\sigma_{n,j_n}}\sup_{t_{i-1}\leq u\leq t_i} \left|A_n(u)-A_n(t_{i-1})\right|\geq \epsilon \right).
		\end{multline*}
For each probability on the right-hand side above, we start by applying the Markov inequality
		\begin{align*}
		\P&\left(\frac1{\sigma_{n,j_n}}\sup_{t_{i-1}\leq u\leq t_i} \left|A_n(u)-A_n(t_{i-1})\right|\geq \epsilon \right)\\
		&\leq \frac{1}{\epsilon^2 \sigma_{n,j_n}^2} \E \left( \sup_{t_{i-1}
		\leq u\leq t_i} \left| \int_{n(1+t_{i-1}/j_n)}^{n(1+u/j_n)} \left( B_{j_n}(t) -\E B_{j_n}(t)\right) \d t\right|^2\right)\\
			&\leq \frac{1}{\epsilon^2 \sigma_{n,j_n}^2} \E \left( \int_{n(1+t_{i-1}/j_n)}^{n(1+t_i/j_n)} |B_{j_n}(t)-\E B_{j_n}(t)|\d t\right)^2\leq \frac{1}{\epsilon^2 \sigma_{n,j_n}^2}\frac{n\delta}{j_n} \int_{n(1+t_{i-1}/j_n)}^{n(1+t_i/j_n)} \Var \left(B_{j_n}(t)\right) \d t,
		\end{align*}
		where in the third inequality we used the Cauchy--Schwarz inequality. Next, by independence we have
		\begin{align*}
			\Var\left(B_{j_n}(t)\right)&=\sum_{\ell=1}^\infty p_\ell^2 \Var\pp{\inddd{\calN_\ell(t)=j_n-1}-\inddd{\calN_\ell(t)=j_n}}\\
			&\le\sum_{\ell=1}^\infty p_\ell^2 \pp{\proba\pp{\calN_\ell(t)=j_n-1}+\proba\pp{\calN_\ell(t)=j_n}}\\
			&=\frac{j_n(j_n+1)}{t^2}\E \wt C_{j_n+1}(t)+ \frac{(j_n+1)(j_n+2)}{t^2} \E \wt C_{j_n+2}(t).
%			&\leq \frac{(j_n+2)^2}{t^2}\left(\E \wt C_{j_n+1}(t)+\E \wt C_{j_n+2}(t)\right).
		\end{align*}
Then, 
		\begin{align*}
			\frac{n \delta}{\epsilon^2 \sigma_{n,j_n}^2 j_n} \int_{n(1+t_{i-1}/j_n)}^{n(1+t_i/j_n)} \Var \left(B_{j_n}(t)\right) \d t&\leq \frac{n \delta (j_n+2)^2}{\epsilon^2 \sigma_{n,j_n}^2 j_n} \int_{s_n(t_{i-1})}^{s_n(t_i)} \frac{1}{t^2} \left(\E \wt C_{j_n+1}(t)+ \E \wt C_{j_n+2}(t)\right)\d t\\
			&=\frac{  \delta (j_n+2)^2}{\epsilon^2 \sigma_{n,j_n}^2 j_n^2}\int_{t_{i-1}}^{t_i} \frac{\E \wt C_{j_n+1}(s_n(t))+\E \wt C_{j_n+2}(s_n(t))}{(1+t/j_n)^2}\d t
\\
&			\le C \frac{\delta^2}{\epsilon^2},
\end{align*}
for some constant $C$ that does not depend on $\delta$ nor $\epsilon$ by \eqref{eq:EC uniform} again. Thus, 
\begin{equation*}
\limsup_{n\to\infty}\P\left(\frac1{\sigma_{n,j_n}} \sup_{u,v\in[0,T],|u-v|<\delta}\left| A_n(u)-A_n(v)\right|\geq 3\epsilon\right)\leq \frac{C m \delta^2}{\epsilon^2},  
		\end{equation*}
		which tends to zero as $\delta\downarrow 0$ (recall $m = \ceil{T/\delta}$). We have thus proved \eqref{eq:A tightness}.
	\end{proof}
%	\newpage
	\subsection{De-Poissonization}
	It remains to show that the approximation error between the Poissonized model and the original model is negligible. 
	\begin{proof}[Proof of Theorem \ref{thm:1}]
To simplify the notation, introduce
\[
\wt W_n(t) := \wt C_{j_n}(s_n(t)) \qmand 
W_n(t) :=C_{j_n}(\floor{s_n(t)}), \quad t\in\R, 
\]
and recall $s_n(t) = n(1+t/j_n)_+$. Let $(\tau_n)_{n\in\N}$ be the consecutive arrival times of the Poisson process $\calN$ for the Poissonized model. 
It is well known that one can couple the Poissonized model and the original Karlin model (i.e., define the two on the same probability space) such that for every $K>0$
\[
\pp{ C_{j_n}(\floor{s_n(t)})}_{t\in[-K,K]} = \pp{\wt C_{j_n}(\tau_{\floor{s_n(t)}})}_{t\in[-K,K]}, \mbox{ almost surely.}
\]
For each $t\in\R$, introduce a random variable $\what t_n$ determined by 
\[%\equh\label{eq:tau_s}
\tau_{\floor{s_n(t)}} = n\pp{1+\frac{\what t_n}{j_n}} =s_n(\what t_n). 
\]%\eque
From now on, assume $n$ large enough such that $s_n(-K)>0$. 
With a little abuse of notation, $\what t_n$ is a random variable depending on $t$.
Thus, 
\[
\pp{W_n(t)}_{t\in[-K,K]} = \pp{\wt W_n(\what t_n)}_{t\in[-K,K]},\quad \mbox{ almost surely.}
\]
Write 
\begin{align*}
\wt m_n(t) &:= \esp \wt W_n(t) = \esp \wt C_{j_n}(s_n(t)),\\
m_n(t) &:= \esp W_n(t) = \esp C_{j_n}(\floor{s_n(t)}).
\end{align*} 
In the notation above, the goal becomes to prove for all $K>0$,
\equh\label{eq:0}
\frac1{\sigma_{n,j_n}}\pp{\wt W_n(\what t_n) -  m_n(t)}_{t\in[-K,K]} \weakto (\zeta(t))_{t\in[-K,K]},
\eque
in $D[-K,K]$ as $n\to\infty$. 
The proof consists of two parts. We shall first prove that for all $K>0$,
\equh\label{eq:1}
\frac1{\sigma_{n,j_n}}\pp{\wt W_n(\what t_n) - \wt m_n(\what t_n)}_{t\in[-K,K]} \weakto (\zeta(t))_{t\in[-K,K]},
\eque
in $D[-K,K]$, and second
\equh\label{eq:2}
\limn\frac1{\sigma_{n,j_n}}\sup_{t\in[-K,K]}\abs{m_n(t)-\wt m_n(\what t_n)}=0 \mbox{ almost surely}.
\eque
The claim  \eqref{eq:0} then follows. \medskip

(i) We first prove \eqref{eq:1}. 
We have shown that 
\equh\label{eq:JG}
\frac1{\sigma_{n,j_n}}\pp{\wt W_n(t) - \wt m_n(t)}_{t\in\R}\weakto (\zeta(t))_{t\in\R},
\eque
as $n\to\infty$. Recall that it is assumed that $j_n\to\infty$, and hence $s_n(t)\sim n$. 
By the law of iterated logarithm, for $n$ large enough we have $|\tau_n-n|\le 2\sqrt{2n\log\log n}$. Since $\floor{s_n(t)} = n(1+o(1))$ uniformly over $t\in[-K,K]$, it follows that 
\[
\sup_{t\in[-K,K]}\abs{\tau_{\floor{s_n(t)}}-s_n(t)}\le C\sqrt{n\log\log n}
\]
for all $n$ large enough, almost surely. 
Then, writing
\[
\tau_{\floor{s_n(t)}} -s_n(t) =  n\pp{\frac{\what t_n-t}{j_n}},
\]
we have, for all $K>0$ fixed, 
\equh\label{eq:what t_n}
\sup_{t\in[-K,K]}\abs{\what t_n-t}\le C\frac{j_n}{\sqrt n}\sqrt{\log\log n}\quad \mbox{ for $n$ large enough, almost surely.}
\eque
We are in the regime that $j_n\ll n^{\alpha/(\alpha+1)}\ll n^{1/2}$. Therefore, the right-hand side above goes to zero almost surely, and hence \eqref{eq:1} now follows from applying Lemma \ref{lem:time change} to \eqref{eq:JG}. (Strictly speaking, viewing $\what t_n$ as a process of $t$, we need $\what{(-K)}_n\ge -K-\epsilon$, $\what K_n\le K+\epsilon$ for all $n\in\N$ to apply the time change lemma. For this purpose one may simply proceed by restricting to the above event, the probability of which goes to one.)
\medskip	

(ii) We now prove \eqref{eq:2}. We first show that
\equh\label{eq:2''}
\limn\frac1{\sigma_{n,j_n}}\sup_{t\in[-K,K]}\abs{m_n(t)-\wt m_n(t)}=0.
\eque
For convenience, introduce
\[
\wt m_{n,j}(u):=\esp\wt C_j(u) = \sif\ell1 \frac{(p_\ell u)^j}{j!}e^{-p_\ell u},\qquad u>0.
\]
Note that $\wt m_n(t) = \wt m_{n,j_n}(s_n(t))$. For each fixed $r\in\{0,1,2\}$, we have $\sigma_{n,j_n+r}\sim\sigma_{n,j_n}$. Hence, applying \eqref{eq:EC uniform} with $j_n+r$ in place of $j_n$, we obtain
\begin{equation}\label{eq:mean bound dP}
\sup_{u\in[s_n(-K-1),s_n(K+1)]}\wt m_{n,j_n+r}(u)\leq C\sigma_{n,j_n}^2,\qquad r\in\{0,1,2\},
\end{equation}
for all sufficiently large $n$.

We first compare the binomial and Poissonized means. We use the following fact: there exists a constant $C>0$ such that, for all $m\geq1$, $j=0,1,\dots,m$, and $p\in[0,1]$,
\begin{align}
\left|\binom{m}{j}p^j(1-p)^{m-j}-\frac{(mp)^je^{-mp}}{j!}\right|
&\leq\frac{C(mp)^je^{-mp}}{j!}\max\ccbb{mp^2,\frac{j^2}{m}}\nonumber\\
&\leq\frac{C(j+2)^2}{m}\max\ccbb{\frac{(mp)^{j+2}e^{-mp}}{(j+2)!},\frac{(mp)^je^{-mp}}{j!}}.
\label{eq:binomial-Poisson-local}
\end{align}
The first inequality is discussed in \citet[Eq.~(1.1)]{barbour92poisson}, and the second is a direct rewrite. Applying \eqref{eq:binomial-Poisson-local} with $m=\floor{s_n(t)}$, and noting that $m\sim n$ uniformly for $t\in[-K,K]$, we have
\[%\equh\label{eq:binomial poisson mean}
\abs{m_n(t)-\wt m_{n,j_n}(\floor{s_n(t)})}
\leq C\frac{j_n^2}{n}\left(\wt m_{n,j_n}(\floor{s_n(t)})+\wt m_{n,j_n+2}(\floor{s_n(t)})\right).
\]%\eque
Consequently, by \eqref{eq:mean bound dP},
\[
\sup_{t\in[-K,K]}\frac{\sabs{m_n(t)-\wt m_{n,j_n}(\floor{s_n(t)})}}{\sigma_{n,j_n}}\leq C\frac{j_n^2\sigma_{n,j_n}}{n}.
\]
It remains to compare $\wt m_{n,j_n}(\floor{s_n(t)})$ and $\wt m_{n,j_n}(s_n(t))$. Differentiating term by term gives
\[
\frac{\partial}{\partial u}\wt m_{n,j}(u)=\frac{j}{u}\wt m_{n,j}(u)-\frac{j+1}{u}\wt m_{n,j+1}(u).
\]
Therefore, by \eqref{eq:mean bound dP},
\[
\sup_{u\in[s_n(-K-1),s_n(K+1)]}\abs{\frac{\partial}{\partial u}\wt m_{n,j_n}(u)}\leq C\frac{j_n}{n}\sigma_{n,j_n}^2
\]
for all sufficiently large $n$. Since $\abs{\floor{s_n(t)}-s_n(t)}<1$, it follows that
\[
\sup_{t\in[-K,K]}\frac{\sabs{\wt m_{n,j_n}(\floor{s_n(t)})-\wt m_n(t)}}{\sigma_{n,j_n}}\leq C\frac{j_n\sigma_{n,j_n}}{n}.
\]
Combining the last two bounds yields
\[
\frac1{\sigma_{n,j_n}}\sup_{t\in[-K,K]}\abs{m_n(t)-\wt m_n(t)}\leq C\left(\frac{j_n^2\sigma_{n,j_n}}{n}+\frac{j_n\sigma_{n,j_n}}{n}\right)\to0.
\]
This proves \eqref{eq:2''}. 

We next show that
\equh\label{eq:2'}
\limn\frac1{\sigma_{n,j_n}}\sup_{t\in[-K,K]}\abs{\wt m_n(t)-\wt m_n(\what t_n)}=0
\qquad\mbox{almost surely}.
\eque
By \eqref{eq:what t_n}, almost surely, for all sufficiently large $n$, we have $\what t_n\in[-K-1,K+1]$ for every $t\in[-K,K]$. Hence, by the mean-value theorem and \eqref{eq:mean bound dP},
\begin{align*}
\sup_{t\in[-K,K]}\abs{\wt m_n(t)-\wt m_n(\what t_n)}
&\leq C\frac{j_n}{n}\sigma_{n,j_n}^2\sup_{t\in[-K,K]}\abs{s_n(t)-s_n(\what t_n)}\\
&=C\sigma_{n,j_n}^2\sup_{t\in[-K,K]}\abs{t-\what t_n}\leq C\sigma_{n,j_n}^2\frac{j_n}{\sqrt n}\sqrt{\log\log n}
\end{align*}
eventually almost surely. Therefore,
\[
\frac1{\sigma_{n,j_n}}\sup_{t\in[-K,K]}\abs{\wt m_n(t)-\wt m_n(\what t_n)}\leq C\frac{j_n\sigma_{n,j_n}}{\sqrt n}\sqrt{\log\log n}\to0
\]
almost surely. Combining \eqref{eq:2''} and \eqref{eq:2'} proves \eqref{eq:2}.
\end{proof}
\subsection{Proof of Theorem \ref{thm:CRP}}\label{sec:0 process}
Theorem \ref{thm:CRP} concerning $(\alpha,\theta)$-partitions now follows as a corollary of Theorem \ref{thm:1} and Proposition \ref{prop:0 process}. We first check that the conditions imposed in Theorem \ref{thm:1} are satisfied under $j_n\to\infty$ and
\equh\label{eq:j_n CRP}
j_n = o\pp{\frac{n^{\alpha/(1+\alpha)}}{(\log\log n)^{1/(1+\alpha)}}}.
\eque
Indeed, for $(\alpha,\theta)$-partitions we have that for all $K\in(0,\infty)$ there exists $C>0$ such that
\equh\label{eq:assump nu}
\rho_{[0,K]}(n)  = \sup_{x\in[0,K]}\abs{\frac{\nu(nx)}{\nu(n)}-x^\alpha} \le Cn^{-\alpha/2}\sqrt{\log\log n}, \mbox{ for all $n$ large enough}.
\eque
By applying the representation in Lemma \ref{lem:N(D)} it suffices to check that $\nu(x) = N(Dx^\alpha)$ satisfies \eqref{eq:assump nu}. This is an almost sure statement. Therefore by part (ii) of Lemma \ref{lem:P_j}, it suffices to show the above with $\theta=0$, in which case $N$ is a standard Poisson process. Indeed, we have
\[
\frac{\nu(nx)}{\nu(n)}-x^\alpha = \frac{N(D(nx)^\alpha) - D(nx)^\alpha}{N(Dn^\alpha)} + \pp{\frac{Dn^\alpha}{N(Dn^\alpha)}-1}x^\alpha. 
\]
By the law of iterated logarithm, 
\[
\limsupn \frac{|N(n)-n|}{\sqrt{2n\log\log n}} = 1,
\]
whence \eqref{eq:assump nu} holds almost surely.

In summary, \eqref{eq:j_n CRP} and \eqref{eq:assump nu} combined yield \eqref{eq:j_n rate} in Theorem \ref{thm:1}, which now says
\[
\pp{\frac{1}{\sigma_{n,j_n} }\pp{C_{j_n}\pp{\floor{s_n(t)}}-\E \pp{C_{j_n}\pp{\floor{s_n(t)}}\mid \calP}}}_{t\in \R} \aswto (\zeta(t))_{t\in \R},
\]
with respect to $\calP$ as $n\to\infty$. Recall that 
\[
\sigma_{n,j_n} = \pp{\frac\alpha{j_n}N(D(n/j_n)^\alpha)}^{1/2}\sim \pp{\frac\alpha{\Gamma(1-\alpha)}S_\alpha\frac{n^\alpha}{j_n^{1+\alpha}}}^{1/2},
\]
almost surely. So, the above convergence is very close to the claimed convergence in Theorem \ref{thm:CRP}, except that the centering is different; but applying Proposition \ref{prop:0 process} concludes the proof. 

It remains to prove Proposition \ref{prop:0 process}.

\begin{proof}[Proof of Proposition \ref{prop:0 process}]
By absolute continuity of
\(\operatorname{PD}(\alpha,\theta)\) with respect to
\(\operatorname{PD}(\alpha,0)\), it suffices to work with
\(\theta=0\).  
%Thus
%\[
%P_\ell^\downarrow=D^{1/\alpha}\Gamma_\ell^{-1/\alpha},
%\qquad
%D=\frac{S_\alpha}{\Gamma(1-\alpha)},
%\]
%where \((\Gamma_\ell)_{\ell\in\N}\) are the consecutive arrival times of a standard Poisson process \(N\).  

We first prove both statements for the corresponding Poissonized model. 
That is, with 
\[
\Delta_n(t):=
\esp\pp{\wt C_{j_n}(s_n(t))\mmid\calP}
-s_n(t)^\alpha p_{j_n}\topp\alpha S_\alpha,
\]
we shall prove
\begin{equation}\label{eq:random centering}
\limn\sup_{t\in[-T,T]}\frac{\abs{\Delta_n(t)}}{\sigma_{n,j_n}}
=0
\end{equation}
in probability, and almost surely under
\eqref{eq:lower growth combined}.

{\em We first prove the convergence in probability in \eqref{eq:random centering}.}
Write
\[
V_n(s):=
\frac{\wb N(D(n/j_n)^\alpha s^\alpha)}{(n/j_n)^{\alpha/2}},
\qmwith \wb N(x):=N(x)-x.
\]
Let \(G_{j}\) be a Gamma random variable with parameter \(j\), independent
of \(N\), and write
\[
W_n:=\frac{j_n-G_{j_n}}{\sqrt{j_n}},
\qmand R_n(t):=\frac{j_n}{G_{j_n}}\pp{1+\frac{t}{j_n}}.
\]
Introduce also
\[
r_n:=j_n^{-1/2}\pp{\frac{n}{j_n}}^{\alpha/2}.
\]
Notice that
\begin{equation}\label{eq:sigma rn combined}
\frac{\sigma_{n,j_n}^2}{r_n^2}
=\alpha
\frac{N(D(n/j_n)^\alpha)}{(n/j_n)^\alpha}
\to\alpha D
\qquad\mbox{almost surely}.
\end{equation}
With
\(\esp_N(\cdot):=\esp(\cdot\mid(N(t))_{t\geq0})\), integration by parts
gives
\begin{align}
\frac{\Delta_n(t)}{r_n}
&=
\frac{\sqrt{j_n}}{\Gamma(j_n+1)}
\int_0^\infty
e^{-z}(j_n-z)z^{j_n-1}
\frac{\wb N(D(s_n(t)/z)^\alpha)}{(n/j_n)^{\alpha/2}}
\d z \nonumber\\
&=\esp_N\pp{W_nV_n(R_n(t))}
=\esp_N\pp{W_n\pp{V_n(R_n(t))-V_n(1)}},
\label{eq:WV simplified}
\end{align}
where the last equality uses \(\esp W_n=0\).
The functional convergence in \citep[Lemma 3.4]{wang26central}, which already takes into account
the dependence between \(D\) and \(N\), yields
\begin{equation}\label{eq:Vn fclt}
\pp{V_n(t)}_{t\in[0,2]}
\weakto
\pp{\B_{Dt^\alpha}}_{t\in[0,2]} 
\end{equation}
in $D[0,2]$ as $n\to\infty$. 
Fix \(M>0\).  On \(\{\abs{W_n}\leq M\}\), for all sufficiently large
\(n\)
\[
\sup_{t\in[-T,T]}\abs{R_n(t)-1}
\leq
2\pp{\frac{M}{\sqrt{j_n}}+\frac{T}{j_n}}
=:\delta_{n,M}.
\]
It follows from \eqref{eq:Vn fclt} that $\sup_{t\in[-T,T]}|V_n(R_n(t))-V_n(1)|\to 0$ in probability, and hence
\begin{equation}\label{eq:le M}
\limn\sup_{t\in[-T,T]}
\abs{
\esp_N\pp{
W_n\pp{V_n(R_n(t))-V_n(1)}
\inddd{\abs{W_n}\leq M}
}}=0
\end{equation}
in probability for every fixed \(M\).

It remains to control 
\[
\Delta_{n,M}^*:=
\sup_{t\in[-T,T]}
\esp_N\pp{
\abs{W_n}
\abs{V_n(R_n(t))-V_n(1)}
\inddd{\abs{W_n}>M}
}.
\]
For fixed \(x>0\), on \(\{D\leq x\}\) and for all sufficiently large
\(n\),
\[
R_n(t)\leq\frac{2j_n}{G_{j_n}},
\qquad t\in[-T,T].
\]
Therefore, we have
\begin{align*}
\inddd{D\le x}&\Delta_{n,M}^*
 \le \esp_N\pp{\inddd{D\le x}\sup_{t\in[-T,T]}\pp{|W_n||V_n(R_n(t))-V_n(1)|}\inddd{|W_n|>M}}\\
& \le \esp_N\pp{\inddd{D\le x}|W_n|\inddd{|W_n|>M}\pp{\sup_{t\in[0,2j_n/G_{j_n}]}\frac{|\wb N(D(n/j_n)^\alpha t^\alpha)|}{(n/j_n)^{\alpha/2}}+\frac{|\wb N(D(n/j_n)^\alpha)|}{(n/j_n)^{\alpha/2}}}}\\
& \le \esp_N\pp{|W_n|\inddd{|W_n|>M}\pp{\sup_{t\in[0,2j_n/G_{j_n}]}\frac{|\wb N(x(n/j_n)^\alpha t^\alpha)|}{(n/j_n)^{\alpha/2}}+\sup_{t\in[0,x]}\frac{|\wb N((n/j_n)^\alpha t)|}{(n/j_n)^{\alpha/2}}}}.
\end{align*}
Now, taking the expectation of the upper bound above, and then computing the conditional expectation with respect to $G_{j_n}$ first, we have
\begin{align*}
\nonumber\esp&\pp{\inddd{D\leq x}\Delta_{n,M}^*}\\
&\nonumber \le \esp\pp{|W_n|\inddd{|W_n|>M}\esp\pp{\sup_{t\in[0,2j_n/G_{j_n}]}\frac{|\wb N(x(n/j_n)^\alpha t^\alpha)|}{(n/j_n)^{\alpha/2}}+\sup_{t\in[0,x]}\frac{|\wb N((n/j_n)^\alpha t)|}{(n/j_n)^{\alpha/2}}\mmid G_{j_n}}}\\
&\le\frac{2x^{1/2}}M
\esp\pp{W_n^2\pp{\pp{\frac{2j_n}{G_{j_n}}}^{\alpha/2}+1}}
\leq\frac C{M},
\nonumber%\label{eq:tail Doob}
\end{align*}
where in the second inequality we applied Doob's maximal inequality $\esp(\sup_{x\in[0,t]}|\wb N(x)|)\le 2\sqrt t$ for all $t>0$, and the constant $C$ at the end depends on $x$ but not $n$.
Consequently, using $\proba(\Delta_{n,M}^*>\epsilon)\le\proba(\Delta_{n,M}^*>\epsilon, D\le x)+\proba(D>x)$, and by taking $x$ arbitrarily large, we have
\[
\lim_{M\to\infty}\limsup_{n\to\infty}
\proba\pp{\Delta_{n,M}^*>\epsilon}=0,
\qquad \mbox{ for all }\epsilon>0.
\]
  Combining this with \eqref{eq:sigma rn combined}, \eqref{eq:WV simplified}, and \eqref{eq:le M} 
 proves
\eqref{eq:random centering} in probability.

\smallskip
\noindent
{\em We now prove \eqref{eq:random centering} under the additional assumption \eqref{eq:lower growth combined}.}  Define
\begin{align*}
q_{j}(z)&:=\proba\pp{N(z) = j} = e^{-z}\frac{z^{j}}{j!},\\
F_{n,d}(y)&:=q_{j_n}\pp{(d/y)^{1/\alpha}n},\\
H_n(d)&:=\int_0^\infty F_{n,d}(y)\pp{N(\d y)-\d y}.
\end{align*}
In this way we have
\[%\begin{equation}\label{eq:H representation combined}
\Delta_n(t)
=H_n\pp{D\pp{1+\frac{t}{j_n}}^\alpha}.
\]%\end{equation}
We shall prove that, for every \(0<a<b<\infty\),
\begin{equation}\label{eq:uniform H}
\limn\sup_{d\in[a,b]}\frac{\abs{H_n(d)}}{r_n}
=0,
\quad\mbox{almost surely}.
\end{equation}
This implies the claimed \eqref{eq:random centering}. Indeed, consider $[a,b] = [1/m, m]$ for each $m\in\N$. For all such intervals we have the above almost sure convergence. At the same time, there exists $m\in\N$ such that $[D(1+(-T)/j_n)^\alpha,D(1+T/j_n)^\alpha]\subset[1/m,m]$ and hence $\sup_{t\in[-T,T]}|\Delta_n(t)|\le \sup_{d\in[1/m,m]}|H_n(d)|$ for all $n$ large enough almost surely. 

Now we prove \eqref{eq:uniform H}. For this purpose we consider a grid $\calG_n:=n^{-4}\N\cap[a,b]$ in $[a,b]$ of mesh $n^{-4}$, and we shall show
\equh\label{eq:max over grid}
\max_{d\in\mathcal G_n}\abs{H_n(d)}=o(r_n)
\qquad\mbox{almost surely},
\eque
and
\equh\label{eq:max osci}
\max_{d\in\calG_n}\sup_{d'\in[a,b], |d-d'|\le n^{-4}}\abs{H_n(d)-H_n(d')} = o(r_n), \mbox{ almost surely.}
\eque

We start by proving \eqref{eq:max over grid}. We first show that for all $\epsilon>0$, there exists a constant $c_\epsilon>0$ such that
\[
\proba\pp{
\max_{d\in\mathcal G_n}\abs{H_n(d)}>\epsilon r_n
}
\leq Cn^4e^{-c_\epsilon\sqrt{j_n}}, n\in\N.
\]
The right-hand side is summable by
\eqref{eq:lower growth combined}.  The Borel--Cantelli Lemma then yields \eqref{eq:max over grid}.
In order to establish this maximal inequality, we first notice that 
$q_{j}(z) = \proba(N(z) = j)$ attains its maximum at $z=j$. Stirling's formula gives
\[
\sup_{y>0}F_{n,d}(y)=q_{j_n}(j_n)
\leq Cj_n^{-1/2}, \mfa n\in\N,
\]
and the constant $C$ does not depend on $d$. Therefore, 
\[%\equh\label{eq:F sup combined}
\max_{d\in[a,b]}\sup_{y>0}F_{n,d}(y)
\leq Cj_n^{-1/2}.
\]%\eque
Moreover, a straightforward calculation gives
\[%\equh
\int_0^\infty F_{n,d}(y)^2\d y
=
\alpha d n^\alpha
\frac{\Gamma(2j_n-\alpha)}
{2^{2j_n-\alpha}\Gamma(j_n+1)^2}
\leq Cr_n^2j_n^{-1/2}.
%\label{eq:F variance combined}
\]%\eque
Recall the Bernstein inequality for compensated Poisson integrals \citep[Proposition 7]{reynaud-bouret03adaptive}
\[
\proba\pp{\abs{\int f(\d N(t)-\d t)}>\eta}\le 2\exp\pp{-\frac{\eta^2}{2\int f^2\d x+(2/3)\eta\nn f_\infty}}.
\]
(Note that their inequality concerns Poisson integrals with respect to an inhomogeneous Poisson random measure of which the intensity measure is finite over $\R$. But here we can apply their results by an approximation argument to obtain the above.) Then, it follows that
for every fixed \(\epsilon>0\), one can find a constant $c_\epsilon>0$ such that 
\[%\begin{equation}\label{eq:H tail combined}
\sup_{d\in[a,b]}
\proba\pp{\abs{H_n(d)}>\epsilon r_n}
\leq2\exp\pp{- C\frac{(\epsilon r_n)^2}{r_n^2j_n^{-1/2}+\epsilon r_nj_n^{-1/2}}}\le 2\exp\pp{-c_\epsilon\sqrt{j_n}},
\]%\end{equation}
for all $n\in\N$. 
We have proved the previous maximal inequality and hence \eqref{eq:max over grid}.

Next, we prove \eqref{eq:max osci}. Note that for all $a\le z<z'\le b$, 
\begin{align}
\nonumber\abs{H_n(z)-H_n(z')}&\le \int_0^\infty \abs{F_{n,z}(y)-F_{n,z'}(y)}(N(\d y)+\d y)\\
& \le  \int_0^\infty \int_z^{z'}\abs{\frac\partial{\partial z''}F_{n,z''}(y)}\d z''\pp{N(\d y)+\d y}\nonumber\\
& \le |z-z'|\int_0^\infty \sup_{d\in[a,b]}\abs{\frac{\partial}{\partial d}F_{n,d}(y)}(N(\d y)+\d y)= L_n|z-z'|,\label{eq:Lipschitz}
\end{align}
with 
\[
L_n:=\int_0^\infty G_n(y)N(\d y)+\int_0^\infty G_n(y)\d y \qmand
G_n(y):=
\sup_{d\in[a,b]}
\abs{\frac{\partial}{\partial d}F_{n,d}(y)}.
\]
The key estimate is 
\begin{equation}\label{eq:int G_n}
\int_0^\infty G_n(y)\d y\leq Cn^\alpha j_n^{1/2-\alpha}\le Cn.
\end{equation}
The above then implies that 
$\esp L_n\le Cn$ for all $n\in\N$ and hence \(L_n\leq n^3\) eventually almost
surely (by the Markov inequality and the Borel--Cantelli lemma).  This and \eqref{eq:Lipschitz} then imply \eqref{eq:max osci}. Note that
in the last inequality in \eqref{eq:int G_n} we take a simple upper bound; this does not affect the constraint on the rate of $j_n\to\infty$ we impose. Similarly, the powers in the grid size $n^{-4}$ and $L_n\le n^3$ are taken for convenience; they can be chosen to be smaller but will not affect the constraint on the rate of $j_n\to\infty$.

It remains to prove \eqref{eq:int G_n}. Writing $z = n(d/y)^{1/\alpha}$, we have
\[
\frac{\partial}{\partial d}F_{n,d}(y)
=\frac\partial{\partial d}q_{j_n}\pp{d^{1/\alpha}ny^{-1/\alpha}}
=\frac{j_n-z}{\alpha d}q_{j_n}(z)
\]
and
\[
\abs{j_n-z}q_{j_n}(z)
\leq j_nq_{j_n}(z)+(j_n+1)q_{j_n+1}(z).
\]
We first bound 
\[
\int_0^\infty \sup_{d\in[a,b]}q_{j_n}(z)\d y = \pp{\frac n{j_n}}^\alpha \int_0^\infty \sup_{d\in[a,b]}q_{j_n}\pp{(d/v)^{1/\alpha}j_n}\d v.
\]
We focus on the integral. 
Recall that $\sup_{z\in\R_+}q_{j_n}(z) = q_{j_n}(j_n) \le C j_n^{-1/2}$, and that
$q_{j_n}(z)$ is monotone over $z<j_n$ and $z>j_n$ respectively. We have that the integration restricted to $[a,b]$ is of order $O(j_n^{-1/2})$, and moreover by monotonicity
\[
\int_0^a\sup_{d\in[a,b]}q_{j_n}\pp{(d/v)^{1/\alpha}j_n}\d v\le \int_0^a q_{j_n}\pp{(a/v)^{1/\alpha}j_n}\d v = \alpha aj_n^\alpha \int_{j_n}^\infty q_{j_n}(z)z^{-\alpha-1}\d z
\]
and similarly
\[
\int_b^\infty\sup_{d\in[a,b]}q_{j_n}\pp{(d/v)^{1/\alpha}j_n}\d v\le \alpha bj_n^\alpha \int_0^{j_n} q_{j_n}(z)z^{-\alpha-1}\d z.
\]
That is, the integral over $[0,a]\cup[b,\infty)$ is bounded by $Cj_n^\alpha \int_0^\infty q_{j_n}(z)z^{-\alpha-1}\d z = Cj_n^\alpha\Gamma(j_n-\alpha)/\Gamma(j_n+1)\le C/j_n$. Combining all the estimates we have thus proved 
\[%\begin{equation}\label{eq:pi envelope combined}
\int_0^\infty
\sup_{d\in[a,b]}
q_{j_n}\pp{j_n(d/v)^{1/\alpha}}\d v
\leq Cj_n^{-1/2}.
\]%\end{equation}
The integral with $q_{j_n}(j_n(d/v)^{1/\alpha})$ replaced by $q_{j_n+1}(j_n(d/v)^{1/\alpha})$ has an upper bound of the same order; it suffices to notice that  $\sup_{d\in[a,b]}q_{j_n+1}\spp{j_n(d/v)^{1/\alpha}}\le \sup_{d\in[a',b']}q_{j_n+1}\spp{(j_n+1)(d/v)^{1/\alpha}}$ for some appropriately chosen $a',b'$. In summary, we have shown that 
\begin{align*}
\int_0^\infty G_n(y)\d y& \le C\pp{j_n\int_0^\infty \sup_{d\in[a,b]}q_{j_n}(z)\d y+(j_n+1)\int_0^\infty \sup_{d\in[a,b]}q_{j_n+1}(z)\d y
}\\
& \le Cj_n\pp{\frac n{j_n}}^\alpha j_n^{-1/2}
\leq Cn^\alpha j_n^{1/2-\alpha},
\end{align*}
which is \eqref{eq:int G_n}.

%\color{blue}
\smallskip
%\newpage

It remains to prove the de-Poissonization.
We show that 
\[%\equh\label{eq:0 process dP}
\limn\sup_{t\in[-T,T]}
\frac{
\sabs{
\esp\pp{C_{j_n}(\floor{s_n(t)})\mmid\calP}
-\esp\spp{\wt C_{j_n}(s_n(t))\mid\calP}
}
}{\sigma_{n,j_n}}
=0
\]%\eque
almost surely, when $j_n\to\infty$ and \eqref{eq:j_n upper} holds (without assuming $j_n/(\log n)^2\to\infty$). 
But this is exactly \eqref{eq:2''} applied to an urn model with random sampling frequencies. This completes the proof.
\end{proof}

%\newpage
\section{The critical regime}\label{sec:crit}
In the critical regime, we consider 
\equh\label{eq:j_n crit}
j_n\equiv j_n(r) := \floor{ rn^{\alpha/(1+\alpha)}}
\eque for some $r>0$. In this case, we work with the Karlin model with asymptotic frequencies satisfying
\[%\equh\label{eq:p_n}
p_j \sim d_0 j^{-1/\alpha} \mmas j\to\infty
\]%\eque
for some constant $d_0\in(0,\infty)$. 

We shall establish point-process convergence results that imply the convergence to an M/M/$\infty$ queue as an immediate corollary. 
It is well known that the asymptotic frequencies of $(\alpha,\theta)$-partitions satisfy 
$P_j^\downarrow= d_0\Gamma_j^{-1/\alpha}\sim d_0 j^{-1/\alpha}$ almost surely as $j\to\infty$ for a random constant $d_0$. It turns out that the fluctuations around $d_0j^{-1/\alpha}$ are so strong that the analysis has to be dealt with as a separate case from the one when the fluctuations are negligible. 
In particular, the difference has an impact at the level of point-process convergence, where different normalizations are needed. For convergence of the normalized counting process the two cases can be unified.

We shall work under the following two sets of conditions separately.
\begin{enumerate}[(i)]
\item
($(\alpha,\theta)$-partitions) In this case we have $p_n = P_n^\downarrow$. Recall that we rewrite it as $P_n^{\downarrow} = (S_{\alpha}/\Gamma(1-\alpha))^{1/\alpha}\Gamma_n^{-1/\alpha}$ with $(\Gamma_n)_{n\in\N}$ as in \eqref{eq:S_alpha 2} and \eqref{eq:Gamma_j}, and these random variables  are all $\calP$-measurable.  We shall simply write
\equh\label{eq:p_n CRP}
p_n = d_0 \Gamma_n^{-1/\alpha}, n\in\N \qmwith d_0 :=\pp{\frac{S_\alpha}{\Gamma(1-\alpha)}}^{1/\alpha}.
\eque
Recall also that when $\theta=0$ then $(\Gamma_j)_{j\in\N}$ has the law of consecutive arrival times of a standard Poisson process, while when $\theta\ne 0$ this fact no longer holds. 
\item (General case) For the random partitions induced by a Karlin model with sampling frequencies $(p_n)_{n\in\N}$, we shall assume for $\alpha\in(0,1)$,
\equh\label{eq:p_n general}
p_n = d_0n^{-1/\alpha}(1+\rho_n) \qmwith  \limn n^{1/2}\rho_n=0.
\eque
In this case we assume $d_0$ to be a constant.
\end{enumerate}

Set
\equh\label{eq:ell_n}
\ell_n :=\floor{\pp{\frac 1{d_0}\frac{j_n}n}^{-\alpha}} =\pp{\frac r{d_0}}^{-\alpha} n^{\alpha/(1+\alpha)}+O(1) \qmand \wt\ell_n := 2\ell_n-\Gamma_{\ell_n}.
\eque
The choice of $\ell_n$ ensures that $np_{\ell_n}\sim j_n$ and $j_n\sim (r^{\alpha+1}/d_0^\alpha)\ell_n$, and the same relations hold with $\ell_n$ replaced by $\wt\ell_n$.
For the Karlin model, write  $K_{m,\ell} = \summ i1m \inddd{Y_i = \ell}$, and let
 \[
 T\topp\ell_{j}:=\min\ccbb{m\in\N: K_{m,\ell} = j}
 \] 
 denote the first time that $\ell$ is sampled $j$ times (the $\ell$-th urn has $j$ balls).
Now, the point processes of interest for the two cases are
\begin{align}%\label{eq:xi_n}
\xi_n &:=\sif\ell1 \ddelta{(\ell-\wt\ell_n)/j_n^{1/2}, j_n(T_{j_n}\topp \ell/n-1), j_n(T\topp \ell_{j_n+1}-T\topp\ell_{j_n})/n},\nonumber\\
\xi_n' &:= \sif\ell1 \ddelta{(\ell-\ell_n)/j_n^{1/2}, j_n(T_{j_n}\topp \ell/n-1), j_n(T\topp \ell_{j_n+1}-T\topp\ell_{j_n})/n}.\label{eq:xi'_n}
\end{align}
In words, we are interested in the time interval during which for each $\ell$ the urn has exactly $j_n$ balls, and for this purpose we record the starting and the duration times of  the time interval in the second and third coordinates of the point process.  In order to have a non-degenerate limit the labels and time of having $j_n$ balls are both re-scaled. The only difference between the two cases is that they  have different centerings for $\ell$ ($\wt\ell_n$ vs.~$\ell_n$). 

Set
\equh\label{eq:mu_alpha}
\sigma_\alpha^2\equiv \sigma_{\alpha,r,d_0}^2:= \pp{\alpha r^{-\alpha-1}d_0^\alpha}^2 \qmand 
\mu_\alpha(\d x) := \frac1{\sqrt{2\pi}}e^{-x^2/(2\sigma_\alpha^2)}\d x.
\eque
Our reference for point-process convergence is \citet{resnick87extreme}. We let $\calM_p(E)$ denote the space of Radon point measures on $E$. 

We first state the results for $(\alpha,\theta)$-partitions. In this case, $\sigma_{\alpha,r} = \alpha r^{-\alpha-1}S_\alpha/\Gamma(1-\alpha)$. 
 \begin{theorem}\label{thm:PP conv}
For $(p_n = P_n^\downarrow)_{n\in\N}$ from the $(\alpha,\theta)$-partitions (i.e., as in \eqref{eq:p_n CRP}), we have
\equh\label{eq:PP xi_n}
\xi_n\aswto \xi\eqd\PPP\pp{\R\times\R\times[0,\infty),\mu_\alpha(\d x)\d y e^{-z}\d z}, \eque
with respect to $\calP$ as $n\to\infty$ in $\calM_p(\R\times\R\times[0,\infty))$. 
\end{theorem}
%YZ0725
Essentially by a continuous mapping argument, we can prove a convergence of measure-valued processes and also a convergence to the $M/M/\infty$ queue (this second part is exactly Theorem \ref{thm:crit}). 
Let $\mathcal M_{p,f}(\mathbb R)$ be the space of finite point measures on $\mathbb R$, equipped with the weak topology. We shall study convergence of $\calM_{p,f}(\R)$-valued processes.  Write
\[
s_n(t) = n\pp{1+\frac t{j_n}}_+.
\]Define
\[
\mathcal Z_n(t):=\sum_{\ell\geq1}\delta_{(\ell-\wt\ell_n)/\sqrt{j_n}}\mathbf 1_{\{K_{\floor{s_n(t)},\ell}=j_n\}}=\sum_{\ell\geq1}\delta_{(\ell-\wt\ell_n)/\sqrt{j_n}}\inddd{T_{j_n}\topp\ell\leq \floor{s_n(t)}<T_{j_n+1}\topp\ell}.
\]
Representing the point process in the limit of \eqref{eq:PP xi_n} as $\xi=\sum_i\delta_{(x_i,y_i,z_i)}$, we set
\[
\mathcal Z(t):=\sum_i\delta_{x_i}\mathbf 1_{\{y_i\leq t<y_i+z_i\}}.
\]
For every $T>0$, the number of intervals $[y_i,y_i+z_i)$ intersecting $[-T,T]$ has conditional mean (given $\calP$)
$\mu_\alpha(\mathbb R)\int_0^\infty(2T+z)e^{-z}\d z=(2T+1)\mu_\alpha(\mathbb R)$.
It is therefore finite almost surely, and $\mathcal Z$ is a well-defined c\`adl\`ag $\mathcal M_{p,f}(\mathbb R)$-valued process. It is clear that $\calZ$ is stationary in $t$ (since $\xi$ is translation invariant), and for each fixed $t$  
 $\mathcal Z(t)$ is readily checked to be a Cox process on $\mathbb R$ with intensity measure
$\mu_\alpha(\d x)\int_0^\infty\int_{t-z}^t \d y e^{-z}\d z=\mu_\alpha(\d x)$,
 and it is piecewise constant between consecutive times of discontinuity. 

\begin{theorem}\label{thm:discrete-cmt}
We have
\equh\label{eq:Z_n}
\pp{\mathcal Z_n(t)}_{t\in\R}\aswto\pp{\mathcal Z(t)}_{t\in\R}
\eque
with respect to $\calP$ in $D(\mathbb R,\mathcal M_{p,f}(\mathbb R))$ with the local $J_1$ topology.
Moreover,
\equh\label{eq:M/M/infty}
\left(C_{j_n}(\floor{s_n(t)})\right)_{t\in\mathbb R}\aswto(\calC_{\alpha,r}(t))_{t\in\R}
\eque
with respect to $\calP$ in $D(\mathbb R)$ with the local $J_1$ topology.

\end{theorem}

Under \eqref{eq:p_n CRP}, both $d_0$ and $\sigma_\alpha$ are $\calP$-measurable. 
For each $t$, $\calZ(t)$ is understood as a Cox process (a Poisson point process with a random intensity measure), and each $\calC_{\alpha,r}(t)$ is understood as a mixed Poisson random variable (a Poisson random variable with a random parameter).  For example, the convergence \eqref{eq:PP xi_n} means that
\[
\limn \esp\pp{e^{-\xi_n(f)}\mmid\calP} = \esp \pp{\exp\pp{-\int_\R\int_\R\int_0^\infty \pp{1-e^{-f(x,y,z)}}\mu_\alpha(\d x)\d y e^{-z}\d z}\mmid\calP}
\] 
almost surely for all non-negative continuous functions $f$ with compact support on $\R\times\R\times[0,\infty)$.

We next state the results for urn counts of a general Karlin model. 
Again, while $(\alpha,\theta)$-partitions are a special case of the Karlin model, Theorem \ref{thm:PP conv} cannot be applied directly as mentioned already. There are two differences. First, the limiting counting process $\calC_{\alpha,r}$ is still an $M/M/\infty$ queue but the rate  is now $\sigma_{\alpha,r} = \alpha r^{-\alpha-1}d_0^\alpha$ with $d_0$ the constant as in \eqref{eq:p_n general}. The random intensity measure $\mu_\alpha$ in the limiting process $\calZ$ now also depends on the new $\sigma_{\alpha,r}$. Second, the pre-limit point process $\xi_n'$ has a different centering in the first coordinate from $\xi_n$.
We also define accordingly
\[
\mathcal Z_n'(t):=\sum_{\ell\geq1}\delta_{(\ell-\ell_n)/\sqrt{j_n}}\mathbf 1_{\{K_{\floor{s_n(t)},\ell}=j_n\}}=\sum_{\ell\geq1}\delta_{(\ell-\ell_n)/\sqrt{j_n}}\inddd{T_{j_n}\topp\ell\leq \floor{s_n(t)}<T_{j_n+1}\topp\ell}.
\]

\begin{theorem}\label{thm:PP conv general}
For $(p_n)_{n\in\N}$ under the assumption \eqref{eq:p_n general}, we have
\equh\label{eq:PP convergence}
\xi'_n\weakto \PPP\pp{\R\times\R\times[0,\infty),\mu_\alpha(\d x)\d y e^{-z}\d z}, \eque
as $n\to\infty$ in $\calM_p(\R\times\R\times[0,\infty))$. Moreover, 
\[
\pp{\calZ_n'(t)}_{t\in\R}\weakto \pp{\calZ(t)}_{t\in\R}
\]
in $D(\R,\calM_{p,f}(\R))$,  and
 \[
 \pp{C_{j_n}(\floor{s_n(t)})}_{t\in\R} \weakto (\calC_{\alpha,r}(t))_{t\in\R}
 \]
 in $D(\R)$ as $n\to\infty$.
\end{theorem}

Here we provide a heuristic overview of the proof. 
It is not hard to see that the limit of $C_{j_n}(n)$ should be a Poisson random variable (this is what \citet{banderier24phase} proved, with a random parameter). Indeed, the expected number of balls in the urn with label $\ell$ after $n$ rounds is $np_\ell$, and naturally we expect the labels of those urns with $j_n$ balls to be such that $np_\ell$ is close to $j_n$. Because of this argument we chose \eqref{eq:ell_n}. 
Also, it is clear that for every fixed $\ell$ close to $\ell_n$, the probability, say $q_{n,\ell}$, that the corresponding urn has exactly $j_n$ balls at time $n$ is negligible, and a quick calculation reveals that the summation of $q_{n,\ell}$ indexed by $\ell$ in a neighborhood of $\ell_n$ of size $j_n^{1/2}$ is of order one, and hence we are in the typical situation of a Poisson limit theorem. 

The point process $\xi_n'$ is then the natural candidate for the more refined point-process convergence supporting the Poisson limit theorem, and the proof shall follow the standard approach by Kallenberg. As usual, the Poissonization technique is applied first. That is, we first translate the question into the one for the Poissonized model, then apply Kallenberg's method to it, and at the end apply a de-Poissonization argument to complete the proof.  

However, in the case of $(\alpha,\theta)$-partitions, it turns out that because of the fluctuations of $\Gamma_j^{-1/\alpha}$ around $j^{-1/\alpha}$ the centering has to be shifted by $\ell_n-\Gamma_{\ell_n}$ so that the point process converges (otherwise it is only tight when evaluated over $[a,b]\times\R\times[0,\infty)$). This relies on a very precise estimate of  $np_{\ell_n+k}-j_n$ for $k$ in a small neighborhood of $\ell_n$ of size $\sqrt{j_n\log\log n}$. This is the key step of the proof and is provided in Lemma \ref{lem:CRP case} in Section \ref{sec:estimates}. At the same time,  a quick calculation below shows why $(p_n)_{n\in\N}$ defined in \eqref{eq:p_n CRP} violate the assumption \eqref{eq:p_n general}.  
%\newpage
\begin{remark}
Assume $\theta =0$. So now $(\Gamma_n)_{n\in\N}$ has the law of consecutive arrival times of a standard Poisson process.  For this example, we have 
\[
\limsupn \frac{\sqrt n|\rho_n|}{\sqrt{2\log\log n}} = \frac1\alpha, \mbox{ almost surely.}
\]Indeed, by Taylor's expansion we have
\[
\rho_n = \pp{\frac{\Gamma_n}n}^{-1/\alpha}-1 = -\frac1\alpha\frac{\Gamma_n-n}n+O\pp{\frac{(\Gamma_n-n)^2}{n^2}},
\]
and the claim follows from the law of iterated logarithm.
We have thus proved the claimed property with $\theta=0$. This is an almost sure result, and hence it holds for all $\theta>-\alpha$ by part (ii) of Lemma \ref{lem:P_j}.
\end{remark}

The rest of the section is devoted to proofs. We first establish key estimates on $np_{\ell_n+k}-j_n$ in Section \ref{sec:estimates} in Lemma \ref{lem:CRP case}, where the crucial difference of the two cases is summarized.  Moreover, the point-process convergence essentially reduces to this local expansion. Then we prove Theorem \ref{thm:PP conv} by the standard Poissonization argument in Sections \ref{sec:Poissonized} and \ref{sec:de-Poissonization}. The proof of Theorem \ref{thm:discrete-cmt} is provided in Section \ref{sec:CMT}. The proof of Theorem \ref{thm:PP conv general} follows the same approach and is sketched in Section \ref{sec:general}. 

%\newpage
\subsection{Key estimates}\label{sec:estimates}
Given $j_n$ as in \eqref{eq:j_n crit}, we have explained that $\ell_n$ is chosen so that $np_{\ell_n+k}\sim j_n$. Here we provide a more precise estimate on $np_{\ell_n+k}-j_n$. 
\begin{lemma}\label{lem:CRP case}
Suppose $\alpha\in(0,1)$. \begin{enumerate}[(i)]
\item 
Suppose $p_\ell = d_0\Gamma_\ell^{-1/\alpha}, \ell\in\N$ as in \eqref{eq:p_n CRP}. 
For every $K>0$,
\equh\label{eq:diff CRP}
\max_{|k|\le K\sqrt{j_n\log\log n}}\abs{np_{\ell_n+k}-j_n   +\frac{j_n(\Gamma_{\ell_n}-\ell_n+k)}{\alpha\ell_n}}= O\pp{j_n^{1/4}(\log n)^{1/2}\pp{\log\log n}^{1/4}},
\eque
almost surely. 
\item Suppose $p_n = d_0n^{-1/\alpha}(1+\rho_n)$. Then for all $K>0$,
\equh\label{eq:diff general}
\max_{|k|\le K j_n^{1/2}}\abs{np_{\ell_n+k}-j_n + \frac{j_nk}{\alpha\ell_n}} = O\pp{1}+O\pp{j_n\max_{|k|\le Kj_n^{1/2}}|\rho_{\ell_n+k}|}.
\eque
\end{enumerate}
\end{lemma}
\begin{proof}
We first prove (i). This is an almost sure estimate, and hence by part (ii) of Lemma \ref{lem:P_j} it suffices to prove the claim for $\theta=0$, which we assume throughout the proof. In this case, $(\Gamma_n)_{n\in\N}$ has the law of consecutive arrival times of a standard Poisson process. All the estimates below are in the almost sure sense. 

Throughout, we fix $K>0$. The constant $C$ below does not depend on $n$ and may change from line to line, and may be random (when it depends on $d_0$).  We write $c_n = O(d_n)$ if $\limsupn c_n/d_n <\infty$, and we write $c_n\sim d_n$ as $n\to\infty$ if $\limn c_n/d_n = 1$. 
For expressions depending on both $n$ and $k$, we write \[
   c_{n,k} = \what O(d_{n,k})
   \] if
\[
\limsupn\max_{|k|\le K\sqrt{j_n\log\log n}}\abs{\frac{c_{n,k}}{d_{n,k}}}<\infty. 
\]
We also recall that 
 $\limsupn|\Gamma_n-n|/\sqrt{2n\log\log n} = 1$ almost surely. 

We start by writing
\begin{align}
np_{\ell_n+k}-j_n & = d_0n\pp{\Gamma^{-1/\alpha}_{\ell_n} -\ell_{n}^{-1/\alpha}}\nonumber \\
 & \quad +d_0n\pp{\Gamma^{-1/\alpha}_{\ell_n+k} - \Gamma_{\ell_n}^{-1/\alpha}}+d_0n\pp{\ell_n^{-1/\alpha}-(\ell_n+k)^{-1/\alpha}}\nonumber\\
& \quad+ d_0n(\ell_n+k)^{-1/\alpha}-j_n.\label{eq:4 terms}
\end{align}
For the first term on the right-hand side of \eqref{eq:4 terms}, write
\[
\Gamma_{\ell_n}^{-1/\alpha} - \ell_n^{-1/\alpha} = \ell_n^{-1/\alpha}\pp{\pp{\frac{\Gamma_{\ell_n}}{\ell_n}}^{-1/\alpha}-1}.
\]
Assume $n$ is large enough that $\Gamma_n/n\in[1-\epsilon,1+\epsilon]$ and $\log\log n>0$. Then,  applying Taylor's expansion, we have
\begin{align}
d_0n\pp{\Gamma_{\ell_n}^{-1/\alpha}-\ell_n^{-1/\alpha}} & = d_0n\ell_n^{-1/\alpha}\pp{-\frac1\alpha \pp{\frac{\Gamma_{\ell_n}}{{\ell_n}}-1} + O\pp{\pp{\frac{\Gamma_{\ell_n}}{\ell_n}-1}^2}}\nonumber\\
& = -\frac{j_n}\alpha\frac{\Gamma_{\ell_n}-\ell_n}{\ell_n}\frac{d_0n}{j_n\ell_n^{1/\alpha}}+ O\pp{\frac{(\Gamma_{\ell_n}-\ell_n)^2}{\ell_n}}\nonumber\\
& = -\frac{j_n(\Gamma_{\ell_n}-\ell_n)}{\alpha\ell_n}\pp{1+O(n^{-\alpha/(1+\alpha)})}+O\pp{\log\log n}\nonumber\\
& = -\frac{j_n(\Gamma_{\ell_n}-\ell_n)}{\alpha\ell_n} + O(\log\log n).\label{eq:1st}
\end{align}
The term $O(\log\log n)$ comes from the application of the law of iterated logarithm.

For the last term on the right-hand side of \eqref{eq:4 terms}, 
we compute
\begin{align}
d_0n{(\ell_n+k)^{-1/\alpha}}-j_n& 
 = 
j_n\pp{\frac {d_0n}{j_n\ell_n^{1/\alpha}}\pp{\frac{\ell_n}{\ell_n+k}}^{1/\alpha}-1}\nonumber\\
& = 
j_n\pp{\frac n{n+O(n^{1/(1+\alpha)})}\pp{1-\frac k{\ell_n+k}}^{1/\alpha}-1}\nonumber\\
& = j_n\pp{- \frac k{\alpha(\ell_n+k)}+\what O\pp{\frac{k^2}{(\ell_n+k)^2}}+O(n^{-\alpha/(1+\alpha)})}  \nonumber\\
&= -\frac{j_nk}{\alpha\ell_n}+O(1)+\what O\pp{\frac{k^2}{\ell_n}} = -\frac{j_nk}{\alpha\ell_n}+\what O(\log\log n). \label{eq:4th}
 \end{align}
The leading terms in \eqref{eq:1st} and \eqref{eq:4th} together correspond to the approximating term on the left-hand side of  \eqref{eq:diff CRP}.

We next deal with the second and the third terms on the right-hand side of \eqref{eq:4 terms} together. Write
\[
\Gamma_{\ell_n+k}^{-1/\alpha}-\Gamma_{\ell_n}^{-1/\alpha}= \Gamma_{\ell_n}^{-1/\alpha} \pp{\pp{1+\frac{\Gamma_{\ell_n+k}-\Gamma_{\ell_n}}{\Gamma_{\ell_n}}}^{-1/\alpha}-1}.
\]
Set $k_n:=\sfloor{K\sqrt{j_n\log\log n}}$. We shall need the following uniform estimate: 
\equh\label{eq:what O(k)}
\max_{|k|\le k_n}\abs{\Gamma_{\ell_n+k}-\Gamma_{\ell_n}-k}=o(k_n), \mbox{ almost surely.}
\eque
Recall that $\ell_n$ depends on $d_0$ which in turn depends on $(\Gamma_n)_{n\in\N}$. We shall need the following stronger estimate:
\equh\label{eq:LIL}
\max_{|k|\le k_n}\abs{\Gamma_{\ell_n+k}-\Gamma_{\ell_n}-k}\le (1+\epsilon)\sqrt{2 k_n\log n}, \mbox{ for all $n$ large enough,}
\eque
almost surely.
(For negative $k$ the estimate is similar.)
Here, we rely on a result by 
\citet[Theorem 3.1.1]{csorgo81strong}; see \citet[Theorem A]{shao89problem} for a more accessible citation and a result on the sharpness of the assumption. Namely, since $k_n$ is of polynomial order, its assumptions are satisfied, and hence
\[
\limsupn\max_{m=0,\dots,n}\max_{k=1,\dots,k_n}\frac{|\Gamma_{m+k}-\Gamma_m-k|}{\sqrt{2k_n(\log(n/k_n)+\log\log n)}} = 1, \mbox{ almost surely.}
\]
So the order of $\max_{k=1,\dots,k_n}$ in \eqref{eq:LIL} follows. Since $k_n = o(\ell_n)$, we also notice $\max_{k=-k_n,\dots,-1}|\Gamma_{\ell_n+k}-\Gamma_{\ell_n}-k| \le\max_{m=0,\dots,\ell_n} \max_{k=1,\dots,k_n}|\Gamma_{m+k}-\Gamma_{m}-k|$ and apply the above. 
(The upper bound in \eqref{eq:LIL} may not be sharp but is good enough for our purposes later.)

Now, applying Taylor's expansion again we have
\begin{align}
d_0n\pp{\Gamma_{\ell_n+k}^{-1/\alpha}-\Gamma_{\ell_n}^{-1/\alpha}}  &
=d_0n\Gamma_{\ell_n}^{-1/\alpha}\pp{\pp{\frac{\Gamma_{\ell_n+k}}{\Gamma_{\ell_n}}}^{-1/\alpha}-1}= -\frac{d_0n}{\alpha}\Gamma_{\ell_n}^{-1/\alpha}\pp{\frac{\Gamma_{\ell_n+k}}{\Gamma_{\ell_n}}-1} + \omega_{n,k}\nonumber\\
& = -\frac{d_0n\ell_n^{-1/\alpha}(\Gamma_{\ell_n+k}-\Gamma_{\ell_n})}{\alpha\ell_n}\pp{\frac{\Gamma_{\ell_n}}{\ell_n}}^{-1/\alpha-1}+  \omega_{n,k}.\label{eq:2 terms}
\end{align}
For the first term in \eqref{eq:2 terms}, we have
\begin{align*}
-\frac{d_0n\ell_n^{-1/\alpha}(\Gamma_{\ell_n+k}-\Gamma_{\ell_n})}{\alpha\ell_n}\pp{\frac{\Gamma_{\ell_n}}{\ell_n}}^{-1/\alpha-1} &= -\frac{d_0n\ell_n^{-1/\alpha}(\Gamma_{\ell_n+k}-\Gamma_{\ell_n})}{\alpha\ell_n}\pp{1+O\pp{\frac{\Gamma_{\ell_n}-\ell_n}{\ell_n}}}\\
& = -\frac{d_0n\ell_n^{-1/\alpha}(\Gamma_{\ell_n+k}-\Gamma_{\ell_n})}{\alpha\ell_n}+\what O\pp{\log\log n},
\end{align*}
where in the last step we applied \eqref{eq:what O(k)} and the law of iterated logarithm. For the term $\omega_{n,k}$ in \eqref{eq:2 terms}, we have
\[
\omega_{n,k}=\what O\pp{n\Gamma_{\ell_n}^{-1/\alpha}{\pp{\frac{\Gamma_{\ell_n+k}}{\Gamma_{\ell_n}}-1}}^2} = \what O\pp{\frac1{j_n} \pp{\Gamma_{\ell_n+k}-\Gamma_{\ell_n}}^2}=\what O( \log\log n).
\]
So, \eqref{eq:2 terms} now yields
\[
d_0n\pp{\Gamma_{\ell_n+k}^{-1/\alpha}-\Gamma_{\ell_n}^{-1/\alpha}} = -\frac{d_0n\ell_n^{-1/\alpha}(\Gamma_{\ell_n+k}-\Gamma_{\ell_n})}{\alpha\ell_n}+\what O\pp{\log\log n}.
\]
We also have
\begin{align*}
d_0n\pp{\ell_n^{-1/\alpha}-(\ell_n+k)^{-1/\alpha}} & = d_0n\ell_n^{-1/\alpha}\pp{1-\pp{1+\frac{k}{\ell_n}}^{-1/\alpha} }=  \frac{d_0n}{\alpha}\ell_n^{-1/\alpha}\frac{k}{\ell_n} + \what O\pp{\frac{k^2}{\ell_n}}.
\end{align*}
Combining the above two expressions, we have
\begin{multline}\label{eq:2nd 3rd}
d_0n \pp{\Gamma_{\ell_n+k}^{-1/\alpha} - \Gamma_{\ell_n}^{-1/\alpha}} +d_0n\pp{\ell_n^{-1/\alpha}-(\ell_n+k)^{-1/\alpha}} \\
 = -\frac{d_0n\ell_n^{-1/\alpha}(\Gamma_{\ell_n+k}-\Gamma_{\ell_n}-k)}{\alpha\ell_n} + \what O(\log\log n).
\end{multline}
We continue to estimate the first term in \eqref{eq:2nd 3rd}, which is of the same order as $\Gamma_{\ell_n+k}-\Gamma_{\ell_n}-k$. 
The bound in \eqref{eq:LIL} gives the leading error term in \eqref{eq:diff CRP}. 
We have proved \eqref{eq:diff CRP}.

 Next, we prove (ii). Indeed, using $p_n = d_0n^{-1/\alpha}(1+\rho_n)$, we write
\begin{align*}
np_{\ell_n+k}-j_n &=np_{\ell_n+k}- d_0n(\ell_n+k)^{-1/\alpha} + d_0n(\ell_n+k)^{-1/\alpha}-j_n \\
&= d_0n(\ell_n+k)^{-1/\alpha}-j_n+ d_0n(\ell_n+k)^{-1/\alpha}\rho_{\ell_n+k}.
\end{align*}
This time we write $c_{n,k} = \wt O(d_{n,k})$ if  
\[
\limsupn \max_{k:|k|\le Kj_n^{1/2}}\abs{\frac{c_{n,k}}{d_{n,k}}} <\infty.
\]
Here, the restriction of $k$ is over a slightly smaller interval when compared to the restriction in the notation $\wt O$. We only need the counterpart of  \eqref{eq:4th}, which now becomes
\[
d_0n\pp{\ell_n+k}^{-1/\alpha}-j_n =  - \frac{j_nk}{\alpha\ell_n}+\wt O(1).
\]
We have proved \eqref{eq:diff general}. 
This completes the proof.
\end{proof}

\subsection{Convergence of the Poissonized model}\label{sec:Poissonized}
In Sections \ref{sec:Poissonized} and \ref{sec:de-Poissonization} we prove Theorem \ref{thm:PP conv}, which concerns $(\alpha,\theta)$-partitions.
We first investigate the Poissonized model. Recall that the sampling frequencies $(p_\ell)_{\ell\in\N}\equiv (P_\ell^\downarrow)_{\ell\in\N}$ follow the Poisson--Dirichlet distribution with parameter $(\alpha,\theta)$, and $(\calN_\ell)_{\ell\in\N}$ are conditionally independent Poisson process with respective rates $p_\ell$.  Let $\tau\topp \ell_j$ denote the $j$-th arrive time of $\calN_\ell$. Consider
\[
\wt\xi_n :=\sif\ell1 \ddelta{\pp{\ell-\wt\ell_n}/j_n^{1/2}, j_n\pp{\tau_{j_n}\topp \ell/n-1}, j_n\pp{\tau\topp \ell_{j_n+1}-\tau\topp\ell_{j_n}}/n}.
\]
Recall $\mu_\alpha$ in \eqref{eq:mu_alpha}. This is a Gaussian measure with a $\calP$-measurable random variance parameter $\sigma_\alpha^2$.  We shall show the following. 
\begin{proposition}\label{prop:PP conv}
For the Poissonized model corresponding to $(\alpha,\theta)$-partitions, 
\[
 \wt\xi_n\aswto \PPP\pp{\R\times\R\times[0,\infty),\mu_\alpha(\d x)\d y e^{-z}\d z},
 \]
 with respect to $\calP$ as $n\to\infty$.
 \end{proposition}

 \begin{proof}
Write $\esp_\calP(\cdot)\equiv \esp(\cdot\mid\calP)$ and $\proba_\calP(\cdot) \equiv \proba(\cdot\mid\calP)$. We proceed by applying Kallenberg's method \citep[Proposition 3.22]{resnick87extreme}, which says that to prove the desired convergence above it suffices to establish

\begin{align}
\limn \esp_\calP\pp{\wt\xi_n((a,b]\times(v,w]\times[0,z])}& = \mu_\alpha((a,b])(w-v)(1-e^{-z}),\label{eq:K1'}\\
\limn \proba_\calP\pp{\wt\xi_n((a,b]\times(v,w]\times(z,z']) = 0}& = \exp\pp{-\mu_\alpha((a,b])(w-v)(e^{-z}-e^{-z'})},\label{eq:K2''}
\end{align}
for all $-\infty<a<b<\infty, -\infty<v<w<\infty$, and $0< z<z'<\infty$, almost surely. In fact, to apply Kallenberg's method we need to show, with $\nu(\d x\d y\d z) = \mu_\alpha(\d x)\d y e^{-z}\d z$, that (i) $\limn\esp_\calP(\wt\xi_n(E)) = \nu(E)$ and (ii) $\limn\proba_\calP(\wt\xi_n(E) = 0) = e^{-\nu(E)}$ where $E=\bigcup_{r=1}^dE_r$ is a finite disjoint union of rectangular sets $E_r = (a_r,b_r]\times(v_r,w_r]\times (z_r,z'_r], r=1,\dots,d$ (strictly speaking, when $z_r = 0$ the last interval is taken as $[0,z_r']$, and the calculations remain the same). The fact that \eqref{eq:K1'} implies (i) follows from linearity of the expectation. The claim \eqref{eq:K2''} is a special case of (ii). Indeed, with $E = \bigcup_{r=1}^d E_r$ as before, and writing
\[
p_{n,\ell}(E) := \summ r1d \inddd{(\ell-\wt\ell_n)/\sqrt{j_n}\in(a_r,b_r]}\proba_\calP\pp{\tau_{j_n}\topp\ell\in(s_n(v_r),s_n(w_r)],\tau_{j_n+1}\topp\ell-\tau_{j_n}\topp\ell\in\bigg(\frac{nz_r}{j_n},\frac{nz_r'}{j_n}\bigg]},
\]
it suffices to show that 
\[
\proba_\calP\pp{\wt\xi_n(E)=0} = \prod_{\ell=1}^\infty \pp{1-p_{n,\ell}(E)}\to e^{-\nu(E)},
\]
and this last step essentially is the same as in the proof of \eqref{eq:K2''} treated below; in particular, to show the above we need the facts that $\limn\sif\ell1 p_{n,\ell}(E_r)= \nu(E_r)$ and $\limn\sup_{\ell\in\N}p_{n,\ell}(E_r)=0$ for all $r=1,\dots,d$ established below. (Strictly speaking, we shall prove the almost sure statement for all rational end points. And hence over a single probability-one event the above holds for all rational end points, and then by an extension argument the above holds for all real end points.)

We first establish \eqref{eq:K1'}. 
Set $\wb\Gamma_n := \Gamma_n-n$ and 
\[
\wt a_n:=\floor{aj_n^{1/2}-\wb\Gamma_{\ell_n}} \qmand \wt b_n:=\floor{bj_n^{1/2}-\wb\Gamma_{\ell_n}}.
\]
So $\wt\ell_n = \ell_n-\wb\Gamma_{\ell_n}$. 
Recall also $s_n(t) = n(1+t/j_n)_+$. 
We have
\equh\label{eq:mean wt xi_n}
 \esp_\calP \wt\xi_n((a,b]\times(v,w]\times[0,z]) 
 \sim \sum_{k=\wt a_n}^{\wt b_n}\proba_\calP\pp{\tau\topp{\ell_n+k}_{j_n}\in(s_n(v),s_n(w)]}\proba_\calP  \pp{\tau\topp{\ell_n+k}_1\le \frac {nz}{j_n}}.
\eque
Indeed, the two expressions may differ by at most two terms near the end points of the summation $k = \wt a_n$ and $k=\wt b_n$. The fact that these two terms are negligible in the limit can be read from the analysis below, and hence the `$\sim$' relation follows.
 
 Now we examine closely the right-hand side of \eqref{eq:mean wt xi_n}. We need a few estimates in the almost sure sense. First, by the law of iterated logarithm, we know that
$|\Gamma_n-n|\le 2\sqrt{2n\log\log n}$ for all $n$ large enough, almost surely. It thus follows that almost surely,
\[%\equh\label{eq:wt ab}
\sabs{\wt a_n},\sabs{\wt b_n}\le K\sqrt{j_n\log\log n} \equiv k_n \mbox{ for all $n$ large enough}. 
\]%\eque
(The constant $K$ here is $\calP$-measurable.)

By \eqref{eq:diff CRP}, we readily check that
$\proba_\calP\spp{\tau\topp{\ell_n+k}_1> {nz}/{j_n}} ={e^{-p_{\ell_n+k}nz/j_n}} \to e^{-z}$ as $n\to\infty$, and the convergence is in fact uniform for $k$ in the range of summation. That is,
\equh\label{eq:factor z}
\limn \max_{|k|\le 2k_n}\abs{\frac{\proba_\calP(\tau_1\topp{\ell_n+k}>nz/j_n)}{e^{-z}}-1} = 0, \mbox{ almost surely.}
\eque
%\newpage
We also write
\begin{align}
\proba_\calP\pp{\tau\topp{\ell_n+k}_{j_n}\in(s_n(v),s_n(w)]} & = \frac1{\Gamma(j_n)}\int_{np_{\ell_n+k}(1+v/j_n)}^{np_{\ell_n+k}(1+w/j_n)} 
 y^{j_n-1}e^{-y}\d y \nonumber\\
 &= \frac1{\Gamma(j_n)}\int_{vnp_{\ell_n+k}/j_n}^{wnp_{\ell_n+k}/j_n}(np_{\ell_n+k}+y)^{j_n-1}e^{-np_{\ell_n+k}-y}\d y.\nonumber%\label{eq:tau_jn}
 \end{align}
So we shall need a uniform control for $y$ over $[v-\epsilon,w+\epsilon]$ for some small $\epsilon>0$ (because of $np_{\ell_n+k}/j_n\to 1$ as $n\to\infty$), and also for $k$ over $\{\wt a_n,\dots,\wt b_n\}$. For this purpose,
we write $c_n(y,k)\sim_{y,k}d_n(y,k)$ 
if 
\[
\limn \sup_{y\in[v-\epsilon,w+\epsilon]}\max_{k=\wt a_n,\dots,\wt b_n}\abs{\frac{c_n(y,k)}{d_n(y,k)}-1} = 0.
\] Introduce also 
\[
y_n\equiv y_n(y,k) = np_{\ell_n+k}+y.
\]
Now, the goal becomes to prove
\equh
 \frac1{\Gamma(j_n)}\int_{vnp_{\ell_n+k}/j_n}^{wnp_{\ell_n+k}/j_n}y_n^{j_n-1}e^{-y_n}\d y
\sim_{y,k} \frac1{\sqrt{2\pi j_n}}\int_v^w\exp\pp{-\frac{(\wb\Gamma_{\ell_n}+k-\sigma_\alpha y)^2}{2j_n\sigma_\alpha^2}}\d y.
\label{eq:tau_jn asymp}
\eque
It follows from \eqref{eq:diff CRP} that 
\[
\sup_{y\in[v-\epsilon,w+\epsilon]}\max_{|k|\le k_n}\abs{\frac{y_n(y,k)}{j_n}-1} = O\pp{\pp{\frac{\log\log j_n}{j_n}}^{1/2}}.
\]
In particular, $y_n/j_n\sim_{y,k}1$.
 Then, we have
\begin{align}
\frac1{\Gamma(j_n)}y_n^{j_n-1}e^{-y_n}&\sim_{y,k} \frac{j_n}{y_n}\frac1{\sqrt{2\pi j_n}}\pp{\frac{y_n}{j_n}}^{j_n}e^{-(y_n-j_n)} \nonumber\\
&\sim_{y,k} \frac1{\sqrt{2\pi j_n}}e^{j_n \log\pp{1+\frac{y_n-j_n}{j_n}} - (y_n-j_n)}\sim_{y,k}\frac1{\sqrt{2\pi j_n}}e^{-(y_n-j_n)^2/(2j_n)}.\label{eq:plug in}
 \end{align}
 Now, again in view of \eqref{eq:diff CRP}, we have
\[
 \frac{(y_n-j_n)^2}{2j_n}  = \frac{(np_{\ell_n+k}-j_n+y)^2}{2j_n} = \frac1{2j_n}{\displaystyle\pp{y-j_n\frac{\wb\Gamma_{\ell_n}+k}{\alpha \ell_n}+\omega_{n,k}}^2},
\]
 with $\omega_{n,k} = \what O\spp{j_n^{1/4}(\log n)^{1/2}(\log\log j_n)^{1/4}}$. Cleaning up, we write
 \begin{align}
\frac{(y_n-j_n)^2}{j_n}& = \frac1{j_n}\pp{ j_n \frac{\alpha y\ell_n/j_n-(\wb\Gamma_{\ell_n}+k)}{\alpha \ell_n}+\omega_{n,k}}^2\nonumber\\
 & = \frac{j_n}{\alpha^2\ell_n^2}
 \pp{\frac{\alpha y\ell_n}{j_n} - (\wb\Gamma_{\ell_n}+k)}^2 + \frac{2\omega_{n,k}}{\alpha\ell_n}\pp{\frac{\alpha y\ell_n}{j_n}-(\wb \Gamma_{\ell_n}+k)} + \frac{\omega_{n,k}^2}{j_n}\nonumber\\
 & = \frac1{(\sigma_\alpha^2+o(1))j_n}\pp{\wb\Gamma_{\ell_n}+k-(\sigma_\alpha+o(1))y}^2 + \what O\pp{\omega_{n,k}(\log\log n/j_n)^{1/2}}\nonumber\\
 & = \frac1{\sigma_\alpha^2j_n}\pp{\wb\Gamma_{\ell_n}+k-\sigma_\alpha y}^2 +\what O\pp{\omega_{n,k}(\log\log n/j_n)^{1/2}}.\label{eq:square}
 \end{align}
Plugging the above into \eqref{eq:plug in}, we have shown
\equh\label{eq:Poisson pmf}
\frac1{\Gamma(j_n)}y_n^{j_n-1}e^{-y_n}\sim_{y,k}\frac1{\sqrt{2\pi j_n}}\exp\pp{-\frac1{2j_n\sigma_\alpha^2}\pp{\wb \Gamma_{\ell_n}+k-\sigma_\alpha y}^2},
\eque
and hence we have established \eqref{eq:tau_jn asymp}.

Now, applying \eqref{eq:factor z} and \eqref{eq:tau_jn asymp} to \eqref{eq:mean wt xi_n}, we have
\begin{align*}
 \esp_\calP &\wt\xi_n((a,b]\times(v,w]\times[0,z]) \nonumber\\
&\sim \frac1{\sqrt{2\pi j_n}}\int_{\wt a_n}^{\wt b_n+1}\int_v^w\exp\pp{-\frac{(\wb\Gamma_{\ell_n}+x-\sigma_\alpha y)^2}{2j_n\sigma_\alpha^2}}\d y\d x(1-e^{-z})\nonumber\\
& \sim\frac1{\sqrt{2\pi}}\int_{(\wt a_n+\wb \Gamma_{\ell_n})/j_n^{1/2}}^{(\wt b_n+\wb \Gamma_{\ell_n}+1)/j_n^{1/2}}\int_v^w\exp\pp{-\frac{(xj_n^{1/2}-\sigma_\alpha y)^2}{2j_n\sigma_\alpha^2}}\d y\d x(1-e^{-z})\nonumber\\
& \to \frac1{\sqrt{2\pi}}\int_a^b\int_v^w e^{-x^2/(2\sigma_\alpha^2)}\d y\d x(1-e^{-z}) = \mu_\alpha((a,b])(w-v)(1-e^{-z}),
\end{align*}
almost surely. We have obtained \eqref{eq:K1'}.

It remains to show \eqref{eq:K2''}. We first rewrite 
\begin{multline*}
 \proba_\calP \pp{ \wt\xi_n((a,b]\times(v,w]\times(z,z'])=0}\\
  \sim \prod_{k=\wt a_n}^{\wt b_n}\proba_\calP\pp{\tau\topp{\ell_n+k}_{j_n}\notin(s_n(v),s_n(w)], \mbox{ or } \tau\topp{\ell_n+k}_{j_n+1}-\tau\topp{\ell_n+k}_{j_n}\notin\bigg( \frac {nz}{j_n},\frac{nz'}{j_n}\bigg]} =: \prod_{k=\wt a_n}^{\wt b_n}q_{n,k}(v,w),
\end{multline*}
and furthermore 
\begin{align*}
q_{n,k}(v,w)&=\proba_\calP\pp{\tau\topp{\ell_n+k}_{j_n}\notin(s_n(v),s_n(w)]}\\
& \quad +\proba_\calP\pp{\tau\topp{\ell_n+k}_{j_n}\in(s_n(v),s_n(w)]}\proba_\calP\pp{\tau\topp{\ell_n+k}_1\notin\bigg( \frac {nz}{j_n},\frac{nz'}{j_n}\bigg]}\\
& =   1-\proba_\calP\pp{\tau\topp{\ell_n+k}_{j_n}\in(s_n(v),s_n(w)]}\proba_\calP\pp{\tau\topp{\ell_n+k}_1\in\bigg( \frac {nz}{j_n},\frac{nz'}{j_n}\bigg]}.
\end{align*}
Using $\log(1-x)\sim -x$ as $x\to 0$ and an analysis similar to the one above, we have this time
\begin{align*}
\prod_{k=\wt a_n}^{\wt b_n}q_{n,k}(v,w)
&\sim\exp\pp{-\sum_{k=\wt a_n}^{\wt b_n}\proba_\calP\pp{\tau\topp{\ell_n+k}_{j_n}\in(s_n(v),s_n(w)]}\proba_\calP\pp{\tau\topp{\ell_n+k}_1\in\bigg( \frac {nz}{j_n},\frac{nz'}{j_n}\bigg]}}\\
&\to \exp \pp{-\mu_\alpha((a,b])(w-v)(e^{-z}-e^{-z'})},
\end{align*}
as $n\to\infty$. 
We have proved \eqref{eq:K2''} and completed the proof. 
\end{proof}
\subsection{De-Poissonization}\label{sec:de-Poissonization}

Now we prove Theorem \ref{thm:PP conv}. We translate the results on $\wt\xi_n$ based on the Poissonized model  in Proposition \ref{prop:PP conv} to $\xi_n$ (based on the original Karlin model). The two models can be naturally coupled and the difference is asymptotically negligible in a sense to be made precise. This step is usually referred to as the de-Poissonization. 
\begin{proof}[Proof of Theorem \ref{thm:PP conv}]
We recall some notation. Let $(\calN_\ell)_{\ell\in\N}$ be the Poisson processes in the Poissonization, $\calN(t) = \sif\ell1\calN_\ell(t)$. Here $(\calN_\ell)_{\ell\in\N}$ are conditionally independent Poisson processes (in fact, Cox processes) with respective random parameters $(p_\ell = d_0\Gamma_\ell^{-1/\alpha})_{\ell\in\N}$, given $\calP$, and $\calN$ is always a standard Poisson process.   Introduce
$\what\tau(n) = \inf\{t>0:\calN(t) = n\}$. 
Recall that by coupling, we can assume
\[
\pp{K_{n,\ell}}_{\ell\in\N} = \pp{\wt \calN_\ell(\what\tau(n))}_{\ell\in\N}
\]
almost surely. Recall $T_j\topp\ell = \min\{m\in\N:K_{m,\ell} = j\}$.  By coupling we now have $\calN(\tau_j\topp\ell) = T_j\topp\ell$. 
Therefore, now we can compare
\begin{align*}
\wt\xi_n &=\sif\ell1 \ddelta{\pp{\ell-\wt\ell_n}/j_n^{1/2}, j_n\pp{\tau_{j_n}\topp \ell/n-1}, j_n\pp{\tau\topp \ell_{j_n+1}-\tau\topp\ell_{j_n}}/n},\\
\xi_n &=\sif\ell1 \ddelta{\pp{\ell-\wt\ell_n}/j_n^{1/2}, j_n\pp{T_{j_n}\topp \ell/n-1}, j_n\pp{T\topp \ell_{j_n+1}-T\topp\ell_{j_n}}/n}\\
& = \sif\ell1 \ddelta{\pp{\ell-\wt\ell_n}/j_n^{1/2}, j_n\pp{\calN\pp{\tau_{j_n}\topp\ell}/n-1}, j_n\pp{\calN\pp{\tau_{j_n+1}\topp\ell}-\calN\pp{\tau_{j_n}\topp\ell}}/n}.
\end{align*}

Morally, replacing $\tau_{j_n}\topp\ell$ in $\wt\xi_n$ by $\calN(\tau_{j_n}\topp\ell)$ to obtain $\xi_n$ does not affect the limiting point process. In practice, we establish \eqref{eq:K1'} and \eqref{eq:K2''} with $\wt\xi_n$ replaced by $\xi_n$. We first compare $\esp\wt\xi_n(E)$ and $\esp\xi_n(E)$ with $E = [a,b]\times[v,w]\times[0,z]$.  Recall that
\[
\esp_\calP\wt\xi_n(E) \sim \sum_{k=\wt a_n}^{\wt b_n}\proba_\calP\pp{\tau_{j_n}\topp{\ell_n+k}\in(s_n(v),s_n(w)],\tau_{j_n+1}\topp{\ell_n+k}-\tau_{j_n}\topp{\ell_n+k}\in\bb{0,\frac{nz}{j_n}}},
\] 
and the corresponding expression for $\esp\xi_n(E)$ is obtained from  the right-hand side above by replacing $\tau_j\topp\ell$ by $\calN(\tau_j\topp\ell)$. Consider the events
\[
\Omega_n:=\ccbb{\sup_{t\in[0,2n]}\abs{\calN(t)-t}\le\sqrt{n\log n}} \qmand \wt\Omega_{n,\ell} := \ccbb{\tau_{j_n+1}\topp\ell\le 2n}.
\]
Now fix $\epsilon\in(0,\min\{(w-v)/2,z/2\})$. Notice that since $n/j_n\gg \sqrt{n\log n}$,  there exists a deterministic $n_0$ such that for all  $n>n_0$, under $\Omega_n\cap\wt\Omega_{n,\ell}$ we have $|\calN(\tau_{k}\topp \ell) - \tau_{k}\topp\ell|\le \epsilon n/j_n$ for $k=j_n,j_n+1$, and hence
\begin{multline*}
\ccbb{\tau_{j_n}\topp\ell\in(s_n(v+\epsilon),s_n(w-\epsilon)],\tau_{j_n+1}\topp\ell - \tau_{j_n}\topp\ell\le\frac{n(z-2\epsilon)}{j_n}}\cap\Omega_n\cap\wt\Omega_{n,\ell}
\\
\subset \ccbb{\calN(\tau_{j_n}\topp\ell)\in(s_n(v),s_n(w)], \calN(\tau_{j_n+1}\topp\ell) - \calN(\tau_{j_n}\topp\ell)\le\frac{nz}{j_n}}\cap \Omega_n\cap\wt\Omega_{n,\ell}\\
\subset\ccbb{\tau_{j_n}\topp\ell\in(s_n(v-\epsilon),s_n(w+\epsilon)],\tau_{j_n+1}\topp\ell - \tau_{j_n}\topp\ell\le\frac{n(z+2\epsilon)}{j_n}}.
\end{multline*}
Therefore, with
\begin{align*}
E_+(\epsilon) &:= (a,b]\times (v-\epsilon,w+\epsilon]\times [0,z+2\epsilon],\\
E_-(\epsilon) &:= (a,b]\times (v+\epsilon,w-\epsilon]\times [0,z-2\epsilon],
\end{align*}
we have
\begin{multline*}
\esp_\calP\wt\xi_n(E_-(\epsilon))-\sum_{k=\wt a_n}^{\wt b_n}\pp{\proba_\calP(\Omega_n^c)+\proba_\calP(\wt\Omega_{n,\ell_n+k}^c)} \\
\le \esp_\calP\xi_n(E)
\le\esp_\calP\wt\xi_n(E_+(\epsilon))+\sum_{k=\wt a_n}^{\wt b_n}\pp{\proba_\calP(\Omega_n^c)+\proba_\calP(\wt \Omega_{n,\ell_n+k}^c)}.
\end{multline*}
Recall the exponential inequality $\proba(\sup_{t\in[0,n]}|\calN(t)-t|>u)\le 2e^{-u^2/(2n+2u/3)}$. (We could not locate a precise reference, but this should be well known. See for example \citep[Chapter II, Proposition 1.8]{revuz99continuous} for a proof of the exponential maximal inequality for Brownian motion. The same strategy applies here: for the upper bound of $\proba(\sup_{t\in[0,n]}\calN(t)-t>u)$ investigate the equivalent probability for the  exponential martingales $(e^{\lambda \calN(t)- t(e^\lambda-1)})_{t\ge 0}$, apply Doob's maximal inequality, and optimize $\lambda$ to obtain the claimed upper bound;  repeat the same on $t-\calN(t)$ for a lower bound.) Then, 
\[
\proba_\calP(\Omega_n^c)\le 2e^{-\frac 14\log n+ o(1)},
\]
and $\limn\sum_{k=\wt a_n}^{\wt b_n} \proba_\calP(\Omega_n^c)= 0$ (this is because $\wt b_n-\wt a_n = O(j_n^{1/2}) = O(n^{\alpha/(2(1+\alpha))}) = o(n^{1/4})$ as $\alpha\in(0,1)$). It is also clear that $\limn\sum_{k=\wt a_n}^{\wt b_n}\proba_\calP(\wt\Omega_{n,\ell_n+k}^c)= 0$. 
Hence,
\[
\limn\esp_\calP\wt\xi_n(E_-(\epsilon))\le \liminfn \esp_\calP \xi_n(E)\le\limsupn\esp_\calP\xi_n(E)\le\limsupn\esp_\calP\wt\xi_n(E_+(\epsilon)).
\]
Letting $\epsilon\downarrow 0$ we have proved that 
\[
 \limn \esp_\calP \xi_n(E) = \limn\esp_\calP\wt\xi_n(E) = \mu_\alpha([a,b])(w-v)(1-e^{-z}),
\]
almost surely.

Next, we compare $\proba_\calP(\xi_n(E) = 0)$ and $\proba_\calP(\wt\xi_n(E) = 0)$. This time for $n$ large enough,
\begin{multline*}
\bigcap_{k=\wt a_n}^{\wt b_n}\ccbb{\tau_{j_n}\topp{\ell_n+k}\notin(s_n(v-\epsilon),s_n(w+\epsilon)]}\cap\wt\Omega_{n,\ell_n+k}\cap\Omega_n
\\
\subset\bigcap_{k=\wt a_n}^{\wt b_n}\ccbb{\calN(\tau_{j_n}\topp{\ell_n+k})\notin(s_n(v),s_n(w)]} \cap\wt\Omega_{n,\ell_n+k}\cap\Omega_n\\
\subset
\bigcap_{k=\wt a_n}^{\wt b_n}\ccbb{\tau_{j_n}\topp{\ell_n+k}\notin(s_n(v+\epsilon),s_n(w-\epsilon)]},
\end{multline*}
and a similar relation holds with the third coordinates of the point processes involved. 
Taking the limit $n\to\infty$ and then $\epsilon\downarrow 0$ it follows that
\[
\limn\proba_\calP\pp{\xi_n(E) = 0} = \limn\proba_\calP\pp{\wt\xi_n(E) = 0} = \exp\pp{-\mu_\alpha((a,b])\times(w-v)\times(1-e^{-z})}.
\]
We have explained the sandwich argument for $E= (a,b]\times(v,w]\times[0,z]$. For more general $E$ as a disjoint union of finite rectangles, the argument can be accordingly adapted. 
We have completed the proof of Theorem \ref{thm:PP conv}.\end{proof}
\subsection{Proof of Theorem \ref{thm:discrete-cmt}}
\label{sec:CMT}
%\begin{proof}%[Proof of Theorem \ref{thm:discrete-cmt}]

The proof follows essentially by a continuous mapping argument. To do so, however, one has to restrict to a compact domain and proceed with an approximation. A key step is needed for this approximation and is proved in Lemma \ref{lem:A} at the end.
\begin{proof}[Proof of Theorem \ref{thm:discrete-cmt}]
For the convergence of the $\calM_{p,f}(\R)$-valued process \eqref{eq:Z_n},  we shall show that for every $T>0$, 
\[
\calZ\topp T_n:=(\mathcal Z_n(t))_{t\in[-T,T]}\aswto(\mathcal Z(t))_{t\in[-T,T]}=:\calZ\topp T
\]
with respect to $\calP$
in $D([-T,T],\mathcal M_{p,f}(\mathbb R))$ with the $J_1$ topology. 
Fix $T>0$. For $K>0$ and $\xi=\sum_i\delta_{(x_i,y_i,z_i)}$, define
\[
\bb{\Theta_K^{(T)}(\xi)}(t):=\sum_i\mathbf 1_{\{|x_i|\leq K\}}\mathbf 1_{\{z_i\leq K\}}\mathbf 1_{\{y_i\leq t<y_i+z_i\}}\delta_{x_i},\qquad t\in[-T,T].
\]
The map depends only on the restriction of $\xi$ to the compact set
\[
\Omega_K\topp T:=[-K,K]\times[-T-K,T]\times[0,K].
\]
We claim that $\Theta_K^{(T)}$ is continuous at $\xi$ almost surely. Indeed, almost surely, $\xi$ has no point on $\partial \Omega_K\topp T$ and hence has only finitely many points in $\Omega_K\topp T$. Mark these points by $(x_i,y_i,z_i)$, $i=1,\dots,m$. Then, $\{[\Theta_K\topp T(\xi)](t)\}_{t\in[-T,T]}$ is piecewise constant, and its jumps occur exactly at  $\{y_1,\dots,y_m,y_1+z_1,\dots,y_m+z_m\}\cap[-T,T]$. Moreover, almost surely these points are distinct and none of them is equal to $-T$ or $T$. Under these assumptions, it is clear that the mapping is continuous at $\xi$. We omit the details. 

Hence the continuous mapping theorem gives, for every fixed $K>0$, 
\begin{equation}\label{eq:truncated-convergence}
\calZ_n\topp{T,K}:=\Theta_K^{(T)}(\xi_n)\aswto\Theta_K^{(T)}(\xi)=:\calZ\topp{T,K},
\end{equation}
 with respect to $\calP$ in $D([-T,T],\mathcal M_{p,f}(\mathbb R))$ as $n\to\infty$.

Define
\[
X_{n,\ell}:=\frac{\ell-\widetilde\ell_n}{\sqrt{j_n}},\qquad Y_{n,\ell}:=j_n\left(\frac{T_{j_n}^{(\ell)}}n-1\right),\qquad Z_{n,\ell}:=\frac{j_n}{n}\left(T_{j_n+1}^{(\ell)}-T_{j_n}^{(\ell)}\right),
\]
and rewrite
\[
\mathcal Z_n(t)=\sum_{\ell\geq1}\delta_{X_{n,\ell}}\mathbf 1_{\{K_{\floor{s_n(t)},\ell}=j_n\}}=\sum_{\ell\geq1}\delta_{X_{n,\ell}}\mathbf 1_{\{Y_{n,\ell}\leq t<Y_{n,\ell}+Z_{n,\ell}\}}.
\]
We notice that
$\mathcal Z_n\topp {T,K}
=\mathcal Z_n\topp T
$ on $(A_{n,K,T}\cup B_{n,K,T})^c$,
with 
\begin{align}
A_{n,K,T}&:=\ccbb{\exists\ell\in\N: |X_{n,\ell}|>K,[Y_{n,\ell},Y_{n,\ell}+Z_{n,\ell})\cap[-T,T]\ne\emptyset},\label{eq:A def}\\
B_{n,K,T}&:=\ccbb{\exists\ell\in\N:|X_{n,\ell}|\leq K, Z_{n,\ell}>K,[Y_{n,\ell},Y_{n,\ell}+Z_{n,\ell})\cap[-T,T]\ne\emptyset}.\label{eq:B def}
\end{align}
A key step in the proof is to show that 
\equh
\lim_{K\to\infty}\limsup_{n\to\infty}\mathbb P_{\calP}
\left(\calZ_n\topp{T,K}\ne\mathcal Z_n\topp T\right)
\leq\lim_{K\to\infty}\limsup_{n\to\infty}
\mathbb P_{\calP}\left(A_{n,K,T}\cup B_{n,K,T}\right)=0,
\label{eq:discrete-approximation}
\eque
where the last convergence follows from Lemma \ref{lem:A} below.
We also have 
\begin{equation}\label{eq:limit-approximation}
\lim_{K\to\infty}\mathbb P_\calP\left(\calZ\topp{T,K}\ne \mathcal Z\topp T\right)=0,
\end{equation}
which implies in particular that $\calZ\topp{T,K}\aswto \calZ\topp T$ as $K\to\infty$ with respect to $\calP$. 
Indeed, the expected number of discarded intervals intersecting $[-T,T]$ is
\begin{multline*}
(2T+1)\mu_\alpha(\{x:|x|>K\})+\mu_\alpha([-K,K])\int_K^\infty(2T+z)e^{-z}\d z\\
=(2T+1)\mu_\alpha(\{x:|x|>K\})+(K+1+2T)e^{-K}\mu_\alpha([-K,K]),
\end{multline*}
which tends to zero. 

Together, \eqref{eq:truncated-convergence},  \eqref{eq:discrete-approximation}, and \eqref{eq:limit-approximation} yield $\calZ_n\topp T\aswto\calZ\topp T$ with respect to $\calP$ as $n\to\infty$, which is Theorem \ref{thm:discrete-cmt} on $[-T,T]$. Taking $T$ through the positive integers gives the asserted convergence in the local $J_1$ topology on $D(\mathbb R,\mathcal M_{p,f}(\mathbb R))$.

It remains to prove the convergence to the $M/M/\infty$ queue in \eqref{eq:M/M/infty}.
To see this, consider the continuous map
$F:\mathcal M_{p,f}(\mathbb R)\to\mathbb R_+$ given by $F(\xi) = \xi(\R)$. 
The induced map
\[
D(\mathbb R,\mathcal M_{p,f}(\mathbb R))
\ni \zeta\mapsto \bigl(F(\zeta(t))\bigr)_{t\in\mathbb R}
\in D(\mathbb R)
\]
is continuous under the local $J_1$ topologies.
Since
\[
F(\mathcal Z_n(t))=C_{j_n}(\floor{s_n(t)}),
\]
the first part of the theorem and the continuous mapping theorem then yield that
\[
\left(C_{j_n}(\floor{s_n(t)})\right)_{t\in\mathbb R}\aswto\left(F(\mathcal Z(t))\right)_{t\in\mathbb R}
\]
with respect to $\calP$ in $D(\mathbb R)$ with the local $J_1$ topology. After integrating out the first coordinate of $\xi$, given $\calP$ the birth times form a Poisson process of rate $\mu_\alpha(\mathbb R)=\sigma_\alpha$, and the lifetimes are independent exponential random variables with mean one. Thus $(F(\mathcal Z(t)))_{t\in\mathbb R}$ is the stationary $M/M/\infty$ queue with arrival rate $\sigma_\alpha\equiv \sigma_{\alpha,r}$ and service rate one, which is exactly $(\calC_{\alpha,r}(t))_{t\in\R}$. 
\end{proof}

The key estimates needed in the proof above are established in the next lemma. %Throughout, $\ell$ denotes an urn label, whereas $k$ denotes the integer shift in $\ell=\ell_n+k$. 
Recall the definition of $A_{n,K,T}$ and $B_{n,K,T}$ in \eqref{eq:A def} and \eqref{eq:B def}.
\begin{lemma}\label{lem:A}
 For every $T>0$, 
\begin{align}\label{eq:A}
\lim_{K\to\infty}&\limsup_{n\to\infty}\mathbb P_{\calP}(A_{n,K,T})=0, \\
\label{eq:B}
\lim_{K\to\infty}&\limsup_{n\to\infty}\mathbb P_{\calP}(B_{n,K,T})=0,
\end{align}
almost surely.
\end{lemma}

\begin{proof}%[Proof of Lemma \ref{lem:A}]
We first prove \eqref{eq:A}. For this purpose, introduce
\[
q^*_{n,\ell}(T):=\mathbb P_{\calP}\left(
[Y_{n,\ell},Y_{n,\ell}+Z_{n,\ell})\cap[-T,T]\ne\emptyset
\right).
\]
It suffices to prove that
\[
\lim_{K\to\infty}\limsup_{n\to\infty}
\sum_{\substack{\ell\in\N\\|\ell-\wt\ell_n|>K\sqrt{j_n}}}q^*_{n,\ell}(T)=0, \mbox{ almost surely.}
\]

This is an almost sure event. It suffices to prove the assertions when $\theta=0$; the general case follows from part (ii) of Lemma \ref{lem:P_j}. 

Let $(\calN_\ell)_{\ell\geq1}$ be the conditionally independent Poisson processes (given $\calP$) with respective rates $(p_\ell)_{\ell\geq1}$ of the Poissonized model. Let $\tau_j^{(\ell)}$ be the $j$-th arrival time of $\calN_\ell$. For $T>0$, set
\[
\widetilde q^*_{n,\ell}(T):=\mathbb P_{\calP}\left(
\tau_{j_n}^{(\ell)}\leq s_n(T),
\ \tau_{j_n+1}^{(\ell)}>s_n(-T)\right)
\]
as an approximation to $q_{n,\ell}^*(T)$. 
We first show that
\begin{equation}\label{eq:poissonized-localization}
\lim_{K\to\infty}\limsup_{n\to\infty}
\sum_{\substack{\ell\in\N\\|\ell-\wt\ell_n|>K\sqrt{j_n}}}\widetilde q^*_{n,\ell}(T)=0.
\end{equation}
Write 
\[
\ccbb{\ell\in\N:|\ell-\wt\ell_n|>K\sqrt{j_n}} = \bigcup_{r=1}^3 I_{n,r}
\]
for $n$ large enough (so that $K\sqrt{j_n}\le \delta \ell_n$) 
with, for some $\delta,M>0$,
\begin{align*}
I_{n,1}\equiv I_{n,1}(\delta)&:=\ccbb{\ell\in\N:K\sqrt{j_n}< |\ell-\wt\ell_n|\le \delta\ell_n},\\
I_{n,2}\equiv I_{n,2}(\delta,M)&:=\ccbb{\ell\in\N:\delta\ell_n< |\ell-\wt\ell_n|\le M\ell_n},\\
I_{n,3}\equiv I_{n,3}(M)&:=\ccbb{\ell\in\N:|\ell-\wt\ell_n|> M\ell_n}.
\end{align*}
We shall bound $\wt q_{n,\ell}^*(T)$ over the different intervals respectively.

We start with $\ell\in I_{n,1}$. 
For every fixed $K,T>0, \epsilon\in(0,1)$, by taking $\delta>0$ small enough, we have that there exists a $\calP$-measurable random variable $n_0$ such that for all $n\ge n_0$:
\equh\label{eq:separation}
(1-\epsilon)j_n\leq s_n(-T)p_\ell\leq s_n(T)p_\ell\leq (1+\epsilon)j_n,
\eque
and
\begin{align}
s_n(t)p_\ell<j_n-(1-\epsilon)\frac{\sabs{\ell-\wt\ell_n}}{\sigma_\alpha}, & \mbox{ if } \ell-\wt\ell_n>K\sqrt{j_n},\label{eq:separation1}
\\
s_n(t)p_\ell>j_n+(1-\epsilon) \frac{|\ell-\wt\ell_n|}{\sigma_\alpha},  & \mbox{ if } \ell-\wt\ell_n<-K\sqrt{j_n},\label{eq:separation2}
\end{align}
for all $\ell\in I_{n,1}$, $|t|\le T$.

We prove \eqref{eq:separation1} and \eqref{eq:separation2}. The claim \eqref{eq:separation} is relatively easy and actually can also be read from the proof of the other claims below. 
Set 
\[%\equh\label{eq:k_n}
k_n\equiv k_n(K_0):=K_0\sqrt{\ell_n\log\log n},
\]%\eque
for a constant $K_0$ to be specified later.
First, we consider
\[%\equh\label{eq:I_n,11}
I_{n,1,1}:=\ccbb{\ell\in\N: K\sqrt{j_n}\le \abs{\ell-\wt\ell_n}\le k_n}.
\]%\eque
Notice that $|\ell-\wt\ell_n|\le k_n$ implies that $|\ell-\ell_n|\le Ck_n$ for some constant $C>0$. Therefore, \eqref{eq:diff CRP} and the fact that $j_n/(\alpha\ell_n)\to\sigma_\alpha\inv$ yield
\[
\limn\max_{\ell\in I_{n,1,1}}\abs{\frac{np_\ell-j_n}{\ell-\wt\ell_n} +\sigma_\alpha\inv} = 0. 
\]
Note also that $(s_n(t)-n)p_\ell = (tn/j_n)p_\ell$ is of order $O(1)$ uniformly for all $\ell\in I_{n,1,1}$. We have thus proved \eqref{eq:separation}, \eqref{eq:separation1}, and \eqref{eq:separation2} for all $\ell\in I_{n,1,1}$. 

Next, we consider
\[%\equh\label{eq:I_n12}
I_{n,1,2}\equiv I_{n,1,2}(K_0,\delta):=\ccbb{\ell\in\N: k_n<\abs{\ell-\wt\ell_n}\le \delta\ell_n},
\]%\eque
and the parameters $K_0,\delta$ shall matter.
This time, we write
\begin{align*}
np_\ell & = d_0n\ell_n^{-1/\alpha}\pp{1+\frac{\ell-\wt\ell_n+\wb\Gamma_\ell-\wb\Gamma_{\ell_n}}{\ell_n}}^{-1/\alpha}\\
& =d_0n\ell_n^{-1/\alpha}\pp{1-\frac1\alpha\frac{\ell-\wt\ell_n}{\ell_n}-\frac1\alpha\frac{\wb\Gamma_\ell-\wb\Gamma_{\ell_n}}{\ell_n}+\omega_{n,\ell}}.
\end{align*}
Recall $d_0n\ell_n^{-1/\alpha}\sim j_n\sim \alpha\ell_n/\sigma_\alpha$. A direct calculation shows (see also \eqref{eq:4th})
\[
d_0n\ell_n^{-1/\alpha} = j_n + O(1).
\]
For the expression in the parenthesis, we first have for some constant $D>2$,
\equh\label{eq:gamma incre}
\max_{\ell\in I_{n,1,2}}|\wb\Gamma_\ell-\wb\Gamma_{\ell_n}|\le D\sqrt{(1+2\delta)\ell_n\log\log n},
\eque
for all $n$ large enough. Indeed, to prove \eqref{eq:gamma incre}, we first notice that 
since $\ell-\ell_n = \ell-\wt\ell_n - \wb\Gamma_{\ell_n}$, 
for all $n$ large enough, $\{\ell\in\N:|\ell-\wt\ell_n|\le \delta\ell_n\}\subset\{\ell\in\N:|\ell-\ell_n|\le 2\delta\ell_n\}$. That is
\[
\max_{\ell\in I_{n,1,2}}\abs{\wb\Gamma_\ell-\wb\Gamma_{\ell_n}}\le \max_{\ell\in\N:|\ell-\ell_n|\le 2\delta\ell_n}\abs{\wb\Gamma_\ell-\wb\Gamma_{\ell_n}},
\]
and the right-hand side above is eventually bounded by the right-hand side of \eqref{eq:gamma incre}. 
The order of this upper bound is the same as that of $k_n$, whose multiplicative constant $K_0$ has not yet been chosen.

Next for $\omega_{n,\ell}$, we compare with $(\ell-\wt\ell_n)/\ell_n$. Indeed, by Taylor's expansion we have  
\begin{align*}
|\omega_{n,\ell}|& \le C_\delta\frac{(\ell-\wt\ell_n+\wb\Gamma_\ell-\wb\Gamma_{\ell_n})^2}{\ell_n^2}\\
&\le C_\delta\frac{(\ell-\wt\ell_n)^2}{\ell_n^2}\pp{1+\frac{D\sqrt{1+2\delta}}{K_0}}^2\le 2C_\delta\frac{(\ell-\wt\ell_n)^2}{\ell_n^2},  \mfa \ell\in I_{n,1,2},
\end{align*}
for $n$ large enough, where the constant $C_\delta$ depends on $\delta$ and can be chosen to be bounded as $\delta\downarrow 0$ (from Taylor's expansion), but does not depend on $K_0$: the second inequality follows from the earlier estimate that
\[
\max_{\ell\in I_{n,1,2}}\frac{\abs{\wb\Gamma_\ell-\wb\Gamma_{\ell_n}}}{|\ell-\wt\ell_n|}\le \frac{D\sqrt{1+2\delta}}{K_0},
\]
and in the third inequality we have taken $K_0$ large enough. 
That is, for fixed $\delta>0$ by taking $K_0$ large enough we have that for $n$ large enough,
\[
\max_{\ell\in I_{n,1,2}}\frac{\ell_n|\omega_{n,\ell}|}{|\ell-\wt\ell_n|}\le 2C_\delta\max_{\ell\in I_{n,1,2}}\frac{|\ell-\wt\ell_n|}{\ell_n}\le 2C_\delta\delta.
\]

Therefore, for all $\epsilon>0$, fixing $D$, taking $\delta>0$ small enough, and then taking $K_0$ large enough we can ensure that
\[
\max_{\ell\in I_{n,1,2}}\abs{\frac{np_\ell - j_n}{\ell-\wt\ell_n} + \sigma_\alpha\inv}\le \frac\epsilon2\sigma_\alpha\inv.
\]
Note also that for $|t|\le T$, $(s_n(t)-n)p_\ell = O(1)$, and hence for $n$ large enough we have
\[
\max_{\ell\in I_{n,1,2}}\abs{\frac{s_n(t)p_\ell - j_n}{\ell-\wt\ell_n} + \sigma_\alpha\inv}< \epsilon\sigma_\alpha\inv, \mfa n \mbox{ large enough.}
\]
The above then implies
\eqref{eq:separation}, \eqref{eq:separation1} and \eqref{eq:separation2} over $I_{n,1,2}$. 

%\newpage

We shall also use the following fact.
For every fixed $0<c_0<C_0<\infty$, there exist constants $c_1,c_2>0$ such that 
\begin{equation}\label{eq:gamma-bound}
\frac{x^je^{-x}}{j!}
+\frac{x^{j-1}e^{-x}}{(j-1)!}
\leq\frac{c_1}{\sqrt j}
\exp\pp{-c_2\frac{(x-j)^2}{j}}, \mfa j\in\N, x\in [c_0j,C_0j].
\end{equation}
Indeed, Stirling's formula gives, for some $C>0$, 
$j!\geq C\sqrt j(j/e)^j$ for all $j\in\N$.
Therefore, writing $u=x/j\in[c_0,C_0]$, we obtain that for some $C>0$,
\[
\frac{x^je^{-x}}{j!}
\leq\frac{C}{\sqrt j}
\exp\pp{-j\pp{u-1-\log u}}, \mfa j\in\N.
\]
We then check that there exists a constant $c_2$ (depending on $c_0,C_0$) such that 
\[
u-1-\log u\geq c_2(u-1)^2,
\qquad u\in[c_0,C_0],
\]
and consequently
\[
\frac{x^je^{-x}}{j!}
\leq\frac{C}{\sqrt j}
\exp\pp{-c_2\frac{(x-j)^2}{j}}.
\]
Finally,
\[
\frac{x^{j-1}e^{-x}}{(j-1)!}
=\frac{j}{x}\frac{x^je^{-x}}{j!}
\leq\frac1{c_0}\frac{x^je^{-x}}{j!},
\]
which proves \eqref{eq:gamma-bound}.

Now, we are ready to estimate $\sum_{\ell\in I_{n,1}}\wt q_{n,\ell}^*(T)$. 
We first claim that there exist constants $C,c>0$ such that
\equh\label{eq:local-tail-bound}
\wt q^*_{n,\ell}(T)\leq \frac{C}{\sqrt{j_n}}\exp\pp{-c\frac{(\ell-\wt\ell_n)^2}{j_n}},\mfa \ell\in I_{n,1}.
\eque
Indeed, by definition
\begin{align}
\wt q^*_{n,\ell}(T)&=\proba_\calP\pp{\calN_\ell(s_n(-T))=j_n}+\proba_\calP\pp{s_n(-T)<\tau_{j_n}\topp\ell\leq s_n(T)}\nonumber
\\
&=\frac{\pp{s_n(-T)p_\ell}^{j_n}e^{-s_n(-T)p_\ell}}{j_n!}+\int_{s_n(-T)}^{s_n(T)}p_\ell\frac{(up_\ell)^{j_n-1}e^{-up_\ell}}{(j_n-1)!}\d u.\label{eq:q* I_n,1}
\end{align}
For the integral above, \eqref{eq:separation}, \eqref{eq:separation1}, and \eqref{eq:separation2} imply that for all $u\in[s_n(-T),s_n(T)]$ we have
\[
(1-\epsilon)j_n\leq up_\ell\leq(1+\epsilon)j_n,
\]
and
\[
\abs{up_\ell-j_n}\geq\frac{1-\epsilon}{\sigma_\alpha}\abs{\ell-\wt\ell_n},\mfa \ell\in I_{n,1}.
\]
Moreover,
\[
\max_{\ell\in I_{n,1}}p_\ell\pp{s_n(T)-s_n(-T)}=\max_{\ell\in I_{n,1}}2T\frac{np_\ell}{j_n}
=O(1).
\]
Thus, \eqref{eq:gamma-bound} yields that there exist $C>0, c>0$ such that both terms on the right-hand side of \eqref{eq:q* I_n,1} are bounded by constant multiples of the right-hand side of \eqref{eq:local-tail-bound}, and we have proved \eqref{eq:local-tail-bound}. By Riemann-sum approximation, we have
\equh\label{eq:I_n,1}
\limsup_{n\to\infty}\sum_{\ell\in I_{n,1}}\wt q^*_{n,\ell}(T)\leq C\int_K^\infty e^{-cx^2}\d x.
\eque

Next, consider $\ell\in I_{n,2}$.
Recall $\wt\ell_n/\ell_n\to 1$ almost surely.  Then, $|\ell-\wt\ell_n|\ge \delta \ell_n$ implies that
\equh\label{eq:two cases'}
\mbox{ either }\quad \ell\ge \ell_n^+:=\ceil{(1+\delta/2)\ell_n}\quad \mbox{ or } \quad\ell\le \ell_n^-:=\floor{(1-\delta/2)\ell_n}.
\eque 
Since $(p_\ell)_{\ell\geq1}$ is decreasing, if $\ell\ge \ell_n^+$ then 
$p_\ell\leq p_{\ell_n^+}$, and similarly in the other case we have
$p_\ell\geq p_{\ell_n^-}$. By the strong law,
$s_n(\pm T)p_{\ell_n^\pm}/j_n\to(1\pm\delta/2)^{-1/\alpha}$. 
Consequently, there exists $\epsilon>0$ (depending on $\delta$) such that, for all sufficiently large $n$, the two cases in \eqref{eq:two cases'} yield, respectively,
\begin{align*}
s_n(T)p_\ell\leq(1-\epsilon)j_n, & \mbox{ if } \ell\ge \ell_n^+,\\
s_n(-T)p_\ell\geq(1+\epsilon)j_n, & \mbox{ if } \ell\le \ell_n^-.
\end{align*}
Then, for $n$ large enough (depending on a $\calP$-measurable random variable),
\begin{align*}
\widetilde q^*_{n,\ell}(T)&=\mathbb P_{\calP}\left(
\tau_{j_n}^{(\ell)}\leq s_n(T),
 \tau_{j_n+1}^{(\ell)}>s_n(-T)\right)\nonumber
\\
&= \proba_\calP\pp{\calN_\ell(s_n(T))\ge j_n, \calN_\ell(s_n(-T))\le j_n}\nonumber\\
& \le \begin{cases}
\displaystyle\proba_\calP\pp{\calN_\ell(s_n(T))\ge \frac{s_n(T)p_\ell}{1-\epsilon}}, & \mbox{ if } \ell\ge \ell_n^+,\\\\
\displaystyle \proba_\calP\pp{\calN_\ell(s_n(-T))\le \frac{s_n(-T)p_\ell}{1+\epsilon}} & \mbox{ if } \ell\le \ell_n^-.
 \end{cases}%\label{eq:q I_n,2}
\end{align*}
By the Chernoff bounds of Poisson random variables we have  that there exists $c_\epsilon$ (depending on $\epsilon$ and hence $\delta$, but not $M$) such 
that
$\max_{\ell\in I_{n,2}}\widetilde q^*_{n,\ell}(T)\leq Ce^{-c_\epsilon j_n}$.
Hence, 
\equh\label{eq:I_n,2}
\sum_{\ell\in I_{n,2}}
\widetilde q^*_{n,\ell}(T)
\leq  CM\ell_n e^{-c_\epsilon j_n}\to 0
\eque
as $n\to\infty$.

For $\ell\in I_{n,3}$, fix $M>2$. Since
$\wt\ell_n/\ell_n\to1$ almost surely, for all sufficiently large $n$,
\[
I_{n,3}\subset\ccbb{\ell\in\N:\ell>M\ell_n}.
\]
 Using $p_\ell= d_0 \Gamma_\ell^{-1/\alpha}$ we have that 
\[
\frac{s_n(T)p_\ell}{j_n}\leq
C\left(\frac{\ell_n}{\ell}\right)^{1/\alpha}, 
\]
for all $\ell\in\N$, and the constant $C$ does not depend on $M$. By taking $M$ large enough so that $CM^{-1/\alpha}<1$ (note that here $C$ is a $\calP$-measurable random variable), it then follows that 
\equh\label{eq:I_n,3}
\sum_{\ell\geq M\ell_n}\widetilde q^*_{n,\ell}(T)
\le \sum_{\ell\geq M\ell_n}\pp{\frac{e s_n(T)p_\ell}{j_n}}^{j_n}
%\le \sum_{\ell\geq M\ell_n}\pp{C\pp{\frac{\ell_n}\ell}^{1/\alpha}}^{j_n}
\leq C\ell_nM\left(CM^{-1/\alpha}\right)^{j_n}\to 0.
\eque
Combining \eqref{eq:I_n,1}, \eqref{eq:I_n,2} and \eqref{eq:I_n,3}, we have proved \eqref{eq:poissonized-localization}.

It remains to transfer the estimate to the original model. Under the coupling,
$\calN(t):=\sum_{\ell\geq1}\calN_\ell(t)$
is a standard Poisson process and $T_j^{(\ell)}=\calN(\tau_j^{(\ell)})$.
Let
\[
\Omega_n:=\left\{
\sup_{0\leq u\leq2n}|\calN(u)-u|\leq\frac{n}{j_n}
\right\}.
\]
On $\Omega_n$, if $[T_{j_n}\topp\ell, T_{j_n+1}\topp\ell]\cap [s_n(-T),s_n(T)]\ne\emptyset$, then $[\tau_{j_n}\topp\ell,\tau_{j_n+1}\topp\ell]\cap[s_n(-T-2),s_n(T+2)]\ne\emptyset$. This essentially follows from $\calN(\tau_j\topp\ell) = T_j\topp\ell$. Set $E_{n,\ell}:=\{[T_{j_n}\topp\ell,T_{j_n+1}\topp \ell]\cap[s_n(-T),s_n(T)]\ne\emptyset\}$. Then,
\begin{align*}
\sum_{\substack{\ell\in\N\\|\ell-\wt\ell_n|>K\sqrt{j_n}}}q^*_{n,\ell}(T)
& \leq
\sum_{\substack{\ell\in\N\\|\ell-\wt\ell_n|>K\sqrt{j_n}}}\widetilde q^*_{n,\ell}(T+2)
+\sum_{\substack{\ell\in\N\\|\ell-\wt\ell_n|>K\sqrt{j_n}}} \proba_\calP\pp{E_{n,\ell}\cap \Omega_n^c}\\
& =
\sum_{\substack{\ell\in\N\\|\ell-\wt\ell_n|>K\sqrt{j_n}}}\widetilde q^*_{n,\ell}(T+2)
+\esp_\calP\pp{\sum_{\substack{\ell\in\N\\|\ell-\wt\ell_n|>K\sqrt{j_n}}}\ind_{E_{n,\ell}}\ind_{ \Omega_n^c}}\\
& \le \sum_{\substack{\ell\in\N\\|\ell-\wt\ell_n|>K\sqrt{j_n}}}\widetilde q^*_{n,\ell}(T+2)+\frac{s_n(T)}{j_n}\proba_\calP(\Omega_n^c).
\end{align*}
A Poisson maximal inequality gives
\[
\mathbb P_{\calP}(\Omega_n^c)
\leq C\exp\pp{-C\frac{n}{j_n^2}}.
\]
Since $n/j_n^2$ is of order $n^{(1-\alpha)/(1+\alpha)}$, the second term tends to zero, and the first assertion follows from \eqref{eq:poissonized-localization}.

Next, we prove  \eqref{eq:B}. 
Note that $T$ is fixed. Consider $K>2T$ from now on. Write $I_{n,\ell}:=[Y_{n,\ell},Y_{n,\ell}+Z_{n,\ell})$. Then, $I_{n,\ell}\cap [-T,T]\ne\emptyset$ and $Z_{n,\ell}>K$ imply that either $-T\in I_{n,\ell}$ or $T\in I_{n,\ell}$, or equivalently  $K_{\floor{s_n(t)},\ell} = j_n$ with $t=-T$ or $T$. Thus,
\begin{multline*}
\proba_\calP\pp{[Y_{n,\ell},Y_{n,\ell}+Z_{n,\ell})\cap [-T,T]\ne\emptyset, Z_{n,\ell}>K}\\
\le\sum_{t\in\{-T,T\}}\proba_\calP\pp{K_{\floor{s_n(t)},\ell}=j_n, T\topp\ell_{j_n+1}-T\topp\ell_{j_n}>\frac{Kn}{j_n}}.
\end{multline*}
Then, 
with $r_{n,K}:=\floor{Kn/(2j_n)}$ we have 
\begin{align*}
&\proba_\calP\pp{K_{\floor{s_n(t)},\ell}=j_n, T\topp\ell_{j_n+1}-T\topp\ell_{j_n}>\frac{Kn}{j_n}}\\
&\le \mathbb P_{\calP}\left(K_{\floor{s_n(t)},\ell}=j_n,\ T_{j_n+1}^{(\ell)}-\floor{s_n(t)}>r_{n,K}\right)+\mathbb P_{\calP}\left(K_{\floor{s_n(t)},\ell}=j_n,\floor{s_n(t)}-T_{j_n}^{(\ell)}>r_{n,K}\right)\\
& = \proba_\calP\pp{K_{\floor{s_n(t)},\ell} = j_n}(1-p_\ell)^{r_{n,K}}+\proba_\calP\pp{K_{\floor{s_n(t)}-r_{n,K}-1,\ell}=j_n}(1-p_\ell)^{r_{n,K}+1}.
\end{align*}

Introduce 
\[
J_{n,K}:=\ccbb{\ell\in\N:|\ell-\wt\ell_n|\le K\sqrt{j_n}}.
\]
Then, 
\[%\equh\label{eq:Z_n,k}
\mathcal Z_n\topp{T,K}(t)=
\sum_{\ell \in J_{n,K}}
\delta_{X_{n,\ell}}
\mathbf 1_{\{Z_{n,\ell}\leq K\}}
\mathbf 1_{\{Y_{n,\ell}\leq t<Y_{n,\ell}+Z_{n,\ell}\}}.
\]%\eque
Lemma \ref{lem:CRP case} yields, for each fixed $K$,
\[
\limn\max_{\ell\in J_{n,K}}\left|\frac{np_{\ell}}{j_n}-1\right|=0.
\]
Consequently, for all sufficiently large $n$,
\[
\max_{\ell\in J_{n,K}}(1-p_\ell)^{r_{n,K}}\leq e^{-K/4}.
\]
Set
\[
t_{n,K}\equiv t_{n,K}(t):=j_n\left(\frac{\floor{s_n(t)}-{\floor{Kn/(2j_n)}}-1}n-1\right).
\]
We therefore obtain, for all sufficiently large $n$,
\begin{multline}
\mathbb P_{\calP}(B_{n,K,T})
\leq e^{-K/4}\sum_{\ell\in J_{n,K}}\sum_{t\in\{-T,T\}}
\proba_\calP(K_{\floor{s_n(t)},\ell}=j_n)\\
+e^{-K/4}\sum_{\ell\in J_{n,K}}\sum_{t\in\{-T,T\}}
\proba_\calP(K_{\floor{s_n(t_{n,K})},\ell}=j_n).\label{eq:B2}
\end{multline}
Note that $\limn t_{n,K} = t-K/2$.

We also have
for every fixed $K,T>0$, 
\begin{equation}\label{eq:central-mass}
\limn\sup_{|t|\leq T}\left|\sum_{\ell\in J_{n,K}}
\mathbb P_{\calP}\left(K_{\floor{s_n(t)},\ell}=j_n\right)-\mu_\alpha([-K,K])\right|=0,
\end{equation}
almost surely.
To see this we have the following estimate. It is convenient to write $\ell = \ell_n+k$ from now on. Note that $\ell-\wt\ell_n = k+\wb\Gamma_{\ell_n}$. We have
\equh\label{eq:binomial pmf}
\proba_\calP\pp{K_{\floor{s_n(t)},\ell_n+k}=j_n}=\frac{1+o(1)}{\sqrt{2\pi j_n}}
\exp\pp{-\frac{(k+\wb\Gamma_{\ell_n})^2}{2\sigma_\alpha^2j_n}},
\eque
uniformly for $t,k$ such that $|t|\le T$, $|k+\wb\Gamma_{\ell_n}|\le K\sqrt{j_n}$. 
Now \eqref{eq:central-mass} is the corresponding Riemann-sum convergence.
Indeed, uniformly for $|t|\leq T$ and
$|k+\wb\Gamma_{\ell_n}|\leq K\sqrt{j_n}$, we have
\[
\floor{s_n(t)}\sim n,
\qquad
p_{\ell_n+k}\sim\frac{j_n}{n}.
\]
Set
\[
y_n(t,k):=\floor{s_n(t)}p_{\ell_n+k}.
\]
Notice that $y_n(t,k) =  np_{\ell_n+k}+t + o(1)$,
where the $o(1)$ term is uniform over $t$ and $k$ in the range of interest. 
Since $j_n^2/n\to0$, the binomial--Poisson comparison
\eqref{eq:binomial-Poisson-local} 
yields
\equh\label{eq:BP}
\proba_{\calP}\pp{K_{\floor{s_n(t)},\ell_n+k}=j_n}
=
\frac{y_n(t,k)^{j_n}e^{-y_n(t,k)}}{j_n!}(1+o(1))
\eque
uniformly over the stated range. 
Consequently, the same calculation leading to \eqref{eq:Poisson pmf}
yields 
\begin{align*}
\frac{y_n(t,k)^{j_n}e^{-y_n(t,k)}}{j_n!}
&=
\frac{1+o(1)}{\sqrt{2\pi j_n}}
\exp\pp{
-\frac{
\pp{k+\wb\Gamma_{\ell_n}-\sigma_\alpha t}^2
}{2\sigma_\alpha^2j_n}}
\\
&=
\frac{1+o(1)}{\sqrt{2\pi j_n}}
\exp\pp{
-\frac{
\pp{k+\wb\Gamma_{\ell_n}}^2
}{2\sigma_\alpha^2j_n}+o(1)}
\end{align*}
uniformly over the same range. This and \eqref{eq:BP} together prove \eqref{eq:binomial pmf}. 
\color{black}

It follows from \eqref{eq:central-mass} (over $|t|\le T+K/2+1$) and \eqref{eq:B2} that 
\[
\limsupn\proba_\calP(B_{n,K,T}) \leq4\mu_\alpha(\mathbb R)e^{-K/4},
\]
whence \eqref{eq:B} holds.
\end{proof}

\subsection{Proof of the general case}\label{sec:general}
In this section, we provide a sketched proof of Theorem \ref{thm:PP conv general}, where we assume $p_n = d_0n^{-1/\alpha}(1+\rho_n)$ with $\limn n^{1/2}\rho_n = 0$. The proof follows the same strategy as in previous sections. We first sketch the proof of \eqref{eq:PP convergence}.  The key difference in the calculation is that this time we have the following for the Poissonized model (compared with \eqref{eq:tau_jn asymp})
\[
\proba\pp{\tau\topp{\ell_n+k}_{j_n}\in(s_n(v),s_n(w)]}
\sim \frac1{\sqrt{2\pi j_n}}\int_v^w\exp\pp{-\frac{(k-\sigma_\alpha y)^2}{2j_n\sigma_\alpha^2}}\d y,
\]
and therefore when working with the summation over $k$ we work with $k = a_n,\dots,b_n$ with 
\[
a_n:=\floor{aj_n^{1/2}} \qmand b_n:=\floor{bj_n^{1/2}},
\]
instead of $k=\wt a_n,\dots,\wt b_n$ (there is no longer the drift term; compare \eqref{eq:diff CRP} and \eqref{eq:diff general}).
Thus, letting $\wt\xi_n'$ denote the point process for the Poissonized model corresponding to $\xi_n'$ in~\eqref{eq:xi'_n}, we have, for all $a<b, v<w, z>0$, 
\begin{align*}
 \esp \wt\xi'_n([a,b]\times[v,w]\times[0,z]) & \sim \sum_{k=a_n}^{b_n}\proba\pp{\tau\topp{\ell_n+k}_{j_n}\in(s_n(v),s_n(w)]}\proba  \pp{\tau\topp{\ell_n+k}_1\le \frac {nz}{j_n}}\\
&\sim \frac1{\sqrt{2\pi j_n}}\int_{a_n}^{b_n+1}\int_v^w\exp\pp{-\frac{(x-\sigma_\alpha y)^2}{2j_n\sigma_\alpha^2}}\d y\d x(1-e^{-z})\nonumber\\
& \to \frac1{\sqrt{2\pi}}\int_a^b\int_v^w e^{-x^2/(2\sigma_\alpha^2)}\d y\d x(1-e^{-z}),
\end{align*}
where in the second step we used the following in place of \eqref{eq:square}:
\begin{align*}
\frac{\spp{np_{\ell_n+k}-j_n}^2}{j_n}&=\frac1{j_n}\pp{\frac{j_nk}{\alpha \ell_n}+\wt O(1)+\wt O(j_n\rho_{\ell_n+k})}^2 = \frac{j_nk^2}{\alpha^2\ell_n^2}+\frac{2k}{\alpha\ell_n}\pp{\wt O(1)+\wt O(j_n\rho_{\ell_n+k})}\nonumber\\
& = \frac{j_nk^2}{\alpha^2\ell_n^2}+\wt O\pp{\frac{k}{\ell_n}} + \wt O(j_n^{1/2}\rho_{\ell_n+k}).\nonumber%\label{eq:d_0 2}
\end{align*}
In the last expression, recall that $\wt O$ is uniform over $|k|\le Kj_n^{1/2}$, and hence for the last $\wt O$ term to be negligible we need the condition $\limn n^{1/2}\rho_n = 0$. 
Moreover, we also have
\begin{align*}
 \proba &\pp{ \wt\xi'_n([a,b]\times[v,w]\times[0,z])=0}\\
&  \sim \prod_{k=a_n}^{b_n}\proba\pp{\tau\topp{\ell_n+k}_{j_n}\notin(s_n(v),s_n(w)], \mbox{ or } \tau\topp{\ell_n+k}_{j_n+1}-\tau\topp{\ell_n+k}_{j_n}> \frac {nz}{j_n}}\\
&\to \exp \pp{-\mu_\alpha([a,b])(w-v)(1-e^{-z})}.
\end{align*}
The above calculations complete the proof of point-process convergence for the Poissonized model. The de-Poissonization follows from the same sandwich argument as before. The proof for the second and third parts of Theorem \ref{thm:PP conv general} follows the proof of Theorem \ref{thm:discrete-cmt}.

\appendix
\section{A time-change lemma}
Here we establish a time-change lemma that generalizes a result from \citet[P.~151]{billingsley99convergence}. 
Assume $a<b$ and $\epsilon>0$. Let $(X_n)_{n\in\N}$ and $X$ be stochastic processes in $D[a-\epsilon,b+\epsilon]$ (we use the abbreviated notation $X_n \equiv (X_n(t))_{t\in[a-\epsilon,b+\epsilon]}$ and similarly for other processes). Let $(\Phi_n)_{n\in\N}$ and $\Phi$ be stochastic processes in $D[a,b]$. 
Set 
\[
D^\uparrow_\epsilon[a,b]:=\{\phi\in D[a,b]: \phi \mbox{ is nondecreasing and } a-\epsilon\le \phi(a)\le\phi(b)\le b+\epsilon\},
\]
with the topology induced by the Skorokhod topology on $D[a,b]$  (the Skorokhod topology is usually defined for the space $D[0,1]$ \citep{billingsley99convergence}; but $D[a,b]$ can be obtained by an affine transformation of $D[0,1]$). 
Assume
 \[
 (X_n,\Phi_n)\weakto (X,\Phi)
 \]
 in $D[a-\epsilon,b+\epsilon]\times D_\epsilon^\uparrow[a,b]$ as $n\to\infty$. Note that in \citet{billingsley99convergence} it is assumed that $a=0,b=1, \epsilon=0$, but we need $\epsilon>0$ for our application. The following lemma adapts \citet[Lemma on p.~151]{billingsley99convergence} with some minor modifications. We keep the proof here for convenience.

\begin{lemma}\label{lem:time change}
Under the notation and assumptions above, assume in addition that $X\in C[a-\epsilon,b+\epsilon]$ almost surely. Then,  we have
\[
X_n\circ \Phi_n\weakto X\circ \Phi \mbox{ in $D[a,b]$}.
\]
\end{lemma}

\begin{proof}
Set 
$\psi:D[a-\epsilon,b+\epsilon]\times D_\epsilon^\uparrow[a,b]\to D[a,b]$ by 
\equh\label{eq:psi}
\psi(x,\phi) = x\circ\phi.
\eque
We recall some notation. Let $\Lambda_{[a,b]}$ be the class of strictly increasing functions $\lambda:[a,b]\to[a,b]$ such that $\lambda(a) = a,\lambda(b) = b$, and $\nn f_{[a,b],\infty}:=\sup_{t\in[a,b]}|f(t)|$. Let `${\rm id}$' denote the function $g(t) = t$.  Then, $\phi_n\to \phi$ in $D_\epsilon^\uparrow[a,b]$ if there exists $(\lambda_n)_{n\in\N}\in \Lambda_{[a,b]}$ such that 
\equh\label{eq:lambda_n A}
\limn \nn{\lambda_n-{\rm id}}_{[a,b],\infty}=0 \qmand \limn\nn{\phi_n-\phi\circ\lambda_n}_{[a,b],\infty}= 0.
\eque 
We apply the continuous mapping theorem which consists of showing that $\psi$ in \eqref{eq:psi} is continuous at $(x,\phi)$ for all $x\in C[a-\epsilon,b+\epsilon]$ and $\phi\in D_\epsilon^\uparrow[a,b]$.
In fact, we show that for $(x_n)_{n\in\N}\subset D[a-\epsilon,b+\epsilon], (\phi_n)_{n\in\N}\subset D_\epsilon^\uparrow[a,b]$  such that $x_n\to x, \phi_n\to \phi$, we have
\[
\limn\nn{x_n\circ\phi_n - x\circ\phi\circ\lambda_n}_{[a,b],\infty}= 0,
\]
where $(\lambda_n)_{n\in\N}$ are as picked above in \eqref{eq:lambda_n A}.
To see that the above holds, it suffices to notice that  for every $t\in[a,b]$, we have
\begin{align*}
\abs{x_n\circ\phi_n(t) - x\circ\phi\circ\lambda_n(t)} &\le \abs{x_n\circ\phi_n(t) - x\circ\phi_n(t)}  + \abs{x\circ\phi_n(t) - x\circ\phi\circ\lambda_n(t)}\\
&\le\nn{x_n-x}_{[a-\epsilon,b+\epsilon],\infty} + \sup_{\substack{s,t\in [a-\epsilon,b+\epsilon]\\ |s-t|\le \nn{\phi_n-\phi\circ\lambda_n}_{[a,b],\infty}}}|x(s)-x(t)|,
\end{align*}
which goes to zero as $n\to\infty$ by assumption (for the second term we used that $x\in C[a-\epsilon,b+\epsilon]$).
\end{proof}
%\newpage
\section{Proof of Theorem \ref{thm:j fixed}}\label{sec:decomp}
Theorem \ref{thm:j fixed} is needed only for comparison and does not concern our main results on the phase transition as $j_n\to\infty$. We therefore only sketch the proof. 
The convergence of the process in the first coordinate (jointly in $j$) has been well known \citep{karlin67central,chebunin16functional,barbour09small,gnedin07notes}. The joint convergence can be established by the methodology in \citet{wang26central}. We only sketch the proof of the convergence of the process in the second coordinate in \eqref{eq:j decomp}. Notice that
\begin{align*}
(nt)^\alpha p_{j}\topp\alpha S_\alpha & =  (nt)^\alpha \frac\alpha{\Gamma(1-\alpha)}\frac{\Gamma(j-\alpha)}{\Gamma(j+1)}S_\alpha= \frac D{\Gamma(j+1)}(nt)^\alpha \pp{j\Gamma(j-\alpha)-\Gamma(j+1-\alpha)}\\
&  = \frac D{\Gamma(j+1)}\int_0^\infty e^{-z}(j-z)z^{j-1} \pp{\frac{nt}z}^{\alpha}\d z .
\end{align*}
Then, working with the Poissonized model, we have for all $t>0$
(recalling also the integral representation of $\esp(\wt C_j(nt)\mid\calP)$ in \eqref{eq:E wt C})
\begin{align*}
\frac{\esp\spp{\wt C_{j}(nt)\mid\calP} - (nt)^\alpha p_j\topp\alpha S_\alpha}{n^{\alpha/2}} & = \frac1{\Gamma(j+1)}\int_0^\infty e^{-z}(j-z)z^{j-1}\frac{N(D(nt/z)^\alpha)-D(nt/z)^\alpha}{n^{\alpha/2}}\d z\\
& \weakto \frac{D^{1/2}}{\Gamma(j+1)}\int_0^\infty e^{-z}(j-z)z^{j-1}\B_{(t/z)^\alpha}\d z\\
& \eqd D^{1/2}\int_0^\infty \proba(N(ty^{-1/\alpha}) = j)\d\B_y,
\end{align*}
as $n\to\infty$ in $D[0,1]$. 
Here, in the $\weakto$ step we used \eqref{eq:Vishakha} and then a standard truncation argument by first working with $\int_\delta^{1/\delta}$ and then letting $\delta\downarrow0$. This part is the most involved and requires some work similar to that in \citep{wang26central}. 
The first equality follows from a change of variables and then  integration by parts for It\^o integrals. One readily checks that the limit process above (indexed by $t$) has the same law as $D^{1/2}(\zeta_{\alpha,j}\topp2(t))_{t\ge 0}$. 
This completes the proof for the Poissonized model. The de-Poissonization step again relies on the binomial--Poisson approximation \eqref{eq:binomial-Poisson-local}. In this way we proved the convergence of the process in the second coordinate.

The second-coordinate processes are \(\calP\)-measurable. Hence the
almost sure weak convergence of the first coordinates, together with
the joint convergence of the second coordinates and \(D\), implies
the joint weak convergence of the two coordinates by conditioning on
\(\calP\). Their joint law is precisely the one represented by \(M_\alpha\)
in the statement.

Finally, the result for every \(\theta>-\alpha\) follows from the
change of measure in part (ii) of Lemma \ref{lem:P_j}. All the convergences
above hold jointly with \(D\), while the Radon--Nikodym derivative is
a function of \(S_\alpha\), and hence of \(D\), only. Truncating this
density transfers the convergence from \(\theta=0\) to general
\(\theta\). Since the argument applies to every finite collection of
indices \(j\), coordinatewise tightness completes the proof in
\((D[0,1]\times D[0,1])^\N\).

%	\newpage
	\bibliographystyle{apalike}
%	\bibliography{../../include/references,../../include/references18}
\bibliography{references,references18}	
\end{document}